\documentclass[12pt,twoside,reqno]{amsart}

\usepackage{amsmath, amssymb, amsthm, enumitem, esint}
\usepackage{xcolor}

\usepackage{hyperref}
\hypersetup{bookmarksdepth=3}

\newcommand{\lto}{\longrightarrow}
\newcommand{\tdr}{\td{r}}

\newcommand{\IR}{\mathbb{R}}
\newcommand{\lb}{\linebreak[1]}
\newcommand{\IC}{\mathbb{C}}
\newcommand{\CC}{\mathcal{C}}

\newcommand{\IN}{\mathbb{N}}

\newcommand{\IZ}{\mathbb{Z}}

\newcommand{\ZZ}{\mathcal{Z}}
\newcommand{\YY}{\mathcal{Y}}
\newcommand{\NN}{\mathcal{N}}
\newcommand{\HH}{\mathcal{H}}
\newcommand{\MM}{\mathcal{M}}

\renewcommand{\SS}{\mathcal{S}}
\newcommand{\XX}{\mathcal{X}}

\newcommand{\s}{\mathfrak{s}}
\newcommand{\eps}{\varepsilon}
\newcommand{\la}{\lambda}

\newcommand{\ov}[1]{\overline{#1}}
\newcommand{\td}[1]{\widetilde{#1}}

\newcommand{\textQQqq}[1]{\qquad \text{#1} \qquad}

\newcommand{\textQQq}[1]{\qquad \text{#1} \quad}
\newcommand{\textQq}[1]{\quad \text{#1} \quad}

\DeclareMathOperator{\genus}{genus}
\DeclareMathOperator{\cat}{cat}
\DeclareMathOperator{\loc}{loc}

\DeclareMathOperator{\Int}{Int}

\DeclareMathOperator{\diam}{diam}
\DeclareMathOperator{\inj}{inj}

\DeclareMathOperator{\supp}{supp}

\DeclareMathOperator{\DIV}{div}

\DeclareMathOperator{\reg}{reg}
\DeclareMathOperator{\sing}{sing}
\DeclareMathOperator{\good}{good}
\DeclareMathOperator{\disc}{disc}
\DeclareMathOperator{\fat}{fat}
\DeclareMathOperator{\gen}{gen}

\newcommand{\dotcup}{\ensuremath{\mathaccent\cdot\cup}}
\newcommand{\EMPTY}[1]{}

\newtheorem{Theorem}[equation]{Theorem}
\newtheorem{Lemma}[equation]{Lemma}
\newtheorem{Corollary}[equation]{Corollary}
\newtheorem{Proposition}[equation]{Proposition}
\newtheorem{Claim}[equation]{Claim}

\newtheorem*{MOC}{Multiplicity One Conjecture}
\newtheorem*{ExGeneral}{Existence for General Initial Data}
\newtheorem*{ExS2}{Well-posedness  for 2-Spheres}
\newtheorem*{ExGeneric}{Existence for Generic Initial Data}

\theoremstyle{definition}
\newtheorem{Definition}[equation]{Definition}
\theoremstyle{remark}

\numberwithin{equation}{section}

\title{On the Multiplicity One Conjecture for Mean Curvature Flows of surfaces}
\author{Richard H  Bamler and Bruce Kleiner}
\address{Department of Mathematics, UC Berkeley, CA 94720, USA}
\email{rbamler@berkeley.edu}
\address{Courant Institute of Mathematical Sciences, New York University, 251 Mercer Street, New York, NY 10012}
\email{bkleiner@cims.nyu.edu}
\thanks{R.B. was supported by NSF grants DMS-1906500 and DMS-2204364.
B.K. was supported by NSF grants DMS-2005553 and DMS-2305397.}
\date{\today}

\begin{document}
\begin{abstract}
We prove the Multiplicity One Conjecture for mean curvature flows of surfaces in~$\IR^3$.
Specifically, we show that any blow-up limit of such mean curvature flows has multiplicity one.
This has several applications.
First, combining our work with results of Brendle and Choi--Haslhofer--Hershkovits--White, we show that any level set flow starting from an embedded surface diffeomorphic to a 2-spheres does not fatten.
In fact, we obtain that the problem of evolving embedded 2-spheres via the mean curvature flow equation is well-posed within a natural class of singular solutions.
Second, we use our result to remove an additional condition in recent work of Chodosh--Choi--Mantoulidis--Schulze.
This shows  that  mean curvature flows starting from any \emph{generic} embedded surface only incur cylindrical or spherical singularities.
Third, our approach offers a new regularity theory for solutions of mean curvature flows that flow through singularities.
Among other things, this theory also applies to the innermost and outermost flow of \emph{any} embedded surface and shows that all singularity models of such flows must have multiplicity one.
It also establishes equality of the fattening time with the discrepancy time.
Lastly, we obtain a number of further results characterizing a separation phenomenon of mean curvature flows of surfaces.
\end{abstract}

\maketitle
\tableofcontents

\section{Introduction}
\subsection{Background  and overview}

A mean curvature flow describes the evolution of a family of embedded hypersurfaces $\MM_t \subset \IR^{n+1}$ that move in the direction of the mean curvature vector:
\[ \partial_t \mathbf{x} = \mathbf{H}. \]
This equation is the gradient flow of the area functional and the natural analog of the heat equation for an evolving hypersurface.
Initially, it tends to smooth out geometries over brief time-intervals. 
However, due to its inherent non-linearity, the mean curvature flow equation frequently leads to the formation of  singularities. 
Beyond these singularities, the flow loses its smooth form and ceases to exist.

Mean curvature flow first arose as a mathematical model in the physics literature in the 1950s and has been studied intensively by mathematicians since Brakke's work in 1978 \cite{Brakke_1978}.
For the last 45 years, one of the central goals was to provide a satisfactory mathematical treatment of this fundamental PDE, which is of interest in geometry and physics, as well as from potential applications within mathematics.
This includes the development of a theory of mean curvature flow that allows for a continuation of the flow after the formation of singularities.
More specifically, in the case of surfaces, the aim has been to establish the following:
\begin{quote} \it
Given a smoothly embedded, compact surface $\MM_0 \subset \IR^3$ there is a mean curvature flow that is smooth on the complement of a ``small'' singular set, near which a change of topology may occur. 
\end{quote}
Moreover, it has been expected that:
\begin{quote} \it
If the surface $\MM_0$ is chosen generically, then this flow is unique and the singularities are modeled by cylinders and spheres.
Moreover, the topological change can be understood in terms of these models.
\end{quote}
Over the years, various approaches have been used to address these problems.
Our work in this paper leads to a novel regularity theory for mean curvature flows of surfaces in $\IR^3$, which unifies these past approaches and closely aligns with numerous anticipated phenomenological aspects.
In particular, we will give a satisfactory answer to both questions above.

\medskip
A central concern arising in the analysis of singularity formation of mean curvature flows in $\IR^3$ has been the issue of multiplicity.
This issue has often been entwined with mass cancellation or non-uniqueness phenomena and is usually seen as one of the key obstacles toward a complete existence and partial regularity theory.
The goal of this paper is to address this issue and rule out the possibility of higher multiplicity.
In particular, we will resolve a conjecture of Ilmanen from~1995 \cite{Ilmanen_sing_MCF_surf} known as the \emph{Multiplicity One Conjecture.}

\medskip
\medskip
Let us now provide some further background.
The question of singularity formation and the role of multiplicity was already implicit in Brakke's seminal work \cite{Brakke_1978}.
Mean curvature flow gained further prominence within the geometric analysis community due to the work of Huisken \cite{Huisken_1984}, which described the singularity formation in the convex case.
A second major breakthrough in the history of mean curvature flow was Huisken's discovery of a monotonicity formula \cite{Huisken_monotonicity}.
This formula laid the foundations for studying  singularity formation in more general settings.
It suggested that tangent flows --  blow-ups of the flow  at a singular point -- must be self-similar and satisfy a ``shrinker'' equation.
Building on these insights, Ilmanen, in his influential 1995 preprint \cite{Ilmanen_sing_MCF_surf}, extended these techniques and showed that tangent flows of 2-dimensional mean curvature flows must be \emph{smooth} self-similar shrinkers.
Notably, Ilmanen's result did not preclude the possibility that such smooth shrinkers occur with higher multiplicity. 
For example, it did not rule out the possibility that a flow could develop several parallel sheets, possibly connected via small catenoidal necks, which would collapse to the smooth shrinker in the limit.
Recognizing this caveat, Ilmanen posed the following conjecture.

\begin{MOC}[Ilmanen \cite{Ilmanen_sing_MCF_surf}] \label{Conj_Mult1}
Every tangent flow of a 2-dimensional mean curvature flow is a smooth self-similar shrinker \emph{of multiplicity one.}
\end{MOC}

In this paper, we prove a strong form of this conjecture, which includes not just the case of smooth mean curvature flows and which characterizes more than just tangent flows.
In fact, our resolution will apply to a more general class of mean curvature flows through singularities, which addresses the earlier question regarding the theory of existence.
Furthermore, we show that \emph{all} blow-ups (not just tangent flows) have multiplicity one.

We remark that the 1-dimensional analog of the Multiplicity One Conjecture, concerning curve shortening flow, has been proven in \cite{Grayson_1987}, see also \cite{Huisken_1998, Andrews_Bryan_2011}.

There have been numerous approaches to making sense of the mean curvature flow equation in the singular setting.
Brakke's work led to the definition of a weak version of the mean curvature flow equation using the concept of varifolds.
The papers \cite{Osher_Sethian_1988,ES_I,ES_II,ES_III,ES_IV,Chen_Giga_Goto_1991} characterized mean curvature flows as evolving of level sets.
The notion of weak set flows, which was introduced in \cite{Ilmanen_ell_reg, White_00},  characterizes the mean curvature flow equation via a viscosity principle.
Another approach was mean curvature flows with surgery   in the 2-convex case \cite{Huisken_Sinestrari_09, Brendle_Huisken_16, Haslhofer_Kleiner_17, Haslhofer_Kleiner_17_surgery}.
To our knowledge, the best known regularity result so far was no better than varifold regularity for almost every time for general initial data, whereas in the generic case, the solution was only known to be smooth almost everywhere for almost every time \cite{Ilmanen_ell_reg}.
Our new existence and partial regularity theory, which arises from our resolution of the Multiplicity One Conjecture, marks a significant departure from these previous approaches.
It is more elementary, but also, in a certain sense, more general as it unifies previous approaches.
Using the new class of \emph{almost regular} mean curvature flows, we show:

\begin{ExGeneral}
Given a smoothly embedded, compact surface $\MM_0 \subset \IR^3$ there is an innermost flow $\MM^-$ and outermost flow $\MM^+$ \cite{Hershkovits_White_nonfattening}, which is almost regular.  
The singular set of both flows has codimension at least 4 in the spacetime sense.
The level set flow starting from $\MM_0$ is non-fattening if and only if $\MM^- = \MM^+$. 
In this case, the mean curvature flow equation has a unique solution within the class of almost regular flows.
\end{ExGeneral}

\medskip

Our understanding of mean curvature flows of surfaces has come a long way, especially in the last two decades.
In landmark work on the mean convex case, White \cite{White_00,White_03} established existence and uniqueness of the flow, optimal bounds on the dimension of the singular set, and showed that singularities are spherical or cylindrical, and have multiplicity one.  

With the mean convex case now well understood, focus has shifted to general mean curvature flows. 
Ilmanen and White \cite{White_icm_2002} provided an example of a level set flow exhibiting a phenomenon known as ``fattening'', which occurs due to the existence of a non-cylindrical or spherical singularity.
In this case, the flow cannot be \emph{uniquely} extended past the first singular time as a flow of surfaces.

Despite these counterexamples, the significance of flows with cylindrical and spherical singularities persists 
beyond the realm of \emph{mean convex} flows. 
This importance was recognized in two pivotal results:
\begin{itemize}
\item Work of Colding--Minicozzi \cite{Colding_Minicozzi_2012_generic_I} identified  cylinders and spheres as the only singularity models that are \emph{stable} under perturbation.
\item Brendle's work \cite{Brendle_genus0_16} showed cylinders and spheres are the only non-trivial 2-dimensional shrinkers of \emph{genus zero}.
\end{itemize}
The work of Hershkovits--White \cite{Hershkovits_White_nonfattening} and Choi--Haslhofer-Hershkovits \cite{Choi_Haslhofer_Hershkovits_2022}  guaranteed the existence of a unique mean curvature flow ``through singularities'' as long as these singularities are modeled on cylinders or spheres \emph{of multiplicity one.}
Combined with the previous two results this suggested that a general mean curvature flow should mimic mean convex flows, given that its initial condition is a topological 2-sphere or chosen generically. 

The caveat here, however, is again the Multiplicity One Conjecture.
For example, while all tangent flows in genus zero flows must indeed be cylinders or spheres due to Brendle's work, it was not known whether they occur with multiplicity one.
This is critical to understanding evolution past these singularities.
A similar challenge arises with generic initial data.
Chodosh--Choi--Mantoulidis--Schulze \cite{chodosh2023meani,chodosh2023mean} showed that the mean curvature flow starting from a generic surface can be evolved uniquely through cylindrical and spherical singularities, provided they have multiplicity one.
The resolution of the Multiplicity One Conjecture is thus the missing piece in both discussions, allowing us to show:

\begin{ExS2}
If $\MM_0 \subset \IR^3$ is a smoothly embedded surface homeomorphic to a 2-sphere, then there is a \emph{unique} almost regular mean curvature flow $\MM$ and this flow only has cylindrical or spherical singularities.
In other words, in the case of topological 2-spheres, the mean curvature flow equation is well-posed within the class of almost regular flows.
\end{ExS2}
\medskip

\begin{ExGeneric}
If $\MM_0 \subset \IR^3$ is a \emph{generic} smoothly embedded, compact surface, then there is a \emph{unique} almost regular mean curvature flow $\MM$ and this flow only has cylindrical or spherical singularities.
\end{ExGeneric}
\medskip

We also refer to further related work \cite{Zhou_20_mult_1, Li_Wang_22}.

\bigskip
\subsection{Statement of the main results} \label{subsec_mainresults}
The main results of this paper concern a class of singular mean curvature flows in $\IR^3$, which we will call \emph{almost regular} (see Definition~\ref{Def_almost_regular}).
Roughly speaking, an almost regular flow is a weak set flow \cite{Ilmanen_sing_MCF_surf, White_00} that is regular at almost every time and which is the support of a unit-regular, multiplicity one Brakke flow.
For the purpose of this introduction, it is sufficient to know that this class includes most of the commonly used types of mean curvature flows ``through singularities''.
In fact, we have the following result, which will actually be a consequence of our theory.

\begin{Theorem} \label{Thm_nonfattening_almost_regular}
Any outer and innermost mean curvature starting from a compact, embedded surface in $\IR^3$ is almost regular.
In particular, every smooth mean curvature flow and every non-fattening level set flow with compact time-slices is almost regular.
\end{Theorem}

We can now state our main result, which resolves the Multiplicity One Conjecture for bounded almost regular mean curvature flows in $\IR^3$.

\begin{Theorem} \label{Thm_main}
Let $\MM$ be a bounded almost regular mean curvature flow in $\IR^3$ over a time-interval $I$.
Consider any blow-up limit of $\MM$, by which we mean a Brakke flow that is a weak limit of a sequence of parabolic rescalings $\MM^i = \lambda_i (\MM - (x_i, t_i))$ for $\lambda_i \to \infty$ and uniformly bounded $(x_i, t_i)$.
This includes any tangent flow based at any finite time $\leq \sup I$.
Then any such blow-up limit is associated with an almost regular mean curvature flow.
In particular, almost all its tangent planes have multiplicity one.
\end{Theorem}

This theorem can be used to remove an extra assumption in  recent results due to Hershkovits--White, Choi--Haslhofer--Hershkovits,   and Chodosh--Choi--Mantoulidis--Schulze \cite{Hershkovits_White_nonfattening,Choi_Haslhofer_Hershkovits_2022,chodosh2023meani,chodosh2023mean}.
As a consequence, we obtain:

\begin{Corollary} \label{Cor_S2}
Given any smoothly embedded surface $\Sigma \subset \IR^3$ that is diffeomorphic to a 2-sphere, its level set flow is almost regular and non-fattening.
Its tangent flows are  planes, cylinders or spheres and we have uniqueness within the class of almost regular flows.
\end{Corollary}

\begin{Corollary} \label{Cor_generic}
Given any smoothly embedded, compact surface $\Sigma \subset \IR^3$, there is a sequence of surfaces $\Sigma_i \xrightarrow{C^\infty} \Sigma$   
such that for each $i$:
\begin{itemize}
\item  The level set flow $\MM^i$ starting from $\Sigma_i$ is almost regular and non-fattening.
\item All tangent flows of $\MM^i$ are planes, cylinders or spheres.
\item $\MM^i$ is the unique almost regular flow starting from $\Sigma_i$.
\end{itemize}
In particular, the set of surfaces whose level set flow is nonfattening and has only mean convex singularities is open and dense in the smooth topology.
\end{Corollary}

Note that in both corollaries, the flows are unique, so they agree with the unit-regular, cyclic Brakke flows defined in the references given above.  
We remark that by \cite{daniels_holgate} the flows in Corollaries~\ref{Cor_S2} and \ref{Cor_generic} may be approximated by mean curvature flows with surgery.

The proof of Theorem~\ref{Thm_main} relies on the following key estimate, which provides an integral lower bound on the local scale function.

\begin{Theorem} \label{Thm_key_thm}
Let $\MM$ be  an almost regular mean curvature flow in $\IR^3$ over the time-interval $[0,T)$, $T<\infty$, starting from a compact, smooth surface $\MM_0$.
Then there is a constant $C(\MM_0, \lb T, \lb \genus(\MM)) \lb < \infty$, which depends continuously on $\MM_0$ in the $C^2$-sense and continuously on $T$, such that for all $x_0 \in \IR^3$, $t_0 \in [0,T]$ and $0\leq  t < t_0$ we have
\[ \int_{\MM_{\reg,t}} \bigg({ \log \Big(\frac{r_{\loc}(\cdot,t)}{\sqrt{t_0 - t}} \Big)  }\bigg)_-  \rho_{(x_0,t_0)}(\cdot,t) \, d\HH^2 \geq    -C . \]
\end{Theorem}

Here $\rho_{(x_0,t_0)} (x,t) = \frac1{4\pi(t_0-t)} e^{-\frac{|x-x_0|^2}{4(t_0-t)}}$ is the usual ambient Gaussian weight and $r_{\loc}(x,t)$ denotes local scale function, which is roughly the largest $r> 0$ such that on $\MM_{\reg}$ intersected with the parabolic ball $P(x,t, r) = B(x,r) \times [t- r^2, t+r^2]$ we have an upper bound on the second fundamental form of the form $|A| \leq c r^{-2}$ and a lower bound on the normal injectivity radius of the form $\inj^\perp \geq cr$.
The notation $a_- = \min \{ a, 0 \}$ denotes the negative part.

\begin{Corollary} \label{Cor_no_min_surface_blowup}
Let $\MM$ be a bounded almost regular mean curvature flow in $\IR^3$ over the time-interval $[0,T)$, $T < \infty$.
Then no blow-up limit (in the Brakke or smooth sense) is a non-trivial minimal surface.
\end{Corollary}

By ``non-trivial minimal surface'' we mean any complete minimal surface $\Sigma \subset \IR^3$ that is not a single affine plane.
So a union of parallel affine planes is also viewed as non-trivial.

We also obtain the following compactness theorem for almost regular mean curvature flows.

\begin{Theorem} \label{Thm_main_compactness}
Let $\MM^i \subset \IR^3 \times I_i$ be a sequence of bounded almost regular mean curvature flows satisfying one of the following two properties:
\begin{itemize}
\item $\MM^i$ is a blow-up sequence of a bounded almost regular mean curvature flow $\MM \subset \IR^3 \times I$ in the sense that $\MM^i = \lambda_i ( \MM - (x_i, t_i))$, where $\lambda_i \to \infty$ and $(x_i, t_i)$ is uniformly bounded.
\item $I_i = [0, T_i)$, we have smooth convergence $\MM_0^i \to M_\infty \subset \IR^3$ to some surface and $\genus(\MM^i)$ is uniformly bounded.
\end{itemize}
Suppose that $I_i \to I_\infty$ for some interval $I_\infty$, which we assume to be open in the first case and of the form $[0,T_\infty)$ in the second case.
Then the following is true:
\begin{enumerate}[label=(\alph*)]
\item \label{Thm_main_compactness_a} For any $[T_1, T_2] \subset  I_\infty$ and any bounded open subset $U \subset \IR^3$ there is a uniform constant $C(T_1,T_2,U) < \infty$ such that for all $i$ we have
\[ \int_{T_1}^{T_2} \int_{U \cap \MM^i_{\reg,t}} r_{\loc}^{-2} \, d\HH^2 dt \leq C(T_1,T_2,U). \]
\item \label{Thm_main_compactness_b} After passing to a subsequence, the Brakke flows associated with the flows $\MM^i$ converge weakly to a  Brakke flow, which is the associated flow of an almost regular mean curvature flow $\MM^\infty \subset \IR^3 \times I_\infty$.
Moreover, for every regular time $t \in I_\infty$ of $\MM^\infty$ we have local smooth convergence $\MM^i_t \to \MM^\infty_t$.
\end{enumerate}
\end{Theorem}

We remark that the methods in our paper provide an alternate proof of the fact that the space of complete, smooth shrinkers of multiplicity one in $\IR^3$ is compact (under mild regularity assumptions at infinity); see \cite{Colding_Minicozzi_12}.

We also answer a question from \cite[Remark~2.3]{Hershkovits_White_nonfattening} in the case of 2-dimensional mean curvature flows.
In terms of this reference, the following theorem asserts that the fattening time $T_{\fat}$ is equal to the discrepancy time $T_{\disc}$.

\begin{Theorem} \label{Thm_Tfat_Tdisc}
Given a smoothly embedded, compact surface $\Sigma \subset \IR^3$, there is a time $T \in (0,\infty]$ such that the level set flow $\MM$ starting from $\Sigma$ is almost regular.
Moreover, on $\IR^3 \times [0,T]$ the flow $\MM$ agrees with the innermost flow $\MM^-$ and outermost flow $\MM^+$ starting from $\Sigma$.

If $T < \infty$, then past time $T$ the flow fattens and the inner and outermost flows begin to differ from $\MM$.
More specifically, there is a sequence of times $t_i \searrow T$ such that:
\begin{enumerate}[label=(\alph*)]
\item \label{Thm_Tfat_Tdisc_a} $\MM_{t_i}$ is fat, i.e., it has non-empty interior.
\item \label{Thm_Tfat_Tdisc_b} $\MM_{t_i} \neq \MM^+_{t_i}, \MM^-_{t_i}$.
\end{enumerate}
Lastly, if $\MM_t$ is not fat for all $t$ in an open time-interval $I \subset [0,\infty)$, then the flows $\MM, \MM^-, \MM^+$ restricted to $\IR^3 \times I$ agree.
\end{Theorem}

Combining Theorem~\ref{Thm_main}
with earlier work 
\cite{Cheeger_Haslhofer_Naber_13,chodosh2023mean,bernstein_wang_topological_property,colding_minicozzi_singular_set_generic,Wang_Lu_asymptotic_2016},
we obtain the following results on the singular set of an almost regular mean curvature flow.

\begin{Theorem}
\label{Thm_Msing_dim}
Let $\MM$ be a bounded almost regular mean curvature flow in $\IR^3$.  Then:
\begin{enumerate}[label=(\alph*)]
\item \label{Thm_Msing_dim_a} The singular set has spacetime Minkowski dimension $\leq 1$.  
\item \label{Thm_Msing_dim_b} The set 
of nongeneric singular points is  closed, countable, and
has spacetime Minkowski  dimension $0$.
\item \label{Thm_Msing_dim_c} The set of nongeneric singular points is backward isolated, i.e. if $(x,t)\in\MM$ is a nongeneric singular point, then $(x,t)$ is the only nongeneric singular point in the backward parabolic ball $P^-(x,t,r)$,   for some $r>0$.    If $\MM$ is an outer or innermost flow, then the nongeneric singularities form 
 a finite set of cardinality bounded by the genus of $\MM_0$.
\item \label{Thm_Msing_dim_d} On the complement of the nongeneric singular points, the singular set is contained in a locally finite collection of Lipschitz curves (with respect to the parabolic metric), plus a collection of isolated spherical singular points.
\end{enumerate}
Similar statements hold if we drop the assumption of $\MM$ being bounded and assume instead that $\MM$ is a blow-up limit of bounded almost regular mean curvature flows as in Theorem~\ref{Thm_main_compactness}.
\end{Theorem}

\bigskip
Lastly, we give a proof of a conjecture of Ilmanen on the asymptotic structure of tangent flows \cite[p.39]{Ilmanen_lectures_mcf_related_equations}, removing a finite topology condition from \cite{Wang_Lu_asymptotic_2016} (see also \cite{song_maximum_principle_self_shrinkers}).
\begin{Theorem}
\label{Thm_asymptotic_structure_tangent_flow}
Let $\MM$ be a tangent flow of a bounded almost regular mean curvature flow or, more generally, a shrinker which is a limit of bounded almost regular mean curvature flows as in Theorem~\ref{Thm_main_compactness}.
Then the collection $\mathcal{E}$ of ends of $\MM_{-1}$  is finite and there a collection $\{\Gamma_E\}_{E\in \mathcal{E}}$ of pairwise disjoint   subsets of the unit sphere $S^2$ such that for every $E\in \mathcal{E}$ one of the following holds:
  
\begin{itemize}
\item $\Gamma_E\subset S^2$ is a smooth simple closed curve, and $r^{-1}E$ converges smoothly on compact subsets of $\IR^3\setminus\{0\}$ to the cone over $\Gamma_E$ as $r\rightarrow\infty$.  
\item $\Gamma_E=\{v_E\}$ is a single point, and $E-rv_E$ converges smoothly on compact sets to the round cylinder of radius $\sqrt{2}$ with axis $\IR v_E$ as $r\rightarrow\infty$.  
\end{itemize}
 \end{Theorem}

We remark that the arguments in this paper adapt with minor modifications to the setting of Riemannian 3-manifolds which are either compact or satisfy appropriate bounds at infinity.  


\bigskip
\subsection{Outline of the proof}
To illustrate the idea of the proof of Theorem~\ref{Thm_main}, let us first consider a simplified case.
Suppose that $\MM$ is a smooth mean curvature flow in $\IR^3$ over the time-interval $[-1,0)$ that develops a singularity at the final time.
Assume that its tangent flow at the origin $(\vec 0, 0)$ is a 2-sphere of multiplicity 2 -- a behavior which the Multiplicity One Conjecture alleges to be impossible.
More specifically, we assume that every time-slice $\MM_t$ is a double sheeted approximation of the round-shrinking sphere $2\sqrt{-t} \, S^2$, where we allow both sheets to be connected by a small catenoidal neck.
For example on the complement of a neighborhood of this catenoidal neck, the time-slice $\MM_t$ could look like 
$$\big((2-\delta(t))\sqrt{-t} \, S^2 \big) \cup  \big( (2+\delta(t)) \sqrt{-t} \, S^2 \big), \qquad \delta(t) \to 0.$$
To make our simplifying assumptions concrete, we choose a family of points $x_{\cat}(t) \in 2\sqrt{-t} \, S^2$ (the center of the catenoidal neck) and radii $0 < r_{\cat}(t) \ll \sqrt{-t}$ (its scale) such that:
\begin{enumerate}[label=(\Alph*)]
\item \label{simass1} The complement $\MM_t \setminus \ov{B(x_{\cat}(t), 100 r_{\cat}(t))}$ is the union of the graphs of two functions $u_{1,t} < u_{2,t}$ over open subsets of the sphere $2\sqrt{-t}\, S^2$ with $|u_1|, |u_2| \ll \sqrt{-t}$.
\item \label{simass2} Near $x_{\cat}(t)$ the flow is close to a Catenoid at scale $r_{\cat}(t)$, which moves and changes size slowly at this scale.
By this we mean that the rescaling $r_{\cat}^{-1}(t) ( \MM_t - x_{\cat}(t))$ restricted to the ball $B(\vec 0, \eps^{-1})$ is $\eps$-close to a Catenoid of fixed size restricted to the same ball. Moreover,
$$|x'_{\cat}(t)|, \; |r'_{\cat}(t)|  \leq \eps r^{-1}_{\cat}(t) .$$
Here $\eps > 0$ is a small constant, which we will choose later.
\end{enumerate}

Let us now introduce the \emph{separation function} $\s$, which will be the key quantity in our analysis.
It will measure the separation of the two sheets defined by the graphs of $u_1$ and $u_2$, so we will roughly have $\s \approx u_2 - u_1$ on the complement of the catenoidal neck.
More specifically, for every point $x \in \MM_t$ we define $\s (x,t)$ to be the infimum over all $r > 0$ such that the intersection $\MM_t \cap B(x,r)$ is disconnected.
In order to simplify our discussion, let us assume that:
\begin{enumerate}[label=(\Alph*),start=3]
\item \label{prop_C} If $(x,t) \in \MM_t \setminus \ov{B(x_{\cat}(t), 100 r_{\cat}(t))}$, then we assume $\s(x,t)$ to be equal to the distance to the opposite graph from Assumption~\ref{simass1}, i.e., the graph that does not contain $x$.
For example, this can be ensured by imposing appropriate  derivative bounds on $u_1, u_2$.
Moreover, we assume that $\s$ is smooth.
\end{enumerate}
In Proposition~\ref{Prop_log_u_super_sol} we will derive the following crucial differential inequality on the separation function:
\begin{equation} \label{eq_sqs_outline}
 \square \log \s = (\partial_t - \triangle) \log \s \geq 0. 
\end{equation}
This inequality intuitively expresses the flow's inclination to uphold separation between two disconnected sheets.
However, it is important to note that this differential inequality  might only be valid in regions where the flow can be understood by the union of two graphs, so on $\MM_t \setminus \ov{B(x_{\cat}(t), 100 r_{\cat}(t))}$.
It may fail near the catenoidal neck.

Let us now recall the Gaussian background density on $\IR^3 \times \IR$
\[ \rho(x,t) := \frac1{4\pi (-t)} e^{-\frac{|x|^2}{4(-t)}} \]
The restriction of $\rho$ to the flow is a subsolution to the conjugate heat equation, that is
\begin{equation} \label{eq_outline_sqrho}
 \square^* \rho = (-\partial_t - \triangle + |\mathbf{H}|^2 ) \rho \geq 0, 
\end{equation}
where $\mathbf{H}$ denotes the mean curvature vector (see Lemma~\ref{Lem_entropy_rho} for further details).
Note that integrating this inequality in space and time implies Huisken's monotonicity formula.
See Lemma~\ref{Lem_entropy_rho} for more details.

Our strategy will now be to integrate the bound \eqref{eq_sqs_outline} against the density $\rho$ over the part of the flow where this bound holds.
To achieve this, we need to introduce a cutoff function on $\MM$.
To do this, fix a large constant $100 \leq R \leq \eps^{-1}$.
By composing the distance function to $x_{\cat}(t)$ divided by $r_{\cat}(t)$ with a suitable cutoff function, we can obtain a smooth function $\eta (\cdot, t) \in C^\infty(\MM_t)$ with
\begin{alignat*}{2}
 \eta &\equiv 0 &&\textQQq{on} \MM_t \setminus B(x_{\cat}(t), Rr_{\cat}(t)), \\
 \eta &\equiv 1 &&\textQQq{on} \MM_t \cap B(x_{\cat}(t), \tfrac12 Rr_{\cat}(t)) 
\end{alignat*}
\begin{equation} \label{eq_eta_bounds_outline}
 |\nabla \eta| \leq 10 R^{-1} r_{\cat}^{-1}(t), \qquad |\partial_t \eta| \leq C(R) \eps r_{\cat}^{-2}(t).  
\end{equation}
Using this function, we can define the following time-dependent quantity, for $\tau \in (0,1]$,
\[ \mathfrak{S}( \tau ) := \int_{\MM_{-\tau}} \big( (1-\eta(\cdot,-\tau)) \log (\s(\cdot,-\tau)) + \eta(\cdot,-\tau) \log ( r_{\cat}(-\tau)) \big) \rho(\cdot, -\tau)\, d\HH^2. \]
Roughly speaking, $\mathfrak{S}(\tau)$ measures the ``average separation'' of the flow at time $t=-\tau$.
The integral on their right-hand side can be viewed as modification of the integral of $\log \s$ over $\MM_t$.
The term $r_{\cat}$, which is usually comparable to $\s$, serves as a substitute for $\s$ in regions where $\s$ has undesirable behavior.
Moreover, the integration is performed over the weight $\rho$, but at least in our setting, we have $\rho \approx \frac{c}{\tau}$ on $\MM_t$.

Let us now compute the time-derivative of the quantity $\mathfrak{S}(\tau)$.
Omitting the arguments for clarity, we obtain using integration by parts
\begin{equation*}
 \mathfrak{S}'(\tau) = -\int_{\MM_{-\tau}}  \square \big( (1-\eta) \log \s + \eta \log r_{\cat}) \big)  \rho \, d\HH^2 + \int_{\MM_{-\tau}} \big( (1-\eta) \log \s + \eta \log r_{\cat}) \big) \square^* \rho \, d\HH^2 .
\end{equation*}
Since $\s, r_{\cat} \ll 1$, the term in the parenthesis of the second integral is negative.
Combining this with \eqref{eq_outline_sqrho} and \eqref{eq_sqs_outline} yields after another integration by parts (see Lemma~\ref{Lem_evol_formula} for further details)
\begin{align}
 \mathfrak{S}'(\tau) &\leq -\int_{\MM_{-\tau}}  \square \big( (1-\eta) \log \s + \eta \log r_{\cat}) \big)  \rho \, d\HH^2 \notag \\
 &\leq -\int_{\MM_{-\tau}}   \bigg(  (1-\eta) \square\log \s - (\square \eta) \log \s + 2 \frac{\nabla \eta \cdot \nabla \s}{\s}  + \square \big( \eta \log  r_{\cat}   \big) \bigg)   \rho \, d\HH^2 \notag \\
 &\leq -\int_{\MM_{-\tau}} \bigg( - (\partial_t \eta) \log \Big({\frac{\s}{r_{\cat}}}\Big) \rho +  \eta \frac{r'_{\cat}}{r_{\cat}} \rho + \frac{\nabla \eta \cdot \nabla \s}{\s} \rho - \log \Big({ \frac{\s}{r_{\cat}}}\Big)   \nabla \eta \cdot \nabla \rho \bigg) d\HH^2. \label{eq_SS_estimate_outline} %
\end{align}
Using the bounds \eqref{eq_eta_bounds_outline} on $\eta$, the fact that $|\nabla \s| \leq 1$ and the fact that
\[ C^{-1}(R) r_{\cat} \leq \s \leq C(R) r_{\cat} \textQQqq{on} U:= \MM_ {-\tau} \cap A(x_{\cat}, \tfrac12 R r_{\cat}, R r_{\cat}), \]
allows us to deduce that
\[ \mathfrak{S}'(\tau) \leq  \frac{C(R)\eps}{\tau} + \frac{C}{\tau} \int_U \frac{|\nabla \s|}{Rr_{\cat} \s} d\HH^2 \]
The last integral is scaling invariant.
So since $r_{\cat}^{-1} (U - x_{\cat})$ is $\eps$-close to the annular region $\hat U := \Sigma_{\cat} \cap A(\vec 0, \frac12 R, R)$ of a Catenoid $\Sigma_{\cat} \subset \IR^3$ of fixed size, we obtain for sufficiently small $\eps$
\[ \int_U \frac{|\nabla \s|}{Rr_{\cat} \s} d\HH^2
\approx \int_{\Sigma_{\cat} \cap A(\vec 0, \frac12 R, R)} \frac{|\nabla \s|}{R \s} d\HH^2. \]
It can be checked easily that this term converges to $0$ as $R \to \infty$, for example, by realizing that on $\Sigma_{\cat}$ we have $\frac{|\nabla \s|}{ \s} = o(r^{-1})$.
So by choosing $R$ sufficiently large and subsequently $\eps$ sufficiently small, depending on a given $\beta > 0$, we obtain 
\[  \mathfrak{S}'(\tau) \leq  \frac{\beta}{\tau} . \]
Thus
\begin{equation} \label{eq_SS_lower_outline}
 \mathfrak{S}(\tau) \geq  \beta \log \tau - C_1 
\end{equation}
However, by the definition of $\mathfrak{S}(\tau)$, and the fact that $\s, r_{\cat} \ll \sqrt{\tau}$, we must have
\[ \mathfrak{S}(\tau) \leq  \big( C_2 + \tfrac12 \log \tau \big) \int_{\MM_{-\tau}} \rho(\cdot, -\tau) \, d\HH^2 = \big( C_2 + \tfrac12 \log \tau \big) \Theta(\tau). \]
The last term, which is the Gaussian area, is uniformly bounded from below.
So we obtain a contradiction to \eqref{eq_SS_lower_outline} for small $\beta$ and $\tau$, which concludes our sketch of the proof under simplifying assumptions.

\medskip

Let us now extend the discussion of our strategy to the general setting. 
When dealing with a general flow, several complexities arise:
\begin{itemize}
\item The flow might lack a straightforward description as presented in Properties~\ref{simass1}--\ref{prop_C} above. 
For instance, it could comprise \emph{numerous} sheets connected by necks of varying scales. Moreover, these necks' geometries might be modeled on more complicated minimal surfaces, and there might not be any bound on the size of the region consisting of such necks.
\item Certain regions within the flow may defy categorization as ``sheets'', i.e., unions of graphs, or ``necks'' modeled on minimal surfaces as we have presumed. 
For instance, some areas might closely resemble the Grim Reaper translating soliton times $\IR$ or exhibit entirely unknown geometries.

An anticipated behavior could unfold as follows: 
Imagine the flow contains a region initially resembling a Catenoid, gradually expanding its hole until the solution resembles a rotationally symmetric soliton with a Grim Reaper-like cross-section. 
Eventually, this process might reverse, leading the flow to converge back to a Catenoid of larger scale.
\item Even if describing the flow as a union of sheets is feasible within a certain region, two sheets might unexpectedly collapse and nullify each other, resulting in a localized drop in mass.
\item Higher multiplicity may occur and there may be non-spherical singularity models, such as affine planes, on which the Gaussian density $\rho$ is not almost constant.
\end{itemize}

A key insight, which allows us to overcome these difficulties, is the fact that the size of the region where the flow deviates too much from our characterization of ``sheets'' and ``necks'' can be controlled sufficiently well.
This will allow us to show that this region hardly affects the evolution of $\mathfrak{S}(\tau)$.
To be more specific, in Section~\ref{sec_int_bound}, we will choose a set of points $\MM_{\reg}^{(\eps)} \subset \MM$, which roughly consists of points near which the flow is either locally close to a union of sheets or a non-trivial minimal surface.
We will then show in Theorem~\ref{Thm_L2_bound} that whenever $r_{\loc}$ is sufficiently small locally, then
\begin{equation} \label{eq_outline_rlocm2}
 \int_{t_0 - \tau_2}^{t_0 - \tau_1} \frac1{t_0 - t} \int_{(\MM_{\reg}  \setminus \MM_{\reg}^{(\eps)})_t \cap B(x_0, A \sqrt{t_0 - t})} r_{\loc}^{-2} \, d\HH^2 \, dt \leq C(A,\eps) \big( \Theta_{(x_0,t_0)}(2\tau_2) - \Theta_{(x_0,t_0)}(\tfrac12 \tau_1) \big). 
\end{equation}
Here the left-hand side is the integral over the complement of $\MM_{\reg}^{(\eps)}$ over the time-interval $[t_0 - \tau_2, t_0-\tau_1]$, localized to balls of radius $B(x_0, A \sqrt{t_0 -t})$.
The right-hand side denotes the difference between two Gaussian area quantities.

In vague terms, we will bound the influence of the region $\MM \setminus \MM_{\reg}^{(\eps)}$ on the change $\mathfrak{S}(\tau_1) - \mathfrak{S}(\tau_2)$ by the left-hand side of \eqref{eq_outline_rlocm2}. 
As this influence is bounded by the change of a bounded and monotone quantity, this will lead to a uniform upper bound, which can be absorbed in he constant $C_1$ from \eqref{eq_SS_lower_outline}.

To overcome the remaining challenges, our approach hinges on two critical elements: a careful selection of the cutoff function $\eta$ and the choice of a proper definition of the subset $\MM_{\reg}^{(\eps)}$.
These finer points are beyond the scope of this outline.

A further aspect that is worth discussing, however, is our strategy for dealing with the possibility of sudden local mass drop.
For instance, we must address scenarios where a region of the flow, at some scale $r_0 > 0$, closely resembles the union of an almost flat sheet and a nearby component, which may be close to a double sheet with a catenoidal neck and which vanishes almost abruptly at some time $t_0$.
In this case the separation function $\s$ restricted to the former sheet increases quickly at time $t_0$.
While this aligns with the notion behind \eqref{eq_sqs_outline}, it presents a problem shortly after time $t_0$.
More precisely, at such times we may have $r_1 :=\s \gg r_0$ and the almost flat sheet may have non-trivial geometry at scale $r_1$.
As a result, the bound \eqref{eq_sqs_outline} may fail in this region approximately for times in the time-interval $[t_0,t_0 + r_1^2]$.
To address this issue, we will use a smoothed version $\td r_{\loc}$ of the local scale function as a surrogate for $\s$ during this time-interval.
This function is spatially almost constant and increases from $r_0$ to $r_1$ in the time-direction, so it also satisfies the bound \eqref{eq_sqs_outline}.
In order to exploit this fact in a general setup, we will merge the quantities $\s$ and $\td r_{\loc}$ into a single quantity by setting
\[ \s_{\loc} := \min \big\{ \tfrac12 \s, \td r_{\loc} \big\}. \]
This modified quantity continues to meet the condition in \eqref{eq_sqs_outline} while encapsulating the features of both $\s$ and $\td r_{\loc}$ in the intended regions (see Subsection~\ref{subsec_mod_sep} for further details).
\bigskip

\subsection{Structure of the paper}
In Section~\ref{sec_Prelim} we discuss the necessary preliminaries for the proof, with Subsection~\ref{subsec_almostreg} containing the definition of almost regular mean curvature flow.
Section~\ref{sec_sep_fct} introduces the separation function and establishes the crucial estimate \eqref{eq_sqs_outline}.
In Section~\ref{sec_smooth_rloc} we define a smoothed version $\td r_{\loc}$ of the local scale function $r_{\loc}$ and discuss its fundamental properties.
In Section~\ref{sec_int_bound}, we derive the integral bound~\eqref{eq_outline_rlocm2}.
In Section~\ref{sec_main_argument} we carry out the main proof.
This section cointains the construction of the cutoff function $\eta$ in Subsection~\ref{subsec_cutoff} and the definition of the modification of the separation function in Subsection~\ref{subsec_mod_sep}.
Subsection~\ref{subsec_integral_estimate} introduces the quantity $\mathfrak{S}(\tau)$ and derives an integral estimate akin to \eqref{eq_SS_estimate_outline}.
In Subsection~\ref{subsec_proof_key} we combine our results thus far to establish the key estimate of this paper, Theorem~\ref{Thm_key_thm}.
Section~\ref{sec_remaining_thms} contains the proofs of the remaining main theorems.

\subsection{Acknowledgements}
We like to thank Reto Buzano, Otis Chodosh, Robert Haslhofer and David Hoffman for comments and corrections on an earlier version of this paper.

\bigskip

\section{Preliminaries} \label{sec_Prelim}
\subsection{Terminology and basic definitions} \label{subsec_terminology}
In the following, we will often consider objects living in the spacetimes $\IR^{n+1} \times \IR$ or $\IR^{n+1} \times I$, where $I$ denotes an interval.
We denote by $\mathbf{x} : \IR^{n+1} \times \IR \to \IR^{n+1}$ and $\mathbf{t} : \IR^{n+1} \times \IR \to \IR$ the projections on the space and time coordinates.
We also define the \textbf{time vector field} $\partial_{\mathbf{t}}$ on $\IR^{n+1} \times \IR$ as the unit vector field on the second factor.
Note that $\partial_{\mathbf{t}} \mathbf{t} \equiv 1$.
Given a point $(x,t) \in \IR^{n+1} \times \IR$, we define the (backwards) parabolic balls
\[ P(x,t,r) := B(x,r) \times [t-r^2, t+r^2], \qquad
P^-(x,t,r) := B(x,r) \times [t-r^2, t]. \]
If $\MM \subset \IR^{n+1} \times \IR$ is a subset of spacetime, $(x,t) \in \IR^n \times \IR$ is a point and $\lambda > 0$, then we define the \textbf{parabolic rescaling by $\lambda$, centered at $(x,t)$,} by
\[ \lambda \big(\MM - (x,t) \big) := \big\{ \big(\lambda (x'-x), \lambda^2 (t'-t) \big) \;\; : \;\; (x',t') \in \MM \big\}. \]

If $\MM \subset \IR^{n+1} \times \IR$ and $t \in \IR$ is a time, then we write $\MM_t := \MM \cap \mathbf{t}^{-1}(t) \subset \IR^{n+1} \times \{t \}$ to denote the \textbf{time-$t$-slice} of $\MM$.
Depending on the context, we will sometimes also view $\MM_t$ as a subset of $\IR^{n+1}$.
For example, we will sometimes write $\MM_t \cap B(x,r)$  to denote the intersection of $\MM_t \subset \IR^{n+1} \times \{t \}$ with $B(x,r) \times \{ t \}$ and we will sometimes view this set as a subset of $\IR^{n+1}$ or of $\IR^{n+1} \times \{ t \}$, depending on the setting.
Similarly, if $u : \MM \to \IR$ denotes a function defined on $\MM$ (or on the ambient spacetime), then we will often write
\[ \int_{\MM_t} u \, d\HH^{n} =  \int_{\MM_t} u(\cdot,t) \, d\HH^{n} \]
to denote the integral over the time-$t$-slice of $\MM_t$ with respect to the $n$-dimensional Hausdorff measure.
In the first integral, we view $\MM_t$ as a subset of $\IR^{n+1} \times \{ t \}$, so the additional argument ``$(\cdot, t)$'' becomes obsolete.
If $I' \subset \IR$ is a set of times, most commonly an interval, then we define the \textbf{time-$I'$-slab} by $\MM_{I'} := \MM \cap \mathbf{t}^{-1}(I')$.
We also sometimes use the shorthand notation $\MM_{<t_0} := \MM_{(-\infty,t_0)}$, etc.

If $\MM \subset \IR^{n+1} \times I$, then we call a point $(x,t) \in \MM$  \textbf{regular} (within $\IR^{n+1} \times I$) if a neighborhood of $(x,t)$ within $\MM$ is a smooth hypersurface with boundary in $\IR^{n+1} \times \partial I$ such that $d\mathbf{t}$ restricted to this hypersurface vanishes nowhere.
This implies that there is a neighborhood $U$ of $x \in \IR^{n+1}$ such that for $t'$ close to $t$ the time-slices $\MM_{t'} \cap U$ form a smooth family of $n$-dimensional submanifolds.
We denote by $\MM_{\reg} \subset \MM$ the (open) subset of regular points and we will commonly write $\MM_{\reg,t}$ instead of $(\MM_{\reg})_t$.
We also denote by $\MM_{\sing} := \MM \setminus \MM_{\reg}$ the set of all \textbf{singular} points.
We define the \textbf{normal time vector field} $\partial_t$ to be unique the vector field along $\MM_{\reg}$, which is tangent to $\MM_{\reg}$ and has the property that the difference $\partial_t - \partial_{\mathbf{t}}$ is spatial (i.e., $d\mathbf{t} (\partial_t - \partial_{\mathbf{t}}) =0$) and normal when restricted to each time-slice $\MM_{\reg, t}$.
Then $\partial_t u$, for some sufficiently regular function $u \in \MM \to \IR$, is the \textbf{normal time derivative,} meaning the time derivative of an observer who is confined to $\MM$ and moves with $\MM$ in its normal direction.
We admit that the symbols $\partial_t$ and $\partial_{\mathbf{t}}$ are typographically similar.
However, in the rest of the paper we will almost exclusively work with the normal time derivative $\partial_t$ whenever we perform a computation on $\MM_{\reg}$.

A \textbf{smooth mean curvature} flow in $\IR^{n+1}$ (over a time-interval $I$) is given by a subset $\MM \subset \IR^{n+1} \times I$ such that $\MM_{\reg} = \MM$ and such that the family of submanifolds $\MM_t \subset \IR^{n+1}$ moves by its mean curvature, meaning that
\begin{equation} \label{eq_MCF_equation}
 \partial_t \mathbf{x} = \mathbf{H}, 
\end{equation}
where $\partial_t$ is the normal time derivative and $\mathbf{H}$ is the spatial vector field that restricts to the mean curvature vector field on every time-slice $\MM_t \subset \IR^{n+1} \times \{t \}$.
Note that \eqref{eq_MCF_equation} is equivalent to $\partial_t = \partial_{\mathbf{t}} + \mathbf{H}$.
\medskip

A \textbf{weak set flow} (over a time-interval $I$) is a closed subset $\MM \subset \IR^{n+1} \times I$ with the property that for any compact, smooth mean curvature flow $\MM'$ over a subinterval $[t_0, t_1] \subset I$ we have the following exclusion principle:
\[ \text{If $\MM_{t_0} \cap \MM'_{t_0} = \emptyset$, then $\MM_t \cap \MM'_t = \emptyset$ for all $t \in [t_0, t_1]$.} \]
We call $\MM$ a \textbf{maximal weak set flow} or \textbf{level set flow} if there is no other weak set flow $\MM' \supsetneq \MM$ with the property that $\MM'_t = \MM_t$ for some time $t \in I$.
Every smooth mean curvature flow with the property that time-slabs over compact time-intervals are compact is a level set flow.
Conversely, every weak set flow without singular points is a smooth mean curvature flow.

For any closed, smoothly embedded hypersurface $\Sigma \subset \IR^{n+1}$ there is a unique level set flow $\MM \subset \IR^{n+1} \times [0,\infty)$ with $\MM_0 = \Sigma$.
However, this flow may \textbf{fatten,} meaning it may have non-empty interior as a subset of $\IR^{n+1} \times [0,\infty)$.
There are two additional canonical ways of evolving $\Sigma$ by a weak set flow.
Let $K^+ \subset \IR^{n+1}$ be the compact domain bounded by $\Sigma$ and define $K^- := \ov{\IR^{n+1} \setminus K^+}$ to be the corresponding outer domain.
Let $\mathcal{K}^-, \mathcal{K}^+ \subset \IR^3 \times [0, \infty)$ be the level set flows starting from $ K^-, K^+$, respectively.
We define $\MM^- := \partial \mathcal{K}^-$ to be the \textbf{innermost flow} and $\MM^+ := \partial \mathcal{K}^+$  the \textbf{outermost flow} starting from $\Sigma$.
Note that the boundary is taken within $\IR^{n+1} \times [0,\infty)$, so the initial time-slices agree $\MM_0 = \MM^-_0 = \MM^+_0 = \Sigma$.
Both flows $\MM^-, \MM^+$ are weak set flows with initial condition $\MM_0 = \Sigma$.
It is elementary to check  that $\MM \subset \MM^\pm$.
Moreover, if $\MM_{[0,T)} = \MM^-_{[0,T)} = \MM^+_{[0,T)}$, then $\MM_{[0,T)}$ is non-fattening \cite[Remark~2.3]{Hershkovits_White_nonfattening}.
The converse statement is unknown in general, but shown for $n=2$ in Theorem~\ref{Thm_Tfat_Tdisc}.
\medskip

We will also use another weak formulation of the mean curvature equation known as Brakke flow.
An $n$-dimensional \textbf{Brakke flow} (over a time-interval $I$) is given by a family of Radon measures $(\mu_t)_{t \in I}$ on $\IR^{n+1}$ such that:
\begin{enumerate}
\item For almost all $t \in I$ the measure is integer $\HH^{n}$-rectifiable and the associated varifold has locally bounded first variation with variational vector field $\mathbf{H} \in L^1(\IR^{n+1}, \mu_t)$.
\item We have $\int_{t_1}^{t_2} \int_{K} |\mathbf{H}|^2 \, d\mu_t \, dt < \infty$ for any compact $K \subset \IR^{n+1}$ and $[t_1,t_2] \subset I$.
\item For any $[t_1,t_2] \subset I$ and compactly supported, non-negative $u \in C^1_c(\IR^{n+1} \times [t_1,t_2])$ we have
\[ \int_{\IR^{n+1}} u(\cdot, t) \, d\mu_t \bigg|_{t=t_1}^{t=t_2} \leq \int_{t_1}^{t_2} \int_{\IR^{n+1}} \big( \partial_{\mathbf{t}} u + \nabla u \cdot \mathbf{H} - u |\mathbf{H}|^2\big) d\mu_t \, dt. \]
\end{enumerate}
A Brakke flow $(\mu_t)_{t \in I}$ is called \textbf{regular} on an open subset $U \subset \IR^{n+1} \times I$ if there is a smooth, properly embedded hypersurface $N \subset U$ with $\partial N \subset \IR^{n+1} \times \partial I$ such that $d\mathbf{t}$ restricted to $N$ is nowhere vanishing and $\mu_t \lfloor U_t = \HH^{n} \lfloor N_t$ for all $t$.
If $\MM \subset \IR^{n+1} \times I$ is a smooth mean curvature flow, then $(\mu_t := \HH^{n} \lfloor \MM_t)_{t \in I}$ is a Brakke flow that is regular everywhere.
The \textbf{support} of a Brakke flow $(\mu_t)_{t \in I}$ is defined to be
\[ \MM := \ov{\bigcup_{t \in I} (\supp \mu_t) \times \{ t \} }. \]
where the closure is taken within $\IR^{n+1} \times I$.
This closure is a weak set flow \cite[10.7]{Ilmanen_ell_reg}.
It is also clear that if $(\mu_t)_{t \in I}$ is regular at some point $(x_0, t_0) \in \MM$, then $(x_0,t_0) \in \MM_{\reg}$.

We define the ambient Gaussian density centered at $(x_0,t_0) \in \IR^{n+1} \times \IR$ by
\begin{equation} \label{eq_def_rho}
 \rho_{(x_0,t_0)} (x,t) := \frac{1}{4\pi(t_0-t)} e^{-\frac{|x-x_0|^2}{4(t_0 -t)}}, \qquad t < t_0. 
\end{equation}
For an $n$-dimensional Brakke flow $(\mu_t)_{t \in I}$ in $\IR^{n+1}$ we define  \textbf{Huisken's Gaussian area}, whenever $t_0 - \tau \in I$, as
\begin{equation} \label{eq_def_Theta}
 \Theta_{(x_0,t_0)}(\tau) := \int_{\IR^{n+1}} \rho_{(x_0,t_0)} (\cdot,t_0-\tau) \, d\mu_{t_0 - \tau}. 
\end{equation}
Recall that $\Theta_{(x_0,t_0}(\tau)$ is non-decreasing in $\tau$ (see also Lemma~\ref{Lem_entropy_rho}), so if $t_0 \leq \sup I$, then we can define the \textbf{Gaussian density at $(x_0,t_0)$} by
\[ \Theta\big((\mu_t)_{t \in I},(x_0,t_0)\big):= \lim_{\tau \searrow 0} \Theta_{(x_0,t_0)}(\tau). \]
Note that $\Theta((\mu_t)_{t \in I},(x_0,t_0) \geq 1$ if and only if $(x_0, t_0)$ is contained in the support of $(\mu_t)_{t \in I}$.
A Brakke flow $(\mu_t)_{t \in I}$ is called \textbf{unit-regular} if it is regular near every point $(x_0, t_0)$ that satisfies $\Theta((\mu_t)_{t \in I}, \lb (x_0,t_0)) = 1$.
We also we define
\[ \Theta\big((\mu_t)_{t \in I} \big) := \sup_{\substack{(x_0,t_0), \tau>0, \\ t_0 -\tau \in I}} \Theta_{(x_0,t_0)}(\tau). \]

\bigskip

\subsection{Local scale and almost regular mean curvature flows} \label{subsec_almostreg}
In this paper we will mainly consider mean curvature flows in $\IR^3$ with relatively mild singularities, which we will call \emph{almost regular} mean curvature flows.
It will turn out that this class is sufficiently general, as it includes the most common types of singular mean curvature flows.
To define this regularity class, consider first a weak set flow $\MM \subset \IR^{n+1} \times I$

\begin{Definition} \label{Def_rloc}
For any point $(x,t) \in \IR^{n+1} \times I$, we define the \textbf{local scale $r_{\loc}(x,t) \in [0, \infty]$} at $(x,t)$ as follows.
If $(x,t) \in \MM_{\sing}$, then we set $r_{\loc}(x,t) := 0$.
Otherwise, we define $r_{\loc}(x,t)$ to be the supremum of all $r > 0$ such that the following is true for all $t'\in [t-r^2, t+r^2] \cap I$:
\begin{enumerate}[label=(\arabic*)]
\item \label{Def_rloc_1} $[t-r^2, t+r^2] \cap I$ is compact.
\item \label{Def_rloc_2} $P(x,t,r) \cap \MM_{\sing} = \emptyset$.
\item \label{Def_rloc_3} We have the following bound on the second fundamental form 
\[ \big| A_{\MM_{\reg,t'}} \big| \leq 10^{-2} r^{-1} \textQQqq{on} B(x,r) \cap \MM \]
\item \label{Def_rloc_4} The intersection $B(x,r) \cap \MM_{t'}$ is connected.
\end{enumerate}
\end{Definition}

Note for any $(x,t) \in \MM_{\reg}$ the definition implies a lower bound on the normal injectivity radius to $\MM_{\reg,t}$ at $x$ of the form $c r_{\loc}(x,t)$, for some dimensional constant $c>0$. 
We also remark that the bound on the second fundamental form in Property~\ref{Def_rloc_3} has been chosen in such a way that any $t'\in [t-r^2, t+r^2]$ the intersection $B(x,r) \cap \MM_{t'}$ is the graph of a smooth function over an open subset of any tangent plane to $B(x,r) \cap \MM_{t'}$.
In addition, the condition guarantees that for any smaller ball $B(x',r') \subset B(x,r)$ the intersection $B(x',r') \cap \MM_{t'}$ is connected as well.
As a result, we obtain the following regularity property, which we will frequently use throughout this paper.

\begin{Lemma} \label{Lem_rloc_Lipschitz}
For any $t \in I$ the function $r_{\loc}(\cdot, t)$ is $1$-Lipschitz.
More generally, we have the following spacetime bound.
If $P(x',t',r') \subset P(x,t, r_{\loc}(x,t))$, then $r_{\loc}(x',t') \geq r'$.
Note that this bound implies a $C^{\frac12}$-H\"older bound on $r_{\loc}$ in spacetime.
\end{Lemma}

We can now state the definition of an almost regular mean curvature flow.

\begin{Definition} \label{Def_almost_regular}
A weak set flow $\MM \subset \IR^3 \times I$ is called \textbf{almost regular} (or an \textbf{almost regular mean curvature flow}) if the following is true:
\begin{enumerate}[label=(\arabic*)]
\item  \label{Def_almost_regular_11} For almost all $t \in I$ and for any $t \in \partial I \cap I$ we have $\MM_t = \MM_{\reg,t}$, so $\MM_t \subset \IR^3$ is a properly embedded, smooth surface.
We call such times \textbf{regular times,} and all other times \textbf{singular times.}
\item \label{Def_almost_regular_2} For any $[t_1, t_2] \subset  I$ and any bounded subset $U \subset \IR^3$ we have
\[ \int_{t_1}^{t_2} \int_{U \cap \MM_{\reg,t}} r_{\loc}^{-2} \, d\HH^2 \, dt < \infty. \]
\item  \label{Def_almost_regular_3} There is an integer $G$ such that for almost all regular times $t$ the genus of $\MM_t$ is $\leq G$.
We denote the smallest such integer by $\genus (\MM)$.
\item  \label{Def_almost_regular_4} There is a unit-regular Brakke flow $(\mu_t)_{t \in I}$ such that $\mu_t = \HH^2 \lfloor \MM_t$ for almost all $t \in I$ and for all $t \in \partial I \cap I$.
\end{enumerate}
We call $\MM$ \textbf{bounded} if all time-slices $\MM_t \subset \IR^3$, $t \in I$ (and thus all time-slabs $\MM_{[t_1,t_2]}$, $[t_1,t_2] \subset I$) are compact. 
\end{Definition}

The next lemma shows that the Brakke flow from Property~\ref{Def_almost_regular_4} is uniquely determined by the subset $\MM \subset \IR^3 \times I$, so we will often refer to it as the \textbf{Brakke flow associated with} $\MM$.
In fact, the existence of this Brakke flow already follows from Properties~\ref{Def_almost_regular_11}--\ref{Def_almost_regular_3}; the primary point of Property~\ref{Def_almost_regular_4} is its unit-regularity.
The following lemma also shows that the Brakke flow associated with an almost regular flow experiences no mass drop (see \cite{Metzger_Schulze_08,Payne_20} for related results).

We also remark that throughout this paper, we will understand that Gaussian area quantities $\Theta_{(x_0,t_0)}(\tau)$ are taken with respect to the  Brakke flow $(\mu)_{t \in I}$ associated to an almost regular mean curvature flow $\MM$.
In particular, we will often write $\Theta(\MM)$ instead of $\Theta((\mu)_{t \in I})$.

\begin{Lemma} \label{Lem_Brakke_unique}
If $\MM$ is an almost regular flow, then the associated Brakke flow is uniquely determined by $\mu_t = \HH^2 \lfloor \MM_{\reg,t}$ for all $t \in I$.
Moreover, for any test function $u \in C^1_c (\IR^3 \times [t_1,t_2])$, where $[t_1, t_2] \subset \Int I$, the map $t \mapsto \int_{\IR^3} u(\cdot, t) \, d\mu_t$ is continuous and we have equality in the defining property of the Brakke flow:
\begin{equation} \label{eq_eq_Brakke}
 \int_{\IR^3} u(\cdot, t) d\mu_t \bigg|_{t=t_1}^{t=t_2}=\int_{t_1}^{t_2} \int_{\IR^3} \big(\partial_{\mathbf{t}} u + \nabla u \cdot \mathbf{H} - u |\mathbf{H}|^2 \big) d\mu_t dt, 
\end{equation}
\end{Lemma}

\begin{proof}
Set $\hat\mu_t := \HH^2 \lfloor \MM_{\reg,t}$, so $\hat\mu_t = \mu_t$ for all regular times.
We first verify that \eqref{eq_eq_Brakke} holds for $\mu_t$ replaced with $\hat\mu_t$.
Since we can split the integral on the right-hand side into two, we may assume without loss of generality that $t_1$ or $t_2$ is a regular time.
Fix a test function $u \in C^1_c (\IR^3 \times [t_1,t_2])$.
Since by Lemma~\ref{Lem_rloc_Lipschitz}, the functions $r_{\loc}(\cdot, t)$ and $r^2_{\loc}(x, \cdot)$ are $1$-Lipschitz for every fixed $t \in I$ and $x \in \IR^3$, respectively, we can find a smoothing $\td r : \IR^3 \times I \to [0,\infty)$ such that $\frac12 r_{\loc} \leq \td r \leq 2 r_{\loc}$ and $|\nabla \td r|,|\partial_t \td r^2 | \leq 2$.
Let $\psi : [0,\infty) \to [0,1]$ be a cutoff function with $\psi \equiv 0$ on $[0,2]$, $\psi \equiv 1$ on $[4, \infty)$ and $|\psi'| \leq 10$.
Define $\psi_\eps : \IR^3 \times [t_1, t_2]$ by
\[ \psi_\eps (x,t) := \psi (\eps^{-1} \td r(x,t)) \]
Since $\psi_\eps$ vanishes in a neighborhood of $\MM_{\sing}$, we find that for small enough $\eps > 0$
\begin{multline} \label{eq_Rut1t2}
 \int_{\IR^3} (u\psi_\eps) (\cdot, t) d\hat\mu_t \bigg|_{t=t_1}^{t=t_2}
= \int_{t_1}^{t_2} \int_{\IR^3} \big(\partial_{\mathbf{t}} (u\psi_\eps) + \nabla (u\psi_\eps) \cdot \mathbf{H} - (u\psi_\eps) |\mathbf{H}|^2 \big) d\hat\mu_t dt  \\
= \int_{t_1}^{t_2} \int_{\IR^3} \big(\partial_{\mathbf{t}} u + \nabla u \cdot \mathbf{H} - u |\mathbf{H}|^2 \big)\psi_\eps d\hat\mu_t dt +  \int_{t_1}^{t_2} \int_{\IR^3} u \big( \partial_{\mathbf{t}} \psi_\eps + \nabla \psi_\eps \cdot \mathbf{H} \big) d\hat\mu_t dt.
\end{multline}
Note that since $u \psi_\eps$ is $C^1$ with compact support in $\MM_{\reg}$, the integral of $u\psi_\eps$ over a time slice $\MM_{\reg,t}$ is well-defined, and the first equality holds by the Fundamental Theorem of Calculus.
The absolute value of the last integral is bounded by
\begin{align*}
 C \int_{t_1}^{t_2} \int_{\IR^3} & \big( |\partial_{\mathbf{t}} \psi_\eps| + |\nabla \psi_\eps|  |\mathbf{H}| \big) d\hat\mu_t dt \\
&\leq C \int_{t_1}^{t_2} \int_{\IR^3} |\partial_{\mathbf{t}} \psi_\eps| \, d\hat\mu_t dt
+ C \bigg( \int_{t_1}^{t_2} \int_{\IR^3}  |\nabla \psi_\eps|^2  d\hat\mu_t dt \bigg)^{1/2} \bigg( \int_{t_1}^{t_2} \int_{\IR^3} |\mathbf{H}|^2  d\hat\mu_t dt \bigg)^{1/2}  \\
&\leq C  \int_{t_1}^{t_2} \int_{\{ \eps \leq r_{\loc} \leq 8\eps \}}  r_{\loc}^{-2} \,d\hat\mu_t \, dt + C \bigg( \int_{t_1}^{t_2} \int_{\{ \eps \leq r_{\loc} \leq 8\eps \}} r_{\loc}^{-2} d\hat\mu_t \, dt \bigg)^{1/2} \xrightarrow[\eps \to 0 ]{} 0,
\end{align*}
where $C$ is a generic constant.
Moreover, by Property~\ref{Def_almost_regular_2} of Definition~\ref{Def_almost_regular},
\[ \int_{t_1}^{t_2} \int_{\IR^3} \big|\big(\partial_{\mathbf{t}} u + \nabla u \cdot \mathbf{H} - u |\mathbf{H}|^2 \big| d\hat\mu_t dt \leq C \int_{t_1}^{t_2} \int_{\MM_t} \big( 1+r_{\loc}^{-1}+ r_{\loc}^{-2} \big) \, d\HH^2 \, dt < \infty. \]
Combine this with \eqref{eq_Rut1t2} and note that for some $i \in \{ 1,2\}$ the time $t_i$ is regular, so
\[  \int_{\IR^3} (u\psi_\eps) (\cdot, t_i) d\hat\mu_{t_i} \xrightarrow[\eps \to 0]{} \int_{\IR^3} u (\cdot, t_i) d\hat\mu_{t_i}, \]
where the limit is finite. 
It follows that the same convergence holds for both $i =1,2$ and that \eqref{eq_eq_Brakke} holds  in the limit for $\mu_t$ replaced with $\hat\mu_t$.

Comparing \eqref{eq_eq_Brakke} with the definition of a Brakke flow implies that if $t_1$ is regular, then
\begin{align*}
 \int_{\IR^3} u(\cdot, t_2) d\mu_{t_2}
&= \int_{\IR^3} u(\cdot, t) d\mu_{t} \bigg|_{t=t_1}^{t=t_2} + \int_{\IR^3} u(\cdot, t_1) d\mu_{t_1} \\
&\leq \int_{t_1}^{t_2} \int_{\IR^3} \big(\partial_{\mathbf{t}} u + \nabla u \cdot \mathbf{H} - u |\mathbf{H}|^2 \big) d\mu_t dt + \int_{\IR^3} u(\cdot, t_1) d\mu_{t_1} \\
&= \int_{t_1}^{t_2} \int_{\IR^3} \big(\partial_{\mathbf{t}} u + \nabla u \cdot \mathbf{H} - u |\mathbf{H}|^2 \big)  d\hat \mu_t dt + \int_{\IR^3} u(\cdot, t_1) d\hat\mu_{t_1} \\
&=  \int_{\IR^3} u(\cdot, t) d\hat\mu_{t} \bigg|_{t=t_1}^{t=t_2} + \int_{\IR^3} u(\cdot, t_1) d\hat\mu_{t_1}
= \int_{\IR^3} u(\cdot, t_2) d\hat\mu_{t_2}.
\end{align*}
This shows that $\mu_t \leq \hat\mu_t$ for all $t \in I$.
Repeating the same argument, but interchanging the roles of $t_1, t_2$ yields the opposite inequality.
So $\mu_t = \hat\mu_t$, which finishes the proof.
\end{proof}
\bigskip

\subsection{Further properties of mean curvature flows} \label{subsec_furth_props}
Let $\MM \subset \IR^{n+1} \times I$ be a weak set flow.
We will often consider the \textbf{heat operator}
\[ \square = \partial_t - \triangle \]
on the regular part $\MM_{\reg}$.
Here $\partial_t$ denotes the normal time derivative and $\triangle$ the spatial Laplacian.
A function $u \in C^2 (\MM_{\reg})$ satisfying $\square u = 0$ is called \textbf{solution to the heat equation.}
Likewise, if $\square u \leq 0$ or $\square u \geq 0$, then we call $u$ a \textbf{sub-solution} or \textbf{super-solution,} respectively.
The formal conjugate operator is called the \textbf{conjugate heat operator}
\[ \square^* = - \partial_t - \triangle + |\mathbf{H}|^2. \]
As before, a function $v \in C^2 (\MM_{\reg})$ satisfying $\square^* u = 0$ is called \textbf{solution to the conjugate heat equation} and we may also use the terms sub-solution and super-solution if $\square^* v \geq 0$ or $\square^* v \leq 0$, respectively. 
We remark that the conjugacy relation between $\square$ and $\square^*$ implies that whenever $u, v \in C^2_c(\MM_{\reg})$ and $I$ is open, then
\[ \int_I \int_{\MM_{\reg,t}} (\square u) v \, d\HH^2 dt 
= \int_I \int_{\MM_{\reg,t}}  u (\square^* v) \, d\HH^2 dt. \]
Now recall the ambient Gaussian density function from \eqref{eq_def_rho}:
\[ \rho_{(x_0,t_0)} (x,t) := \frac{1}{4\pi(t_0-t)} e^{-\frac{|x-x_0|^2}{4(t_0 -t)}}, \qquad t < t_0  \]
and the Gaussian area quantity $\Theta_{(x_0,t_0)}(\tau)$ from \eqref{eq_def_Theta}.
The following lemma states Huisken's monotonicity formula \cite{Huisken_monotonicity} for the entropy both in global (i.e., integral) and local (i.e., differential) form.

\begin{Lemma} \label{Lem_entropy_rho}
Let $(\mu_t)_{t \in I}$ be a Brakke flow on $\IR^{n+1}$  with $\Theta((\mu_t)_{t \in I}) < \infty$ and denote by $\mathbf{H}$ its variational (mean curvature) vector field.
Fix $(x_0, t_0) \in \IR^{n+1} \times \IR$.
\begin{enumerate}[label=(\alph*)]
\item \label{Lem_entropy_rho_a} 
The Gaussian area $\Theta_{(x_0,t_0)}(\tau)$ is non-decreasing in $\tau$, as long as $\tau > 0$ and $t_0 - \tau \in I$.
\item \label{Lem_entropy_rho_b} In fact for $0 < \tau_1 < \tau_2$ with $[t_0 -\tau_2,t_0-\tau_1] \subset I$ we have
\begin{equation} \label{eq_TT_minus}
 \Theta_{(x_0,t_0)}(\tau_2) - \Theta_{(x_0,t_0)}(\tau_2)
\geq \int_{t_0-\tau_2}^{t_0-\tau_1} \int_{\MM_{\reg,t}}\bigg| \frac{\mathbf{x}^\perp}{2(t_0-\mathbf{t})} + \mathbf{H} \bigg|^2 \rho_{(x_0,t_0)} \, d\mu_t \, dt, 
\end{equation}
where $\mathbf{x}^\perp$ denotes the projection of $\mathbf{x}$ onto the normal bundle of each time-slice $\MM_{\reg,t}$.
\item  \label{Lem_entropy_rho_bb} If $(\mu_t)_{t \in I}$ is the Brakke flow associated with a bounded almost regular mean curvature flow $\MM \subset \IR^3 \times I$, then we even have equality in \eqref{eq_TT_minus}.
\item \label{Lem_entropy_rho_c} Consider a weak set flow $\MM \subset \IR^{n+1} \times I$ and let $\square^* = -\partial_t - \triangle + |\mathbf{H}|^2$ be the conjugate heat operator on $\MM_{\reg}$.
Then we have on $\MM_{\reg, <t_0}$
\[ \square^* \rho_{(x_0,t_0)} = \bigg| \frac{\mathbf{x}^\perp}{2(t_0- \mathbf{t})} + \mathbf{H} \bigg|^2 \rho_{(x_0,t_0)} \geq 0 . \]
\end{enumerate}
\end{Lemma}

\begin{proof}
Assertions~\ref{Lem_entropy_rho_a}, \ref{Lem_entropy_rho_b} are well known; see for example \cite[Corollary~5.13]{Schulze_intro_Brakke}.
Assertion~\ref{Lem_entropy_rho_bb} follows from the same computation combined with the fact that we have equality in the defining property of the Brakke flow by Lemma~\ref{Lem_Brakke_unique}.

To verify Assertion~\ref{Lem_entropy_rho_c}, suppose without loss of generality that $(x_0,t_0) = (\vec 0, 0)$ and write $\rho = \rho_{(x_0,t_0)}$.
Its normal time derivative is
\[ -\partial_t \rho = \bigg({ - \frac{1}{|\mathbf{t}|}  + \frac{|\mathbf{x}|^2}{4\mathbf{t}^2} }\bigg) \rho  -  \mathbf{H} \cdot \nabla \bigg( \frac1{4\pi|\mathbf{t}|} e^{-|\mathbf{x}|^2/4|\mathbf{t}|} \bigg) 
= \bigg({ - \frac{1}{|\mathbf{t}|}  + \frac{|\mathbf{x}|^2}{4\mathbf{t}^2} +  \frac{\mathbf{x} \cdot \mathbf{H}}{2|\mathbf{t}|}} \bigg) \rho,    \]
Write $\mathbf{N}$ for a unit normal vector field of $\MM_{\reg, t}$ and $\mathbf{H} = H \, \mathbf{N}$.
Then the spatial Laplacian of $\rho$ can be computed as
\[ \triangle \rho 
= \triangle_{\IR^3} \rho  - \partial^2_{\mathbf{N}} \rho  + H \, \partial_{\mathbf{N}} \rho 
= \bigg({ - \frac{3}{2|\mathbf{t}|} + \frac{|\mathbf{x}|^2}{4|\mathbf{t}|^2} + \frac1{2|\mathbf{t}|} - \frac{(\mathbf{x} \cdot \mathbf{N})^2}{4|\mathbf{t}|^2} - H \frac{\mathbf{x} \cdot \mathbf{N}}{2|\mathbf{t}|} }\bigg) \rho   . \]
It follows that
\[ \square^* \rho = - \partial_t \rho  - \triangle \rho  + H^2 \rho 
= \bigg( \frac{(\mathbf{x} \cdot \mathbf{N})^2}{4|\mathbf{t}|^2} + H \frac{\mathbf{x} \cdot \mathbf{N}}{|t|} + H^2 \bigg) \rho 
= \bigg( \frac{\mathbf{x} \cdot \mathbf{N}}{2|\mathbf{t}|} + H \bigg)^2 \rho  \geq 0
. \qedhere \]
\end{proof}
\bigskip

The next lemma, which is essentially \cite[Theorem~4]{Ilmanen_sing_MCF_surf} provides an upper bound on the $L^2$-norm of the second fundamental form of a bounded almost regular mean curvature flow in terms of its entropy and genus.

\begin{Lemma} \label{Lem_intA2bound}
Let $\MM \subset \IR^3 \times I$ be a bounded almost regular mean curvature flow and $(x_0, t_0) \in \IR$ and $0 < \tau_1 < \tau_2$ with $\tau_2 \geq 2\tau_1$ and $[t_0 - \tau_2, t_0 - \tau_1] \subset I$.
Suppose that
\[ \Theta_{(x_0,t_0)}(\tau_2) \leq A, \qquad \genus\big(\MM_{[t_0 - \tau_2, t_0 - \tau_1]} \big) \leq A. \]
Then 
\[ \int_{t_0 - \tau_2}^{t_0 - \tau_1} \frac1{t_0 -t} \int_{\MM_{\reg,t} \cap B(x_0, A \sqrt{t_0 - t})} |A|^2 d\HH^2 dt \leq C(A) \log \bigg({\frac{\tau_2}{\tau_1}}\bigg). \]
\end{Lemma}

\begin{proof}
Note that it is enough to verify the bound assuming that $\tau_2 = 2 \tau_1$; in this case $\log(\frac{\tau_2}{\tau_1}) = \log 2$.
Lemma~\ref{Lem_entropy_rho}\ref{Lem_entropy_rho_b} implies  that
\[ \int_{t_0 - \tau_2}^{t_0 - \tau_1} \frac1{t_0 -t} \int_{\MM_{\reg,t} \cap B(x_0, A \sqrt{t_0 - t})} |\mathbf{H}|^2 d\HH^2 dt \leq C(A). \]
Integrating \cite[Theorem~3]{Ilmanen_sing_MCF_surf} and using this bound implies the desired estimate.
Note that $\MM_{\reg,t}= \MM_t$ for almost all $t$.
\end{proof}

The next lemma, which is also essentially due to Ilmanen \cite{Ilmanen_sing_MCF_surf}, concerns limits of bounded almost regular mean curvature flows.
It states that if the mean curvature of such flows converges to zero in the $L^2_{\loc}$-sense, then the limit is given by a smooth stationary minimal surface of possibly higher, but locally constant, multiplicity.
Similarly, if the Gaussian area (based at the origin) of such flows converges to a constant function, then the limit is a smooth shrinker of possibly higher, but constant multiplicity.

We recall here that a shrinker is defined to be a properly embedded smooth hypersurface $\Sigma \subset \IR^{n+1}$ satisfying the equation
\[  \frac{\mathbf{x}^\perp}{2|t|} + \mathbf{H} = 0, \]
where $\mathbf{x}^\perp$ denotes the projection of $\mathbf{x}$ to the normal space of $\Sigma$.
As a result, for any $(x_0, t_0) \in \IR^{n+1} \times \IR$, the set
\[  (x_0,t_0) + \bigcup_{t < 0} \big( |t|^{1/2} \Sigma \big) \times \{ t \}  \subset \IR^{n+1} \times \IR, \]
is a smooth mean curvature flow over $(-\infty, t_0)$, which we will often also refer to as \textbf{shrinker based at $(x_0, t_0)$.}

\begin{Lemma} \label{Lem_Brakke_to_min_shrink}
Consider a sequence of bounded almost regular mean curvature flows $\MM^i$ in $\IR^3 \times I_i$ whose associated Brakke flows weakly converge to a Brakke flow $(\mu_{t}^\infty)_{t \in I_\infty}$ in $\IR^3 \times I_\infty$, where we assume $I_\infty$ to be open.
Suppose that $\Theta (\MM^i)$ and $\genus(\MM^i)$ are both uniformly bounded.
Then the following is true:
\begin{enumerate}[label=(\alph*)]
\item \label{Lem_Brakke_to_min_shrink_a} Suppose that for every bounded open subset $U \subset \IR^3$ and compact interval $[t_1, t_2] \subset I_\infty$ we have
\begin{equation} \label{eq_H2_to0}
 \int_{t_1}^{t_2} \int_{\MM^i_{\reg,t} \cap U} |\mathbf{H}|^2 d\HH^2 \, dt \xrightarrow[i \to \infty]{} 0. 
\end{equation}
Then $(\mu_{t}^\infty)_{t \in I_\infty}$ is stationary.
More specifically, there is a minimal surface $\Sigma \subset \IR^3$ of finite total curvature and a locally constant function $k : \Sigma \to \IN$ such that $d\mu_{\infty,t} = k \, d(\HH^2 \lfloor \Sigma)$ for all $t \in I_\infty$.
\item \label{Lem_Brakke_to_min_shrink_b} Suppose that $\sup I_\infty \leq 0$ and that for some $0 < \tau'_i < \tau''_i$ with $\lim_{i\to\infty} (- \tau'_i) = \sup I_\infty$ and $\lim_{i\to\infty} (- \tau''_i) = \inf I_\infty$ we have
\begin{equation}\label{eqn_gaussian_area_asym_constant}
 \Theta^{\MM^i}_{(\vec 0,0)} (\tau''_i) - \Theta^{\MM^i}_{(\vec 0,0)} (\tau'_i) \xrightarrow[i\to\infty]{} 0. 
 \end{equation}
Then $(\mu_{t}^\infty)_{t \in I_\infty}$ is a smooth shrinker with possibly higher multiplicity.
That is there is a smooth shrinker $\Sigma \subset \IR^3$ and a locally constant function $k : \Sigma \to \IN$ such that $d\mu^\infty_{t} = k_t \, d(\HH^2 \lfloor (|t|^{1/2} \Sigma))$ for all $t \in I_\infty \setminus \{ 0\}$, where $k_t (x) := k(|t|^{-1/2} x)$.
\end{enumerate}
Furthermore, in both cases we have
\begin{equation} \label{eq_Theta_genus_bounded_limit}
 \Theta(\Sigma) \leq \liminf_{i \to \infty} \Theta(\MM^i), \qquad 
\genus(\Sigma) \leq  \liminf_{i \to \infty} \genus(\MM^i). 
\end{equation}
\end{Lemma}

\begin{proof}
We follow the lines of the proof in \cite{Ilmanen_sing_MCF_surf}.
First note that the Brakke flow limit $(\mu_t^\infty)_{t\in I_\infty}$ is stationary in the setting of Assertion~\ref{Lem_Brakke_to_min_shrink_a}  and a (Brakke flow) shrinker in the setting of Assertion~\ref{Lem_Brakke_to_min_shrink_b}; this follows in \ref{Lem_Brakke_to_min_shrink_a} from \eqref{eq_H2_to0}  and the fact that for almost regular mean curvature flows equality holds in the Brakke flow condition \eqref{eq_eq_Brakke}, while  in \ref{Lem_Brakke_to_min_shrink_b} the self-similarity follows as in \cite[Lemma 8]{Ilmanen_sing_MCF_surf}. 

Consider first the setting of Assertion~\ref{Lem_Brakke_to_min_shrink_a}.
The bounds on $\mathbf{H}$ and $|A|^2$ imply that after passing to a subsequence, there is a dense subset $I' \subset I_\infty$  of times $t$ with the following properties: $t$ is a regular time for all $\MM^i$, we have weak convergence $\HH^2 \lfloor \MM^i_t \to \mu^\infty_t$ and for any bounded, open subset $U \subset \IR^3$ we have
\[ \sup_i \int_{\MM^i_t \cap U} |A|^2 d\HH^2 < \infty, \qquad \int_{\MM^i_t \cap U} |\mathbf{H}|^2 \xrightarrow[i\to\infty]{} 0. \]
Then following the arguments in the proof of \cite[Theorem~1]{Ilmanen_sing_MCF_surf}, we may conclude that for any $t \in I'$ we have $d\mu^\infty_t = k_t d( \HH^2 \lfloor \Sigma_t)$ for some smooth, complete, embedded minimal surface $\Sigma_t$ and locally constant function $k_t : \Sigma_t \to \IN$.
As $(\mu_t^\infty)_{t\in I_\infty}$ is stationary, both $k_t$ and $\Sigma_t$ are constant in time.
Moreover, \eqref{eq_Theta_genus_bounded_limit} holds for $\Sigma = \Sigma_t$.
By continuity, we can extend the characterization $d\mu^\infty_t = k_t \, d(\HH^2 \lfloor \Sigma_t)$, $t \in I'$, to all $t \in I_\infty$.

The proof adapts in a straight forward way to the setting of Assertion~\ref{Lem_Brakke_to_min_shrink_b}.
\end{proof}

We also recall the following fact, which implies that we can take the $k$ in Lemma~\ref{Lem_Brakke_to_min_shrink}\ref{Lem_Brakke_to_min_shrink_b} to be a constant.

\begin{Lemma} \label{Lem_shrinker_connected}
Every shrinker $\Sigma \subset \IR^3$ is connected.
\end{Lemma}

\begin{proof}
See \cite[Theorem~3.11]{Choi_Haslhofer_Hershkovits_White_22}, or the arguments of \cite{Brendle_genus0_16}.
\end{proof}

\bigskip

\subsection{Minimal surfaces}
Given a properly embedded minimal surface $\Sigma \subset \IR^3$, we denote by 
\[ \Theta(\Sigma) := \lim_{r \to \infty} r^{-2} \HH^2 (\Sigma \cap B(\vec 0, r) ) \]
its asymptotic area ratio.
Note that this is the same as $\Theta(\MM)$ for the corresponding stationary flow $\MM := \Sigma \times \IR$.
We recall that the condition $\Theta(\Sigma), \genus(\Sigma) < \infty$ is equivalent to the fact that $\Sigma$ has finitely many components and finite total curvature, i.e., $\int_{\Sigma} |K_{\Sigma}| d\HH^2 < \infty$, where $K_\Sigma$ denotes the intrinsic sectional curvature \cite{Osserman_64,White_87}, \cite[Theorems~8.1, 8.3]{Chodosh_Mantoulidis_MS_notes}.
The following lemma discusses further properties of such surfaces. 
Amongst other things, it characterizes the asymptotic geometry of their ends.
In the following we define $\mathbf{r}(x):=|x|$ for $x \in \IR^3$ to be the radial function.

\begin{Lemma} \label{Lem_ends_min_surf}
Let $\Sigma \subset \IR^3$ be a properly embedded minimal surface with $\Theta(\Sigma), \genus(\Sigma) < \infty$.
Then
\begin{enumerate}[label=(\alph*)]
\item \label{Lem_ends_min_surf_aa1} $\Sigma$ is either connected or a union of finitely many parallel affine planes.
\item \label{Lem_ends_min_surf_aa2} If $\Sigma$ is not a union of affine planes, then $\int_{\Sigma} |A|^2 d\HH^2 \geq 2\pi$.
\end{enumerate}
Moreover, after possibly applying a rotation around the origin, there are pairwise disjoint open subsets $\Sigma_1, \ldots, \Sigma_k \subset \Sigma$ such that $\Sigma \setminus (\Sigma_1 \cup \ldots \cup \Sigma_k)$ is compact and such that for some $R > 0$:
\begin{enumerate}[label=(\alph*), start=3]
\item \label{Lem_ends_min_surf_a} Each $\Sigma_i$ is the graph of a smooth function $u_i \in C^\infty( \IR^2 \setminus \ov{B(\vec 0, R)})$.
\item \label{Lem_ends_min_surf_b} We have $u_1 < \ldots < u_k$.
\item \label{Lem_ends_min_surf_c} For any $j = 1, \ldots, k$ and $m \geq 0$ we have the asymptotics $\mathbf{r}^{m-1} |\nabla^m u_j| \to 0$ as $\mathbf{r} \to \infty$.
\item \label{Lem_ends_min_surf_d} For any $j = 1, \ldots, k-1$ we have the following asymptotics as $\mathbf{r} \to \infty$
\[ \mathbf{r}  \, \frac{|\nabla (u_{j+1} - u_j)|}{u_{j+1}-u_j} \lto 0, \qquad \mathbf{r}  \, (u_{j+1} - u_j) \lto \infty \]
\end{enumerate}
\end{Lemma}

\begin{proof}
Assertions~\ref{Lem_ends_min_surf_aa1} and \ref{Lem_ends_min_surf_aa2} are proven in \cite{Hoffman_Meeks_90, Meeks_Simon_Yau_82} and \cite[Theorems~8.1, 8.3]{Chodosh_Mantoulidis_MS_notes}.
Assertions~\ref{Lem_ends_min_surf_a}--\ref{Lem_ends_min_surf_c} are well known; for example they follow from the fact that the Gauss maps extends continuously to the conformal compactification of $\Sigma$ (see again \cite[Theorems~8.3]{Chodosh_Mantoulidis_MS_notes}) together with standard higher derivative estimates. 

To see Assertion~\ref{Lem_ends_min_surf_d} fix some $j = 1, \ldots, k-1$, set $v := u_{j+1} - u_j$ and consider the minimal surface equations for $\Sigma_{j}, \Sigma_{j+1}$ in terms of $u_j, u_j + v$.
\[ \DIV \bigg({ \frac{\nabla u_j}{\sqrt{1+|\nabla u_j|^2}} }\bigg) = \DIV \bigg({ \frac{\nabla (u_j+v)}{\sqrt{1+|\nabla (u_j+v)|^2}} }\bigg) = 0. \]
Our goal will be to show that $\mathbf{r} |\nabla \log v| \to 0$ as $\mathbf{r} \to \infty$.
We will achieve this by a blow-up argument, which allow us to linearize the equation for $v$.
For this purpose let us define the non-linear operators for $a > 0$ on smooth vector-valued functions $(u', v') \in C^\infty(U; \IR^2)$ on open subsets $U \subset \IR^2$:
\[ F_a (u', v') := \bigg( \DIV \bigg({ \frac{\nabla u'}{\sqrt{1+|\nabla u'|^2}} }\bigg), \frac1{a} \DIV \bigg({ \frac{\nabla (u'+a v')}{\sqrt{1+|\nabla (u'+a v')|^2}} - \frac{\nabla u'}{\sqrt{1+|\nabla u'|^2}} }\bigg) \bigg).  \]
So $F_{a}(u_j, a^{-1} v) = 0$ for any $a > 0$.
We have the following compactness property:

\begin{Claim} \label{Cl_cptness}
For fixed $U \subset \IR^2$ consider a sequence of solutions $(u''_i,v''_i) \in C^\infty(U; \IR^2)$ of $F_{a_i}(u''_i, v''_i) = (0,0)$, where the nunbers $a_i > 0$ are uniformly bounded from above.
Suppose that both $u''_i$ and $v''_i$ are uniformly bounded in $C^1_{\loc}$.
Then, after passing to a subsequence, we have $C^\infty_{\loc}$-convergence $u''_i \to u''_\infty$ and $v''_i \to v''_\infty$, where $(u''_\infty, v''_\infty) \in C^\infty_{\loc}(U; \IR^2)$.
If $a_\infty := \lim_{i \to \infty} a_i > 0$, then $F_{a_\infty} (u''_\infty, v''_\infty) = 0$.
If $a_\infty = 0$ and $u''_\infty \equiv const$, then $\triangle v''_\infty = 0$.
\end{Claim}

\begin{proof}
This follows from the fact that the family of operators $F_a$ satisfies uniform Schauder estimates, as long as $a  > 0$ is bounded.
The last statement follows from passing the identity $F_{a_i}(u''_i, v''_i) = (0,0)$ to the limit.
\end{proof}

As a first step we prove the following uniform bound:

\begin{Claim} \label{Cl_bound_log}
We have $\mathbf{r} |\nabla \log v| \leq C$ for some uniform $C < \infty$.
\end{Claim}

\begin{proof}
Suppose by contradiction that there is a sequence $z_i \in \IR^2$ with $z_i, (\mathbf{r} |\nabla \log v|)(z_i) \to \infty$.
We will use a point picking argument to find a similar sequence with better properties.
Fix some large $i$ and set $r_i := \frac1{10} |z_i|$.
We apply the following algorithm.
Let $z'_i := z_i$ and $r'_i := r_i$.
If $|\nabla \log v| \leq 2|\nabla \log v|(z'_i)$ on $B(z'_i, r'_i)$, then we stop.
Otherwise, we replace $z'_i$ with a point $z''_i \in B(z'_i, r'_i)$ such that $|\nabla \log v|(z''_i) > 2|\nabla \log v|(z'_i)$ and we replace $r'_i$ with $r''_i := \frac12 r'_i$.
The algorithm has to terminate at a finite number of steps since it produces a sequence of points $z'_i$ with consecutive distances $r_i$, $\frac12 r_i$, \ldots and the values of $|\nabla \log v|$ at these points diverge.
After termination, we obtain a point $z'_i \in B(z_i, 2 r_i)$ and a number $r'_i \leq r_i = \frac1{10} |z_i| \leq 2 |z'_i|$ such that
\begin{equation} \label{eq_pp_res_1}
 \sup_{B(z'_i, r'_i)} |\nabla \log v| \leq 2|\nabla \log v|(z'_i) 
\end{equation}
and
\begin{equation} \label{eq_pp_res_2}
 r'_i |\nabla \log v|(z'_i) \geq r_i |\nabla \log v|(z_i) \xrightarrow[i \to \infty]{} \infty. 
\end{equation}
From now on we will work with the sequence of points $z'_i \to \infty$ and radii $r'_i$. 

Set $\lambda_i := |\nabla \log v|(z'_i)$  and consider the rescalings
\[  u'_i (x) := \lambda_i u_j ( \lambda_i^{-1} x + z'_i), \qquad v'_i (x) := \lambda_i v ( \lambda_i^{-1} x + z'_i). \]
Note that the graphs of $u'_i$ and $u'_i + v'_i$ are the shifted and rescaled minimal surfaces $\lambda_i (\Sigma_j - z'_i), \lambda_i (\Sigma_{j+1} - z'_i)$.
Due to the choice of $\lambda_i$ we have
\begin{equation} \label{eq_nabvpi_is_1}
 |\nabla \log v'_i |(\vec 0) = 1 
\end{equation}
and by \eqref{eq_pp_res_1} and \eqref{eq_pp_res_2} we have $\lambda_i r'_i \to \infty$ and
\[ |\nabla \log v'_i | \leq 2 \textQQq{on} B(\vec 0, \lambda_i r'_i). \]
Moreover, since $\mathbf{r} \geq \frac12 |z_i|$ on  $B(z'_i, r'_i)$, we also have $\nabla^m u'_i \to 0$ in $C^0_{\loc}$ for any $m \geq 1$ due to Assertion~\ref{Lem_ends_min_surf_c}.
By the same assertion we also have
\[ a_i := v'_i (\vec 0) = \lambda_i v(z'_i) = |\nabla v|(z'_i) \xrightarrow[i \to \infty]{} 0. \]
So if we set $u''_i := u'_i - u'_i (\vec 0)$ and $v''_i := a_i^{-1} v'_i$, then we can use Claim~\ref{Cl_cptness} to conclude that for a subsequence
\[ (u''_i, v''_i) \xrightarrow[i\to\infty]{C^\infty_{\loc}} (0, v''_\infty), \]
where $v'' \in C^\infty(\IR^2)$ with $\triangle v''_\infty = 0$.
The fact that $v \geq 0$ implies that $v''_\infty \geq 0$.
Since $v''_\infty$ is harmonic and defined on $\IR^2 = \IC$, we can find a holomorphic function $f  \in C^\infty(\IC)$ such that $v''_\infty = \Re (f) \geq 0$.
But this implies that $f$ and therefore $v''_\infty$ is constant, in contradiction to \eqref{eq_nabvpi_is_1}.
\end{proof}

To finish the proof of Assertion~\ref{Lem_ends_min_surf_d}, suppose by contradiction that for a sequence $z_i \in \IR^2$ with $z_i \to \infty$ we have
\begin{equation} \label{eq_lower_rlog}
 \big( \mathbf{r} |\nabla \log v| \big)(z_i) > c > 0. 
\end{equation}
Let $\lambda_i := |z_i|^{-1} \to 0$ and consider  the rescalings
\[ u'_i (x) := \lambda_i u_j(\lambda_i^{-1} x), \qquad
 v'_i (x) := \lambda_i v(\lambda_i^{-1} x). \]
Note again that the graphs of $u'_i$ and $u'_i + v'_i$ are the rescaled minimal surfaces $\lambda_i \Sigma_j, \lambda_i \Sigma_{j+1}$.
By Assertion~\ref{Lem_ends_min_surf_c} we have $u'_i \to 0$ in $C^\infty_{\loc}$ on $\IR^2 \setminus \{ \vec 0 \}$.
As the bound from Claim~\ref{Cl_bound_log} is scaling invariant, we still have $\mathbf{r} |\nabla \log v'_i| \leq C$ for the same $C$.
Fix some $z_0 \in \IR^2 \setminus \{ \vec 0 \}$ and note that by Assertion~\ref{Lem_ends_min_surf_c}
\[ a_i := v'_i (z_0) = \lambda_i v(\lambda_i^{-1} z_0) \xrightarrow[i\to\infty]{} 0. \]
It follows that we still have $\mathbf{r} |\nabla \log (a_i^{-1} v'_i)| \leq C$.
So by Claim~\ref{Cl_bound_log} we obtain, after passing to a subsequence,
\[ (u'_i, a_i^{-1} v'_i) \xrightarrow[i\to\infty]{C^\infty_{\loc}} (0, v''_\infty) \]
on $\IR^2 \setminus \{ \vec 0 \}$, where $\triangle v''_\infty = 0$.
Due to \eqref{eq_lower_rlog} we must have $v''_i \not\equiv const$ and since $v \geq 0$, we have $v''_\infty \geq 0$.
Since $v''_\infty$ is harmonic on $\IR^2 \setminus \{ \vec 0 \} = \IC \setminus \{ 0 \}$, we can find a holomorphic function $f \in C^\infty(\IC)$ such that $v''_\infty (e^{ix}) = \Re ( f(x) ) \geq 0$.
This implies again that $f$ and thus $v''_\infty$ is constant, giving us the desired contradiction.
This implies the first statement of Assertion~\ref{Lem_ends_min_surf_d}.
The second statement follows by radial integration of the first statement.
\end{proof}
\bigskip

The next lemma is a compactness theorem for minimal surfaces with uniformly bounded genus and asymptotic area ratio, assuming only uniform control near the origin.
It states that such surfaces either subsequentially converge in the local smooth topology or the component passing through the origin  converges to a plane in a local intrinsic sense.

\begin{Lemma} \label{Lem_min_surf_conv_or_plane}
Consider a sequence of properly embedded minimal surfaces $\Sigma_i \subset \IR^3$ containing the origin for which $\Theta(\Sigma_i)$ and $\genus(\Sigma_i)$ are uniformly bounded.
Suppose that in a neighborhood of the origin the second fundamental form and normal injectivity radius of $\Sigma_i$ are uniformly bounded from above and below, respectively.
Then after passing to a subsequence, one of the following is true:
\begin{enumerate}[label=(\alph*)]
\item \label{Lem_min_surf_conv_or_plane_a} After identifying the tangent spaces $T_{\vec 0} \Sigma_i$ at the origin with $\IR^2$, the intrinsic exponential maps at the origin converge in $C^\infty_{\loc}$ to the exponential map of a plane passing through the origin. 
\item \label{Lem_min_surf_conv_or_plane_b} We have local smooth convergence $\Sigma_i \to \Sigma_\infty$, where $\Sigma_\infty \subset \IR^3$ is a connected, properly embedded minimal surface with
\[ \Theta (\Sigma_\infty) \leq \liminf_{i \to \infty} \Theta(\Sigma_i), \qquad 
\genus( \Sigma_\infty) \leq \liminf_{i \to \infty} \genus(\Sigma_i). \]
\end{enumerate}
\end{Lemma}

\begin{proof}
By Lemma~\ref{Lem_Brakke_to_min_shrink} (applied to the associated stationary flows $\MM^i = \Sigma_i \times \IR$), we can pass to a subsequence of the $\Sigma_i$, which converges to a minimal integer varifold whose support is a smooth, complete minimal surface $\Sigma_\infty \subset \IR^3$ and whose multiplicity is locally constant.
Due to the uniform bounds at the origin, the component $\Sigma'_\infty \subset \Sigma_\infty$ containing the origin must have multiplicity 1.
If $\Sigma'_\infty \subsetneq \Sigma_\infty$, then by Lemma~\ref{Lem_ends_min_surf}\ref{Lem_ends_min_surf_aa1} the surface $\Sigma_\infty$ must be a union of parallel affine planes, which implies Assertion~\ref{Lem_min_surf_conv_or_plane_a}.
If $\Sigma'_\infty = \Sigma_\infty$, then Assertion~\ref{Lem_min_surf_conv_or_plane_b} holds.
\end{proof}
\bigskip

\subsection{Analytic preliminaries}
The following lemma shows that any viscosity supersolution to a linear parabolic equation is also a supersolution in the distributional sense.
We remark that the asserted bound \eqref{eq_Lstarleq} follows directly from the identity \eqref{eq_def_Lstar} if $u$ is of regularity $C^2$.

\begin{Lemma} \label{Lem_conj_L}
Let $T_1 < T_2$ and $U \subset \IR^n$ be an open subset and consider elliptic operators
\[ L = a_{ij} \partial_i \partial_j + b_i \partial_i + c, \qquad L^* = a_{ij}^* \partial_i \partial_j + b_i^* \partial_i + c^* \]
on $U \times [T_1,T_2]$, where $a_{ij},a^*_{ij}, b_i, b^*_i, c, c^* \in C^\infty (U \times [T_1, T_2])$ such that $(a_{ij}), (a^*_{ij})$ are symmetric and positive definite matrices at every point and time.
Suppose that the associated parabolic operators $\partial_t - L$ and $-\partial_t - L^*$ are formally adjoint with respect to a weight function $\varphi \in C^\infty (U \times [T_1, T_2])$, $\varphi > 0$, in the sense that for any compactly supported $u', v' \in C^2_c (U \times (T_1, T_2))$ we have
\begin{equation} \label{eq_def_Lstar}
 \int_{T_1}^{T_2} \int_{U} \big( (\partial_t - L)u'\big) (x,t) v' (x,t) \varphi(x,t) \, dx\, dt 
= \int_{T_1}^{T_2} \int_{U} u'(x,t)  \big( (-\partial_t - L^*) v' \big) (x,t) \varphi(x,t) \, dx\, dt. 
\end{equation}
Consider a function $u : U \times [T_1, T_2] \to \IR$ with the following properties:
\begin{enumerate}[label=(\roman*)]
\item \label{Lem_conj_L_i} $|u| \leq C$ for some uniform $C < \infty$.
\item \label{Lem_conj_L_ii} For all $t \in [T_1,T_2]$ the function $u(\cdot, t)$ is Lipschitz for some uniform Lipschitz constant.
\item $u$ is lower semi-continuous on $U \times [T_1,T_2]$
\item \label{Lem_conj_L_iv} $u$ is continuous from the left in time, in the sense that whenever $(x_i, t_i) \to (x,t)$ with $t_i \leq t$, then $u(x,t) = \lim_{i \to \infty} u(x_i,t_i)$.
\item We have $(\partial_t - L) u \geq 0$ in the viscosity sense, meaning that whenever we have $u \geq \ov u$ for some $\ov u \in C^\infty(U \times [T_1, t_0])$, $t_0 \in (T_1, T_2]$ with equality at some point $(x_0, t_0)$, then $((\partial_t - L) \ov u) (x_0, t_0) \geq 0$.
\end{enumerate}
Then for any compactly supported $v \in C^2_{c} (U\times [T_1, T_2])$ with $v \geq 0$ we have
\begin{equation} \label{eq_Lstarleq}
 \int_{T_1}^{T_2} \int_{U} u(x,t) \big( (-\partial_t - L^*) v\big) (x,t) \varphi(x,t) \, dx \, dt \geq - \int_{U} u (x,t) v(x,t) \varphi(x,t) \, dx \bigg|_{t=T_1}^{t=T_2}. 
\end{equation}
\end{Lemma}

\begin{proof}
Without loss of generality, we may assume that $[T_1, T_2] = [0,T]$.
Choose compact domains $\Omega' \subset \Omega \subset U$ with smooth boundary such that $\supp v \subset \Int \Omega' \times [T_1,T_2]$ and $\Omega' \subset \Int \Omega$.
We need the following maximum principle.

\begin{Claim} \label{Cl_max_princ}
Let $[t_0, t_1] \subset [0,T]$ and consider a solution $$\td u \in C^0(\Omega \times [t_0,t_1]) \cap C^2(\Int \Omega \times (t_0, t_1))$$ of $(\partial_t - L) \td u = 0$ with $u \geq \td u$ on $(\Omega \times \{ t_0 \} ) \cup ( \partial \Omega \times [t_0, t_1])$.
Then $u \geq \td u$ on $\Omega \times [t_0, t_1]$.
\end{Claim}

\begin{proof}
For $\eps > 0$ consider the function $\td u_\eps (x,t) := \td u(x,t) - \eps (t-t_0) - \eps$.
Then $u > \td u_\eps $ on $(\Omega \times \{ t_0 \} ) \cup ( \partial \Omega \times [t_0, t_1])$.
Choose $t^* \in [t_0, t_1]$ maximal with the property that $u \geq \td u_\eps$ on $\Omega \times [t_0 , t^*)$.
By the continuity from the left in time, we even have $u \geq \td u_\eps$ on $\Omega \times [t_0 , t^*]$ and by the lower semi-continuity, we must have equality at some point if $t^* < t_1$.
However, this would imply that $-\eps = (\partial_t - L) \td u_\eps \geq 0$ at this point.
So $u \geq \td u_\eps$ on $\Omega \times [t_0, t_1]$ and the claim follows by taking the limit $\eps \to 0$.
\end{proof}

As an application of this claim we obtain:

\begin{Claim} \label{Cl_find_tdu}
There is a function $\Psi : \IR_+ \to \IR_+$ with $\lim_{\eps \to 0} \Psi(\eps)= 0$ such that for any $[t_0, t_1] \subset [T_1,T_2]$ we can find a solution  $\td u \in C^0(\Omega \times [t_0,t_1]) \cap C^2(\Int \Omega \times (t_0, t_1))$ of $(\partial_t - L) \td u = 0$ such that 
\begin{enumerate}[label=(\alph*)]
\item $u(\cdot, t_0) = \td u (\cdot, t_0)$ on $\Omega'$.
\item $u \geq \td u$ on $\Omega \times [t_0, t_1]$.
\item $|\td u(x,t) - u(x, t_0) | \leq \Psi (t_1 -t_0)$ for all $(x,t) \in \Omega' \times [t_0 , t_1]$.
\end{enumerate}
\end{Claim}

\begin{proof}
Choose a continuous function $F : \Omega \times [-C,\infty) \to \IR$, where $C$ is the constant from Assumption~\ref{Lem_conj_L_i}, such that:
\begin{enumerate}[label=(\arabic*)]
\item $F (x,s) = s$ whenever $x$ is in a neighborhood of $\Omega'$.
\item $F(x,s) \leq s$ for all $(x,s) \in \Omega \times \IR$.
\item $F(x,s) = -C$ whenever $x \in \partial\Omega$.
\end{enumerate}
Let $\td u$ be the unique solution of $(\partial_t - L) \td u = 0$ with the boundary condition
\[ \td u(x,t) = F(x,u(x,t)) \textQQq{if} (x,t) \in (\Omega \times \{ t_0 \} ) \cup ( \partial \Omega \times [t_0, t_1]). \]
Then the bound $u \geq \td u$ follows from Claim~\ref{Cl_max_princ} and the bound on $|\td u(x,t) - u(x, t_0) |$ follows from the fact that $\td u(\cdot, t_0) = u(\cdot, t_0)$ in a neighborhood of $\Omega'$ and Assumptions~\ref{Lem_conj_L_i}, \ref{Lem_conj_L_ii}, for example via a blow-up argument.
\end{proof}

Let now $k \geq 0$ be a large number and set $t_i := \frac{i}k T$.
For each $i = 0, \ldots, k-1$ consider the solution $\td u_i \in C^0 (\Omega \times [t_i, t_{i+1}]) \cap C^2 (\Int \Omega \times (t_i, t_{i+1}))$ from Claim~\ref{Cl_find_tdu}, for $[t_0, t_1]$ replaced with $[t_i, t_{i+1}]$.
Choose $C' < \infty$ such that $|(-\partial_t - L^*) v| \leq C'$ and define the function $u^{(k)} : \Omega \times [0,T) \to \IR$ by
\[ u^{(k)} (x, t) := u(x, t_i) \qquad \text{if} \quad (x,t) \in [t_i, t_{i+1}). \]
By the fact that $u$ is measurable and by Assumptions~\ref{Lem_conj_L_i}, \ref{Lem_conj_L_ii}, \ref{Lem_conj_L_iv}, we find that $u^{(k)} \to u$ in $L^2 (\Omega \times [0,T))$ as $k \to \infty$ due to dominated convergence.
Therefore, omitting the arguments $(x,t)$ for clarity, we have
\begin{multline} \label{eq_approx_steps}
 \bigg| \int_{0}^{T} \int_{U} u  \big( (-\partial_t - L^*) v \big) \varphi \, dx \, dt
- \sum_{i=0}^{k-1} \int_{t_i}^{t_{i+1}} \int_{\Omega} \td u_i \big( (-\partial_t - L^*) v \big) \varphi \, dx \, dt \bigg| \\ \displaybreak[1] 
= \bigg| \int_{0}^{T} \int_{U} \big( u  - u^{(k)} \big) \big( (-\partial_t - L^*) v \big) \varphi \, dx \, dt
- \sum_{i=0}^{k-1} \int_{t_i}^{t_{i+1}} \int_{\Omega} \big( \td u_i - u(\cdot, t_i) \big) \big( (-\partial_t - L^*) v \big) \varphi \, dx \, dt \bigg| \\
\leq \bigg| \int_{0}^{T} \int_{U} \big( u  - u^{(k)} \big) \big( (-\partial_t - L^*) v \big) \varphi \, dx \, dt \bigg| + \int_0^T \int_\Omega C' \Psi(\tfrac1k T) \, dx \, dt \xrightarrow[k\to\infty]{} 0.   
\end{multline}
Since $\td u_i$ is $C^2$ on $(\Int \Omega) \times (t_0, t_1)$, we obtain, using the fact that $v \geq 0$ and $\td u_{i} (\cdot, t_i) = u(\cdot, t_i) \geq \td u_{i-1} (\cdot, t_i)$ and $u (\cdot, T) \geq \td u_{k-1} (\cdot ,T)$  on $\Omega'$,
\begin{align*}
 \sum_{i=0}^{k-1} \int_{t_i}^{t_{i+1}} \int_{\Omega} & \td u_i \big( (-\partial_t - L^*) v \big) \varphi \, dx \, dt \\
&= \sum_{i=0}^{k-1} \bigg( \int_{t_i}^{t_{i+1}}  \int_{\Omega} \big(( \partial_t - L)\td u_i \big) v \varphi \, dx \, dt - \int_{\Omega} \td u_i (x, t) v(x, t) dx  \bigg|_{t=t_i}^{t=t_{i+1}} \bigg) \\
&= - \int_{\Omega} \td u_{k-1} (x,T) v(x,T) \varphi(x,T) \, dx + \int_{\Omega} u (x,t) v(x,0) \varphi(x,0) \, dx \\
&\qquad\qquad + \sum_{i=1}^{k-1} \int_{\Omega} \big(\td u_{i} (x,t_i) - \td u_{i-1} (x,t_i)\big) v(x,t_i) \varphi(x,t_i) \, dx \\
&\geq - \int_{\Omega} u (x,t) v(x,t) \varphi(x,t) \, dx \bigg|_{t=0}^{t=T}.
\end{align*}
Combining this with \eqref{eq_approx_steps} yields the desired result.
\end{proof}
\bigskip

\section{Evolution of the separation function} \label{sec_sep_fct}
Let $\MM \subset \IR^{n+1} \times I$ be a (not necessarily maximal) weak set flow.
In this section, we define a separation function, which is roughly comparable to twice the normal injectivity radius.
This separation function and its evolution inequality \eqref{eq_sq_log_s} will play a key role in establishing a separation bound between neighboring sheets in a mean curvature flow.

\begin{Definition}
For any $(x,t) \in \MM$ we define the \textbf{separation} $\s(x,t) \in [0, \infty]$ to be the infimum over all $r > 0$ such that $\MM_t \cap B(x,r)$ is not connected.
The function $\s : \MM \to [0, \infty]$ is called the \textbf{separation function.}
We call $(x,t) \in \MM$ \textbf{good} if the following holds:
\begin{enumerate}
\item $0 < r:=\s(x,t) < \infty$.
\item The closure of $\MM_t \cap B(x,r)$ is contained in $\MM_{\reg,t}$.
\item For any $r' \in (0,r]$ the distance sphere $S(x,r')$ is transverse to $\MM_{\reg,t}$ at any point that belongs to the closure of $\MM_t \cap B(x,r)$.
\end{enumerate}
We denote the subset of good points by $\MM_{\good} \subset \MM$.
\end{Definition}

Note that by Definition~\ref{Def_rloc}\ref{Def_rloc_4} we have $\s \geq r_{\loc}$.

To elucidate the goodness property, observe that if $(x,t) \in \MM_{\good}$ and $r:=\s(x,t) < \infty$, then the intersection  $\MM_t \cap \ov{B(x,r)}$ is disconnected.
The component of this intersection that contains $x$ is the closure of $\MM_t \cap B(x,r)$.
It is a smooth hypersurface  whose boundary lies in $S(x,r)=\partial B(x,r)$ and meets $S(x,r)$ transversally.
All other components lie in $S(x,r)$.
Moreover, the intersections $\MM_t \cap B(x,r')$ for $0 < r' < r$ are all connected.
On the other hand, if $(x,t) \not\in \MM_{\good}$ and $r:=\s(x,t) < \infty$, then $\MM_t \cap \ov{B(x,r)}$ may be connected, but tangent to $S(x,r)$.

Let us now discuss basic regularity properties of the separation function (restricted to $\MM_{\good}$).

\begin{Lemma} \label{Prop_sep}
The following is true:
\begin{enumerate}[label=(\alph*)]
\item \label{Prop_sep_a} $\MM_{\good} \subset \MM$ is backwards  open, meaning that any $(x,t) \in \MM_{\good}$ has a small backwards parabolic neighborhood $P^- (x,t,\eps) = B(x,\eps) \times [t- \eps^2,t]$, which is contained in $\MM_{\good}$.
In particular, all time-slices $\MM_{\good,t} \subset \MM_t$ are open.
\item \label{Prop_sep_bb} $\s$ restricted to $\MM_{\good}$ is lower semi-continuous.
\item \label{Prop_sep_b} $\s$  restricted to $\MM_{\good}$  is continuous from the left in time, in the sense that for any $(x_i,t_i) \to (x,t) \in \MM_{\good}$ with $t_i \leq t$ we have $\s(x_i,t_i) \to \s(x,t)$.
\item \label{Prop_sep_c} $\s(\cdot, t)$ restricted to each time-slice $\MM_{\good,t}$, $t \in I$, is  locally 1-Lipschitz.
\end{enumerate}
\end{Lemma}

\begin{proof}
To prove Assertions~\ref{Prop_sep_a}--\ref{Prop_sep_c} fix $(x_i, t_i) \to (x,t) \in \MM_{\good}$ and set $r := \s(x,t)$.
Choose an open neighborhood $U \subset \MM_{\reg}$ of the closure of $\MM_t \cap B(x,r)$ that is relatively compact in $\MM_{\reg}$ and has the following property for all $r''  \in (0,r+\eps)$ for some $\eps > 0$: 
\begin{enumerate}
\item The intersection $U \cap \MM_t \cap B(x,r'')$ is connected and its closure is contained in $U$.
\item The distance spheres $S(x,r'')$ is transverse to $U \cap \MM_t$.
\end{enumerate}

Let us first prove Assertion~\ref{Prop_sep_bb}.
Fix some $r' \in (0 , r)$.
By the openness of the transversality condition, we obtain that for large $i$ the set $U \cap \MM_{t_i} \cap B(x_i, r'')$ is connected and its closure is contained in $U$ for any $r'' \in (0,r']$.
So to prove Assertion~\ref{Prop_sep_bb}, it suffices to verify that $U \cap \MM_{t_i} \cap B(x_i, r') = \MM_{t_i} \cap B(x_i, r')$ for large $i$.
Suppose by contradiction that, after passing to a subsequence, we can find points $y_i \in \MM_{t_i} \cap B(x_i, r')$, which are not contained in $U$.
After passing to another subsequence, we have $y_i \to y_\infty \in \MM_t \cap B(x,r) \subset U$, in contradiction to the openness of $U$.

Let us now prove Assertion~\ref{Prop_sep_b}; so assume that $t_i \leq t$.
Due to the lower semi-continuity it suffices to assume by contradiction that $r_\infty := \lim_{i \to \infty} \s(x_i,t_i) > \s(x,t)$.
Let $r' \in  (r, r_\infty)$ such that $r' < r+\eps$.
By openness of the transversality condition, it follows that for large $i$ the connected component of $\MM_{t_i} \cap B(x_i, r')$ containing $x_i$ is equal to $U \cap \MM_{t_i} \cap B(x_i,r')$.
On the other hand, since $\s(x_i,t_i) > r'$ for large $i$, we find that $\MM_{t_i} \cap B(x_i,r')$ is connected, so it must be contained in $U$.
Since $r' > \s(x,t)$, we can find a point $z \in \MM_t \cap B(x,r')$ with $z \not\in U$.
We claim that there is a sequence of points $z_i \in \MM_{t_i}$ with $z_i \to z$.
If not, then we could find a $\delta > 0$ such that $B(z,\delta) \cap \MM_{t_i} = \emptyset$ for some large $i$, which contradicts $z \in \MM_t$ due to the fact that $\MM$ is a weak set flow.
By our previous conclusions, we conclude that for large $i$ we have $z_i \in \MM_{t_i} \cap B(x_i, r') \subset U$, in contradiction to the openness of $U$.

To see Assertion~\ref{Prop_sep_a}, suppose again that $t_i \leq t$ and note that Assertion~\ref{Prop_sep_b} implies $\s(x_i, t_i) \to \s(x,t) = r$.
So due to the openness of the transversality condition we have $(x_i,t_i) \in \MM_{\good}$ for large $i$.

Lastly, we need to verify Assertion~\ref{Prop_sep_c}; so assume that $t_i = t$.
By the openness of the transversality condition, we obtain that for large enough $i$ the intersections $U \cap \MM_t \cap B(x_i, r'')$ are connected for all $r'' \in (0, r+ \frac12 \eps)$.
Thus, for large $i$ we have $\s(x_i,t) = d( x_i, \MM_t \setminus U )$,
which immediately implies that $|\s(x_i,t)- \s(x,t)| \leq |x_i - x|$.
\end{proof}
\medskip

We will now state the main estimate involving the separation function.
Recall that $\square = \partial_t - \triangle$ denotes the heat operator on $\MM_{\reg}$, where $\partial_t$ denotes the normal time derivative and $\triangle$ the spatial Laplacian.

\begin{Proposition} \label{Prop_log_u_super_sol}
If $\MM$ is a weak set flow, then on $\MM_{\good}$ we have in the viscosity sense
\begin{equation} \label{eq_sq_log_s}
 \square \log \s \geq 0.
\end{equation}
\end{Proposition}

We recall that this means that for any $(x,t) \in \MM_{\good}$ the following is true. Any smooth function that is locally defined in a neighborhood of $(x,t)$ in $\MM_{\reg, \leq t}$ and satisfies $v \leq \log \s$ with equality at $(x,t)$ must satisfy $\square v \geq 0$.
To prove Proposition~\ref{Prop_log_u_super_sol}, we need the following lemma.

\begin{Lemma} \label{Lem_Sigma_compare}
Let $\MM \subset \IR^{n+1} \times (-T,0]$, $T > 0$, be a weak set flow and consider a smooth family $\{ F_t : \Sigma \to \IR^{n+1} \}_{t \in (-T,0]}$ of embeddings, where $\Sigma$ is  a closed $n$-dimensional manifold.
Suppose that each $\Sigma_t := F_t (\Sigma)$, $t \in (-T,0)$, bounds a closed domain that is disjoint from $\MM_t$ and suppose that there is a point $y = F_0 (\ov y) \in \Sigma_0 \cap \MM_0$.
Denote by $\mathbf{H}_y$, $\partial_{t} F_0(\ov y)$, $\mathbf{N}_y$, the mean curvature, velocity and outward pointing normal vectors of $\Sigma_0$ at $y$.
Then $\partial_{t} F_0(\ov y) \cdot  \mathbf{N}_y \geq \mathbf{H}_y \cdot  \mathbf{N}_y$.
\end{Lemma}

\begin{proof}
After possibly shrinking $T$, we can find smooth solutions $\{ \hat F^t_{t'} : \Sigma \to \IR^{n+1} \}_{t'   \in [t,T]}$ to the mean curvature flow equation with initial condition $\hat F^t_{t} = F_t$, for all $t \in (-T,0)$.
Since $\MM$ is a weak set flow, we find that for all $t \in (-T,0)$ the hypersurface $\hat F_{0}^t (\Sigma)$ bounds a domain that disjoint from $\MM_0$.
Thus $y$ lies on the outside of this domain for all $t \in (-T,0)$.

After possibly shrinking $T$ once again, we may express $\hat F^t_{t'}(\Sigma)$ as a graph of a function $\hat h_{t'}^t \in C^\infty (\Sigma)$ over $\Sigma_0 = F_0(\Sigma)$ using the normal exponential map.
Note that $(t,t'') \mapsto \hat h^t_{t+t''}$ forms a smooth family of functions in the $C^\infty$-sense.
Suppose now by contradiction that $\partial_{t} F_0(\ov y) \cdot \mathbf{N}_y  < \mathbf{H}_y \cdot \mathbf{N}_y$.
It follows that
\[ \frac{d}{dt} \bigg|_{t=0} \hat h_{t}^t (\ov y) < \frac{d}{dt'} \bigg|_{t'=0} \hat h_{t'}^0 (\ov y), \]
which  implies
\[ \frac{d}{dt} \bigg|_{t=0} \hat h_{0}^t (\ov y) < 0. \]
Since $\hat h_{0}^0 (\ov y) = 0$, we obtain that $\hat h_{0}^t (\ov y) > 0$ for $t< 0$ close to $0$, which contradicts the fact that $y$ lies on the outside of the domain bounded by $\hat F^t_{0}(\Sigma)$.
\end{proof}
\medskip

\begin{proof}[Proof of Proposition~\ref{Prop_log_u_super_sol}.]
After applying a time-shift and parabolic rescaling, we may assume that $(\vec 0, 0) \in \MM_{\good}$ and it suffices to verify the bound \eqref{eq_sq_log_s} there.
After applying parabolic rescaling and a rotation, we may assume furthermore, that $\s(\vec 0, 0) = 1$ and that $\MM_{\reg,0}$ is tangent to $\IR^{n} \times \{ 0 \}$.
So a neighborhood of $(\vec 0,0)$ we can express $\MM_{\reg}$ as a graph of the form
\[ \{ \big( x^1 \ldots, x^{n}, f(x^1, \ldots, x^{n},t), t  \big) \big\}, \]
where $f$ is a smooth function defined in a neighborhood of the origin $V \subset \IR^{n} \times \IR$, which satisfies
\begin{equation} \label{eq_f_der_zero}
 f(0, \ldots, 0, 0) = \frac{\partial f}{\partial x^i} (0, \ldots, 0, 0) = 0, \qquad i = 1, \ldots, n. 
\end{equation}

Since $(\vec 0, 0) \in \MM_{\good}$, we can find an open neighborhood $U \subset \MM_{\reg}$ of the closure of $\MM_{\reg}(0) \cap B(\vec 0, 1)$ and an $r' > 1$ with the following property.
For any $(x,t) \in \IR^{n+1} \times \IR$ close enough to $(\vec 0, 0)$ and any $r'' \in (0, r')$ the intersection $U \cap \MM_t \cap B(x,r'')$ is connected and its closure is contained in $U$.
Let $z \in \MM_0 \setminus U$ be the point closest to the origin.
Since $\s(\vec 0, 0) = 1$, we must have $|z| = 1$.
After applying another rotation in the first $n$ coordinates and possibly a reflection in the last coordinate, we may assume without loss of generality that
\[ z = (0, \ldots, 0, \sin \theta, \cos \theta ) \]
for some $\theta \in [0, \frac{\pi}2]$.

Consider now a smooth function $u \in C^\infty( W )$, where $W \subset \MM_{\leq 0}$ is a neighborhood of $(\vec 0, 0)$, such that $v \leq \log \s$ on $W$ with equality at $(\vec 0,0)$.
We need to show that then $\square v(\vec 0, 0) \geq 0$.
We will assume by contradiction that 
\[ \square v(\vec 0, 0) < 0. \]
After possibly shrinking $W$, we may assume that for any $( x,t) = (x^1,\lb \ldots, \lb x^{n+1}, t) \in W$ we have $(x^1 \ldots, \lb x^{n}, \lb t) \in V$, so that $f(x^1 \ldots, \lb x^{n}, \lb t)$ is defined.
Set $u := e^v$ so that we have $u \leq \s$ with equality at $(\vec 0, 0)$ and
\begin{equation} \label{eq_squnabul0}
  \big(\square u + |\nabla u|^2 \big) (\vec 0, 0) = ( \square v )(\vec 0, 0) < 0. 
\end{equation}
After replacing $u$ with $u + \delta t$, for small $\delta > 0$, we may even assume that
\[ u(\cdot, t) < \s(\cdot, t) \qquad \text{if} \quad t < 0. \]

\begin{Claim} \label{Cl_u_der}
At the origin in $W \subset \IR^{n+1} \times \IR$ we have 
\[ \frac{\partial u}{\partial x^1} = \ldots  =  \frac{\partial u}{\partial x^{n-1}} = 0, \qquad
\frac{\partial u}{\partial x^{n}} = -\sin \theta. \] 
Recall here that $\frac{\partial}{\partial x^1},\ldots, \frac{\partial}{\partial x^n}$ are tangent to $\MM$ at the origin, so these derivatives should be viewed as intrinsic derivatives within $\MM$.
\end{Claim}
 
\begin{proof}
If $x \in \MM_0$ is close enough to $\vec 0$, then by our choice of $U$, the intersection $U \cap \MM_0 \cap B(x,r'')$ is connected for all $r'' \in (0, r'')$.
Therefore, $$u(x,0) \leq \s(x,0) \leq |x-z|.$$
The claim follows by taking the first derivative at $x = \vec 0$.
\end{proof}

Set $a := 1- \cos \theta$ and consider  the map $F : W \times \IR \to \IR^{n+1}$ given by
\begin{equation*}
 F( (x,t) ,r) := \big( x^1 , \ldots, 
x^{n-1} , x^{n} + r \sin (\theta + a x^{n}),  f(x^1, \ldots, x^{n},t) + r  \cos (\theta + a x^{n}) \big). 
\end{equation*}
Define for $(x,t) \in W$ close enough to the origin 
\[ \td F_t (x) := F \big( (x,t) , u(x,t) \big). \]
Using  \eqref{eq_f_der_zero} and Claim~\ref{Cl_u_der}, we can compute the following first derivatives at the origin:
\begin{align*}
 \frac{\partial \td F_0}{\partial x^1} \bigg|_{(\vec 0, 0)}  &= \frac{\partial}{\partial x^1}, \qquad \ldots, \qquad
\frac{\partial \td F_0}{\partial x^{n-1}} \bigg|_{(\vec 0, 0)}  = \frac{\partial}{\partial x^{n-1}}, \quad \\
\frac{\partial \td F_0}{\partial x^{n}} \bigg|_{(\vec 0, 0)} &= (1 -  \sin^2 \theta + a \cos \theta) \frac{\partial}{\partial x^{n}} + (-\sin \theta \cos \theta - a \sin \theta ) \frac{\partial}{\partial x^{n+1}} \\ &= (\cos \theta )  \frac{\partial}{\partial x^{n}} -(\sin \theta ) \frac{\partial}{\partial x^{n+1}}. %
\end{align*}
Note that these derivatives are linearly independent and orthogonal to $z$.
So after possibly shrinking $W$, we find that the maps $\td F_t : W_t \to \IR^{n+1}$ parameterize a smooth family of embedded hypersurfaces $\Sigma_t \subset \IR^{n+1}$.
By definition we have $z \in \Sigma_0$ and $z$ is also a unit normal vector to $\Sigma_0$ at $z$.
Moreover, the mean curvature of $\Sigma_0$ at $z$, in the direction of $z$, equals (all derivatives of $u$ and $f$ are taken at the origin)
\begin{multline} \label{eq_mean_curvature}
 z \cdot (\triangle \td F_0 )(\vec 0)
= \sin \theta \bigg( (\triangle u)  \sin \theta + 2\frac{\partial u}{\partial x^{n}} a \cos \theta - a^2 \sin \theta \bigg) \\
+ \cos \theta \bigg( \triangle f +  (\triangle u)  \cos \theta - 2\frac{\partial u}{\partial x^{n}}  a \sin \theta - a^2 \cos \theta \bigg) 
=   \triangle u + (\cos \theta) \triangle f - a^2. 
\end{multline}
The normal velocity of the family $\Sigma_t$ at the origin $(\vec 0, 0)$, in the direction of $z$, equals (all derivatives of $u$ and $f$ are again taken at the origin)
\begin{equation} \label{eq_velocity}
 z \cdot \frac{\partial \td F_0}{\partial t} (\vec 0)
=  \partial_t u + (\cos \theta) \partial_t f . 
\end{equation}
We will now derive an inequality between \eqref{eq_mean_curvature} and \eqref{eq_velocity} using Lemma~\ref{Lem_Sigma_compare}.
To do this, we need the following claim.

\begin{Claim} \label{Cl_local_Omega}
If $((x, t),s) \in W \times \IR$ is sufficiently close to $((\vec 0, 0), 1)$ and $t < 0$ and $s \leq 1$, then $F(( x,t), s u (x,t)) \not\in \MM_t$.
\end{Claim}

\begin{proof}
If $(x,t) \in W$ is sufficiently close to $(\vec 0, 0)$ and $t < 0$, then  $u(x,t) < r'$, where $r' > 1$ constant from the beginning of the proof of the proposition.
So since $u(x,t) < \s(x,t)$, we obtain that there is a $\delta_{(x,t)} > 0$ such that for any $r'' \in (0, u(x,t) + \delta_{(x,t)})$ the intersection $\MM_t \cap B(x,r'')$ is connected and the closure of intersection $U \cap \MM_t \cap B(x,r'')$ is contained in $U$.
Thus we must have $\MM_t \cap B(x,r'') = U \cap \MM_t \cap B(x,r'')$.
Now suppose that $F(( x,t), s u (x,t)) \in \MM_t$ for some $0 \leq s \leq 1$.
Since $F(( x,t), s u (x,t)) \in B(x,r'')$, we must have $F(( x,t), s u (x,t)) \in U$.
In conclusion, we have shown that whenever $((x, t),s) \in W \times \IR$ is sufficiently close to $((\vec 0, 0), 1)$ and $t < 0$, then $F(( x,t), s u (x,t)) \in \MM_t$ implies that $F(( x,t), s u (x,t)) \in U$.
However, since $F((\vec 0, 0) ,1) = z \not\in U$, the latter cannot be true for $((x, t),s) \in W \times \IR$ is sufficiently close to $((\vec 0, 0), 1)$.
\end{proof}

Claim~\ref{Cl_local_Omega} implies that there is a family of compact domains $\{ \Omega_t \subset \IR^{n+1} \}_{t \leq 0}$ with the following properties:
\begin{enumerate}
\item $\partial \Omega_0 \cap \Sigma_0$ contains a neighborhood of $z$ in $\Sigma_0$.
\item $\Omega_t \cap \MM_t = \emptyset$ for $t < 0$ close to $0$.
\item $\partial \Omega_t$ depends smoothly on $t$ for $t \leq 0$ close to $0$.
\end{enumerate}
For example, we may choose $\Omega_t$ to be a smoothing of
\[ \ov{B(\vec 0, \eps)} \cap \big\{ F( (x,t), s u(x,t)) \;\; : \;\; (x,t) \in W, s \in [1-\beta, 1] \big\}, \]
where $\eps, \beta > 0$ are chosen appropriately small.
Applying Lemma~\ref{Lem_Sigma_compare} to the family $\{ \partial \Omega_t \}_{t \leq 0}$ implies that at $(\vec 0, 0)$
\[   \partial_t u + (\cos \theta) \partial_t f \geq \triangle u + (\cos \theta) \triangle f - a^2. \]
Combining this with \eqref{eq_squnabul0}, Claim~\ref{Cl_u_der}, and the fact that $\partial_t f = \triangle f$ due to the mean curvature flow equation at $(\vec 0, 0)$, implies
\[     0 > |\nabla u|^2  - a^2 = (\sin \theta)^2  - a^2 = (\sin \theta)^2  - (1-\cos \theta)^2 = 2 \cos \theta - 2 (\cos \theta)^2 \geq 0. \]
This yields the desired contradiction.
\end{proof}
\medskip

We will also need the following lemma.

\begin{Lemma} \label{Lem_ends_min_surf_improved}
In the setting of Lemma~\ref{Lem_ends_min_surf}, we also have
\[ \mathbf{r} \frac{|\nabla \s|}{\s} \lto 0, \qquad \mathbf{r} \, \s \lto \infty. \]
\end{Lemma}

\begin{proof}
Recall the ends $\Sigma_1, \ldots, \Sigma_k \subset \Sigma$, which are graphs of functions $u_1 <  \ldots <u_k$ with
\begin{equation} \label{eq_old_nab_u_bound}
 \mathbf{r} \frac{|\nabla (u_{j+1}- u_j)|}{u_{j+1}-u_j} \lto 0. 
\end{equation}
For any point $x \in \IR^3$, we denote by $\ov x \in \IR^2$ its projection onto the first two coordinates.
Consider the separation function $\s$ restricted to a fixed $\Sigma_j$.
Outside of some compact subsets we have $\s = \min \{ \s^-, \s^+ \}$, where $\s^\pm = d(\cdot, \Sigma_{j\pm1})$ denote the distance functions to the nearby end (whenever they exists).
It suffices to derive the desired derivative bound for these functions individually.
To do this, note first the functions $\s^\pm$ are smooth outside some bounded subset and that we have
\begin{equation} \label{eq_s_ujp1u}
 \frac{\s^\pm(x)}{|u_{j\pm1}-u_j|(\ov x)}  \xrightarrow[x \to \infty]{} 1. 
\end{equation}
Fix some point $x \in \Sigma_j$, far enough away from the origin, and let $x' \in \Sigma_{j\pm1}$ be the point of shortest distance to $x$.
If we denote by $\mathbf{N}_x, \mathbf{N}_{x'}$ the upwards pointing normal vectors at $x$ and $x'$, then
\[
 |\nabla \s^\pm| = \tan \angle (\mathbf{N}_x, \mathbf{N}_{x'}) 
= F(\nabla u_j(\ov x),\nabla u_{j\pm1}(\ov x')- \nabla u_{j}(\ov x)), \]
where $F$ is a smooth function defined in a neighborhood of the origin of $\IR^2 \times \IR^2$ with the property that $F(v, 0) = 0$ for all $v \in \IR^2$ close to the origin.
It follows that there is a constant $C < \infty$ such that for $x$ far enough from the origin, we have
\begin{equation} \label{eq_nabspm_C}
 |\nabla \s^\pm| \leq C|\nabla u_{j\pm1}(\ov x')- \nabla u_{j}(\ov x)|. 
\end{equation}
Since $|\nabla u_j|, |\nabla u_{j\pm1}|\to 0$ at infinity, we have 
\begin{equation} \label{eq_xxp_spm}
\frac{|\ov x - \ov x' |}{\s^\pm(x)} \xrightarrow[x \to \infty]{} 0.
\end{equation}
So since $\mathbf{r} |\nabla^2 u_{j\pm1} | \to 0$ at infinity, we therefore have by \eqref{eq_nabspm_C}, \eqref{eq_xxp_spm}, \eqref{eq_old_nab_u_bound} and Lemma~\ref{Lem_ends_min_surf}\ref{Lem_ends_min_surf_d} and assuming $C$ to be generic
\begin{align*}
 |x| \frac{|\nabla \s^\pm|(x)}{|u_{j \pm 1} - u_j|(\ov x)}
&\leq C  |x| \frac{|\nabla u_{j\pm1}(\ov x')- \nabla u_{j}(\ov x)|}{|u_{j \pm 1} - u_j|(\ov x)} \\
&\leq C |x| \frac{|\nabla u_{j\pm1}(\ov x')- \nabla u_{j\pm1}(\ov x)|}{|u_{j \pm 1} - u_j|(\ov x)} + C |x| \frac{|\nabla u_{j\pm1}(\ov x)- \nabla u_{j}(\ov x)|}{|u_{j \pm 1} - u_j|(\ov x)} \\
&\leq C \frac{|\ov x - \ov x'|}{|u_{j \pm 1} - u_j|(\ov x)} + C |x| \frac{|\nabla u_{j\pm1}(\ov x)- \nabla u_{j}(\ov x)|}{|u_{j \pm 1} - u_j|(\ov x)} \\
&\leq C \frac{\s^\pm(x)}{|u_{j \pm 1} - u_j|(\ov x)} \, \frac{|\ov x - \ov x' |}{\s^\pm(x)} + C |x| \frac{|\nabla u_{j\pm1}(\ov x)- \nabla u_{j}(\ov x)|}{|u_{j \pm 1} - u_j|(\ov x)} \xrightarrow[x \to \infty]{} 0.
\end{align*}
Combining this with \eqref{eq_s_ujp1u} implies the desired result.
\end{proof}
\bigskip

\section{Smoothing of the local scale function} \label{sec_smooth_rloc}
Before establishing our key integral bound in Section~\ref{sec_int_bound}, we need to introduce a smoothed version $\td r_{\loc}$ of the local scale function $r_{\loc}$.
This smoothened local scale function will be comparable to $r_{\loc}$ and enjoy additional uniform local derivative bounds.
These derivative bounds will enable us to pass differential inequalities involving $\td r_{\loc}$ to limits.
In what follows, we will describe the construction of $\td r_{\loc}$, which we will then fix for the remainder of the paper.
So we will view $\td r_{\loc}$ as a well-defined quantity for any mean curvature flow.

Consider a weak set flow $\MM \subset \IR^{n+1} \times I$ with regular part $\MM_{\reg} \subset \MM$ and  recall the local scale function $r_{\loc}$ from Subsection~\ref{subsec_almostreg}.
Fix a smooth cutoff function $\psi : [0,\infty) \to [0,1]$ such that $\psi \equiv 1$ near $0$ and $\psi \equiv 0$ on $[1,\infty)$.
Let $\zeta > 0$ be a constant, which we will determine in Lemma~\ref{Lem_zeta_choice} below, and define
\begin{multline*}
 \tdr_{\loc} (x,t) := \int_I \int_{\MM_{\reg,t'}} r_{\loc}(x',t') \, \psi \bigg( \frac{|x-x'|^2}{ \zeta \, r_{\loc}^2(x',t')} \bigg) \, \psi \bigg( \frac{|t-t'|}{\zeta \, r_{\loc}^2(x',t')} \bigg) d\HH^2(x') dt' \Bigg\slash \\ \int_I \int_{\MM_{\reg,t'}}\psi \bigg( \frac{|x-x'|^2}{ \zeta \, r_{\loc}^2(x',t')} \bigg) \, \psi \bigg( \frac{|t-t'|}{\zeta \, r_{\loc}^2(x',t')} \bigg) d\HH^2(x') dt'. 
\end{multline*}
Note that the definition only makes sense if  $(x,t)$ is sufficiently close to $\MM_{\reg}$ in order to guarantee non-vanishing of the denominator.
In fact, in this paper we will only consider the restriction of $\td r_{\loc}$ to $\MM_{\reg}$.

\medskip 
\begin{Lemma} \label{Lem_zeta_choice}
There is a choice for $\zeta > 0$ (which we will fix henceforth) such that the following is true for any weak set flow $\MM \subset \IR^{n+1} \times I$ over an open time-interval $I$:
\begin{enumerate}[label=(\alph*)]
\item \label{Lem_zeta_choice_a} $0.9 r_{\loc} \leq \tdr_{\loc} \leq  1.1 r_{\loc}$
\item \label{Lem_zeta_choice_b} There is a constant $C= C(n) < \infty$ such that on $\MM_{\reg}$ 
\[ |\nabla \tdr_{\loc}| \leq C, \qquad |\nabla^2 \tdr_{\loc}|, |\partial_t \tdr_{\loc}| \leq C\tdr_{\loc}^{-1}. \]
\item \label{Lem_zeta_choice_c} The quantity $\tdr_{\loc}$ scales like a distance under parabolic rescaling in the sense that for any $\lambda > 0$ we have
\[ \tdr_{\loc}^{\lambda \MM} (\lambda x,\lambda^2 t) = \lambda \, \tdr_{\loc}^{\MM} (x,t). \]
\item \label{Lem_zeta_choice_dd} Consider a sequence of bounded almost regular mean curvature flows $\MM^i \subset \IR^3 \times I_i$ with $I^i \to \IR$, which locally smoothly converge to the stationary flow $\Sigma \times \IR$ corresponding to a properly embedded minimal surface $\Sigma \subset \IR^3$.
Suppose that $\Sigma$ is not an affine plane.
Then for every sequence $(x_i, t_i) \in \MM^i$ with $(x_i, t_i) \to (x_\infty, t_\infty)$ we have
\[ \lim_{i \to \infty} \big|\partial_t \td r_{\loc}^{\MM^i} \big| (x_i,t_i) = 0. \]
\item \label{Lem_zeta_choice_d} Consider a sequence of bounded almost regular mean curvature flows $\MM^i \subset \IR^3 \times I_i$ with $I^i \to \IR$, whose associated Brakke flows weakly converge to a Brakke flow $(\mu^\infty_t)_{t \in \IR}$.
Suppose that there is an affine plane $\Sigma \subset \IR^3$ and a time $t^* \in \IR$ such that $d\mu_t = k_t d (\HH^2 \lfloor \Sigma)$, where $k_t \geq 2$ for $t < t^*$ and $k_t = 1$ for $t > t^*$.
Then the flows $\MM^i$ restricted to the time-interval $(t^*,\infty)$ converge locally smoothly to the stationary flow $\Sigma \times (t^*, \infty)$ and for every sequence $(x_i, t_i) \in \MM^i$ with $(x_i,t_i) \to (x_\infty, t_\infty) \in \Sigma \times (t^*,\infty)$ we have
\[ \lim_{i \to \infty} |\nabla \td r^{\MM^i}| (x_i, t_i) = 0, \qquad \lim_{i \to \infty} |\nabla^2 \td r^{\MM^i}| (x_i, t_i) = 0, \qquad \lim_{i \to \infty} |\partial_t \td r^{\MM^i}| (x_i, t_i) > 0. \]
\end{enumerate}
\end{Lemma}

\begin{proof}
If $\zeta$ is sufficiently small, then
\[ \psi \bigg( \frac{|x-x'|^2}{ \zeta \, r_{\loc}^2(x',t')} \bigg) \, \psi \bigg( \frac{|t-t'|}{\zeta \, r_{\loc}^2(x',t')} \bigg)\neq 0 \]
implies that $(x,t) \in P(x',t', C_0 \zeta^{1/2} r_{\loc}(x',t'))$ for some universal $C_0 < \infty$.
Therefore, by Lemma~\ref{Lem_rloc_Lipschitz} the value $\td r_{\loc}(x,t)$ only depends on $r_{\loc}$ restricted to $\MM_{\reg} \cap P(x,t,C_1 \zeta^{1/2} r_{\loc}(x,t))$ for some universal $C_1 < \infty$.
Since $r_{\loc}$ can be bounded on this parabolic ball again by Lemma~\ref{Lem_rloc_Lipschitz}, we obtain Assertion~\ref{Lem_zeta_choice_a} for sufficiently small $\zeta$, which we will now fix.
Assertion~\ref{Lem_zeta_choice_c} is clear by definition.
To see Assertion~\ref{Lem_zeta_choice_b}, we may assume after parabolic rescaling that $r_{\loc}(x,t) =1$.
The derivative bounds now follow from uniform derivative bounds on $\psi$.
Assertion~\ref{Lem_zeta_choice_dd} follows similarly.

It remains to show Assertion~\ref{Lem_zeta_choice_d}.
The statement about the local smooth convergence is standard and follows from the fact that the limiting flow has multiplicity one at times $t > t^*$.
It follows that we have local uniform convergence $r_{\loc} \to r_\infty$ on $\IR^3 \times (t^*,\infty)$, where $r_\infty(x,t) = \sqrt{t-t^*}$.
Due to the definition of $\td r_{\loc}$, this implies local smooth convergence $\td r_{\loc} \to \td r_\infty$ on $\IR^3 \times (t^*,\infty)$, where
\begin{align*}
 \td r_{\infty}(x,t) &= \int_{t_*}^\infty \int_{\Sigma} \sqrt{t' - t_*} \, \psi \bigg( \frac{|x-x'|^2}{ \zeta \, r_{\infty}^2(x',t')} \bigg) \, \psi \bigg( \frac{|t-t'|}{\zeta \, r_{\infty}^2(x',t')} \bigg) d\HH^2(x') dt' \Bigg\slash \\
 &\qquad\qquad\qquad\qquad\qquad\qquad \int_{t_*}^\infty \int_{\Sigma}  \psi \bigg( \frac{|x-x'|^2}{ \zeta \, r_{\infty}^2(x',t')} \bigg) \, \psi \bigg( \frac{|t-t'|}{\zeta \, r_{\infty}^2(x',t')} \bigg) dt'  \displaybreak[1] \\
  &= \int_{t_*}^\infty  \sqrt{t' - t_*}  \, \psi \bigg( \frac{|t-t'|}{\zeta \, (t'-t^*)} \bigg)  dt' \Bigg\slash  \int_{t_*}^\infty  \psi \bigg( \frac{|t-t'|}{\zeta \, (t'-t^*)} \bigg)  dt' = c_\psi \sqrt{t-t^*},
\end{align*}
for some constant $c_\psi > 0$, which may depend on the function $\psi$.
This implies Assertion~\ref{Lem_zeta_choice_d}.
\end{proof}

\bigskip

\section{An integral bound on the local scale} \label{sec_int_bound}
The proof of the key estimate in this paper, Theorem~\ref{Thm_key_thm}, hinges on a novel spacetime integral bound, Theorem~\ref{Thm_L2_bound}.
This bound concerns the integral of $r_{\loc}^{-2}$ over the domain of points within a mean curvature flow that experience local collapse, but which, crucially, are not locally modeled on minimal surfaces. 
The bound on this integral can be quantified in terms of an increase of the Gaussian area over successive scales.
It's worth noting that the integration domain also encompasses areas where the flow behaves in a manner akin to a translating soliton, for instance the grim reaper times $\IR$. 
While the evolution of the flow in these regions may be poorly understood,  our integral bound allows us to control their impact on the flow as a whole.

From a more philosophical standpoint, an integral bound on $r_{\loc}^{-2}$ may be appear unachievable, primarily due to the critical choice of the exponent.
This is related to the fact that $r_{\loc}^{-2}$ is comparable to the maximal function of the squared norm $|A|^2$ of the second fundamental form.
While an integral bound on $|A|^2$ is usually available, the possible occurrence of bubbling could ostensibly preclude a similar bound on $r_{\loc}^{-2}$.
Consequently, one may only expect a \emph{weak} $L^1$-bound on $r_{\loc}^{-2}$.
However, the linchpin making Theorem~\ref{Thm_L2_bound} work is that ``bubbles'', i.e., regions modeled on minimal surfaces, are explicitly excluded from the domain of integration.

Before stating our integral bound, we first define a set of points that are locally, at scale $r_{\loc}$, modeled by minimal surfaces or planes.
In the case in which the minimal surface is a plane and this plane is located near other points of the flow, we require an additional differential condition on the smoothed version $\td r_{\loc}$ of the local scale function (see Section~\ref{sec_smooth_rloc}).

\begin{Definition} \label{Def_CNT_new}
Consider an almost regular mean curvature flow $\MM \subset \IR^3 \times I$ and numbers $\eps, A > 0$.
We define $\MM_{\reg}^{(\eps,A)} \subset \MM_{\reg}$ to be the set of points $(x,t) \in \MM_{\reg}$ %
for which one of the following is true:
\begin{enumerate}[label=(\arabic*)]
\item \label{Def_CNT_new_2} We have 
\begin{equation} \label{eq_eps_in_I}
[t-\eps^{-1} r_{\loc}^2(x,t),t+\eps^{-1} r_{\loc}^2(x,t)] \subset I
\end{equation}
and the parabolically rescaled flow $r_{\loc}^{-1}(x,t) (\MM  - (x , t))$ restricted to the interior of the parabolic neighborhood $P(\vec 0, 0, \eps^{-1})$ is $\eps$-close\footnote{Here and in the rest of the paper we take $\eps$-closeness to mean that the intersection of the rescaled flow with the interior of $P(\vec 0, 0, \eps^{-1})$ can be expressed as a normal graph of a function over an open subset of the stationary flow of a minimal surface such that the function, along with its first $\lfloor \eps^{-1} \rfloor$ covariant derivatives, is bounded by $\eps$.} to an open subset of the stationary flow associated with a non-trivial, smooth,  properly embedded, multiplicity one minimal surface $\Sigma \subset \IR^3$ satisfying the bounds
\begin{equation} \label{eq_Theta_genus_bound}
   \Theta (\Sigma) \leq A \textQQqq{and} \genus(\Sigma) \leq A.
\end{equation}
\item \label{Def_CNT_new_3} There is a plane $P \subset \IR^3$ through the origin and a connected, open neighborhood $U \subset \MM_{\reg}$ of $(x,t)$ with the property that the parabolic rescaling $r_{\loc}^{-1}(x,t) (U - (x,t))$ is $\eps$-close to the product $B^P(\vec 0, \eps^{-1}) \times (-1+\eps, \eps^{-1})$, where the first factor denotes the $\eps^{-1}$-ball within $P$.
We also have the following bound on $U_t = \MM_t \cap U$ for all $t \in (-1+\eps, \eps^{-1})$
\[ (1-\eps)  r_{\loc}(x,t) \leq r_{\loc}(\cdot, t)  \leq (1+\eps) r_{\loc}(x,t). \] 
Moreover, if for some $(x',t') \in U$ we have $\s(x',t') \geq (1+\eps) r_{\loc}(x',t')$, then 
\begin{equation} \label{eq_def_squ_log_r_positive}
 \square  \big(\log \tdr_{\loc}\big)(x',t')  > 0. 
\end{equation}
\end{enumerate}
\end{Definition}
\medskip

The precise purpose of Property~\ref{Def_CNT_new_3} will become clear in Subsection~\ref{subsec_mod_sep}.
Roughly speaking, in the setting of Property~\ref{Def_CNT_new_3}, the separation function $\s$ may not satisfy a useful differential inequality, as it may describe the geometry of the flow at a much larger scale than $r_{\loc}$.
To address this, we will use $\td r_{\loc}$ as a local substitute and the bound \eqref{eq_def_squ_log_r_positive} ensures that the desired differential inequality holds for $\td r_{\loc}$ instead of $\s$.

The following theorem, which is the main theorem of this section, gives us an integral bound on $r_{\loc}^{-2}$ over an open neighborhood $\mathcal{U}$ of the complement of $\MM_{\reg}^{(\eps,A)} \subset \MM$.
Note that this implies that $\MM_{\sing} \subset \mathcal{U}$.

\begin{Theorem} \label{Thm_L2_bound}
For any $A < \infty$, $\eps > 0$ there are constants $\delta(A,\eps) > 0$ and $C(A,\eps) < \infty$ such that the following is true.
Consider a bounded almost regular mean curvature flow $\MM \subset \IR^3 \times I$ and the subset $\MM_{\reg}^{(\eps,A)} \subset \MM_{\reg}$ from Definition~\ref{Def_CNT_new}.
Then there is an open neighborhood $\mathcal{U}$ of $\MM \setminus \MM_{\reg}^{(\eps,A)}$ within $\MM$ such that the following holds.

Consider a point $(x_0, t_0) \in \IR^3 \times \ov I$ and a $\tau_0 > 0$ such that $t_0 - \tau_0 \in I$ and
\[ \Theta_{(x_0, t_0)} (\tau_0)  \leq A \textQQqq{and} \genus(\MM_{I \cap (-\infty,t_0)})  \leq A. \]
Define
\begin{equation} \label{eq_XXt0delta}
 \XX^{t_0,\delta} := \big\{ (x,t) \in \MM_{< t_0} \;\; : \;\;  r_{\loc}(x,t) \geq \delta \sqrt{t_0 -t} \big\}, 
\end{equation}
Then for any $0 < \tau_1 < \tau_2 \leq \tfrac{1}2 \tau_0$ 
we have the bound
\begin{multline} \label{eq_int_bounded_TT}
   \int_{t_0-\tau_2}^{t_0- \tau_1} \frac1{t_0 -t} \int_{((\mathcal{U} \cap \MM_{\reg}) \setminus \XX^{t_0,\delta})_t \cap B(x_0, A \sqrt{t_0-t})} r_{\loc}^{-2}  \, d\HH^2 \, dt \\
\leq C(A,\eps) \big( \Theta_{(x_0, t_0)} (2 \tau_2) - \Theta_{(x_0, t_0)} ( \tfrac12 \tau_1) \big).
\end{multline}
\end{Theorem}
\bigskip

Note that if $\MM = \MM_{\reg}$ is fully regular, then we can just choose $\mathcal{U}$ to be a slightly larger open subset containing $\MM \setminus \MM_{\reg}^{(\eps,A)}$.
So the purpose of the subset $\mathcal{U}$ is to include a neighborhood of the singular set in the integrand of \eqref{eq_int_bounded_TT}.

The proof of Theorem~\ref{Thm_L2_bound} relies on the identity  \eqref{eq_TT_minus}, which allows us to bound an integral over the quantity $|\frac{(\mathbf{x}-x_0)^\perp}{2(t_0 - \mathbf{t})} + \mathbf{H}|^2$ by the increase of the Gaussian area on the right-hand side of \eqref{eq_int_bounded_TT}.
To convert this integral estimate into the desired integral bound over $r_{\loc}$, we take advantage of the fact that the flow sufficiently deviates locally from a minimal surface on the domain of the integral on the left-hand side of \eqref{eq_int_bounded_TT}.
Consequently, for any point $(x,t)$ in this domain of integration, it is possible to deduce that the average of  $|\frac{(\mathbf{x}-x_0)^\perp}{2(t_0 - \mathbf{t})} + \mathbf{H}|^2$ over a parabolic ball of radius $C r_{\loc}(x,t)$, where $C$ is  a suitably large constant, is roughly bounded from below by the average over $r_{\loc}^{-2}$ over a smaller parabolic ball of radius $\frac1{10}r_{\loc}(x,t)$.
We will then proceed by covering the domain of integration by these smaller parabolic balls, thereby allowing us to bound the left-hand side of \eqref{eq_int_bounded_TT} by a sum of integrals of $|\frac{(\mathbf{x}-x_0)^\perp}{2(t_0 - \mathbf{t})} + \mathbf{H}|^2$ over the larger parabolic balls having radius $C r_{\loc}(x,t)$.
Ideally, if the degree of the overlap of these larger parabolic balls could be bounded from above, it would immediately imply the desired bound.
However, this may not be the case in general.
To address this issue, we have to be more cautious in formulating a lower bound on the average of $|\frac{\mathbf{x}^\perp}{2(t_0 - \mathbf{t})} + \mathbf{H}|^2$.
We will achieve this by excluding points from the domain of the averaging integral that may fall within too many of these larger parabolic balls.
This is done by excluding points where the Gaussian area at scale $r_{\loc}(x,t)$ is approximately constant.

The following lemma delineates this lower integral (i.e., average) bound in a contrapositive form.
In other words, it states that whenever the integral $|\frac{\mathbf{x}^\perp}{2(t_0 - \mathbf{t})} + \mathbf{H}|^2$ over a large parabolic ball, excluding specific points, is too small, then the center of this ball must lie in $\MM_{\reg}^{(\eps,A)}  \cup \XX^{t_0,\delta}$.
The lemma will follow via a limit argument.

\begin{Lemma} \label{Lem_close_to_min}
For every $A < \infty$ and $\eps > 0$ there is a constant $\delta(A,\eps) > 0$ such that the following holds.
Let $\MM$ be a bounded almost regular mean curvature flow on $\IR^3 \times (-T,t_0)$ and $x_0 \in \IR^3$ a point.
Suppose that:
\begin{enumerate}[label=(\roman*)]
\item \label{Lem_close_to_min_i} $0 < t_0 <  T$.
\item \label{Lem_close_to_min_ii} $|x_0| \leq A \sqrt{t_0}$.
\item \label{Lem_close_to_min_iii} We have
\begin{equation} \label{eq_Lem_Theta_genus}
\sup_{0 < \tau < T} \Theta_{(x_0,t_0)}(\tau)  \leq A , \qquad \genus(\MM) \leq  A.
 \end{equation}
\item \label{Lem_close_to_min_iv} We have $(\vec 0, 0) \in \MM_{\reg}$ and $r_{\loc}(\vec 0, 0) = 1$.
\item \label{Lem_close_to_min_v} There is a subset $\ZZ \subset \MM$ with the following properties:
\begin{itemize}
\item For every $(x,t) \in \ZZ$ we have $t - \delta^{-2} > -T$ and
\[ \big| \Theta_{(x,t)} (\delta^{-2}) - \Theta_{(x,t)}(\delta^2) \big| \leq \delta. \]
\item On the complement of the subset $\ZZ$ we have the following $L^2$-bound:
\begin{equation} \label{eq_int_setminus_Z_delta}
   \int_{- \min \{ \delta^{-2},T \}}^{\min\{ \delta^{-2}, \frac12 t_0 \}} \int_{(\MM_{\reg} \setminus \ZZ)_t \cap B(\vec 0,  \delta^{-1})} \bigg| \frac{(\mathbf{x}-x_0)^{\perp}}{2(t_0 - \mathbf{t})} + \mathbf{H} \bigg|^2\rho_{(x_0,t_0)}  \, d\HH^2 \, dt \leq \frac{\delta}{t_0}. 
\end{equation}
\end{itemize}
\end{enumerate}
Then $(\vec 0, 0) \in \MM_{\reg}^{(\eps,A)}  \cup \XX^{t_0,\delta}$.
\end{Lemma}

\begin{proof}
We argue by contradiction.
Fix $A, \eps$ and consider a sequence of counterexamples consisting of bounded almost regular mean curvature flows $\MM^i$ defined on $(-T_i, t_{0,i})$, points $x_{0,i} \in \IR^3$, subsets $\ZZ^i \subset \MM^i$ and a sequence $\delta_i \to 0$ such that the assumptions of the lemma hold, but such that
\begin{equation} \label{eq_rloc_less_deltai}
r_{\loc}(\vec 0, 0) < \delta_i \sqrt{t_{0,i}}
\end{equation}
and
\begin{equation} \label{eq_origin_not_CNT}
 (\vec 0, 0) \not\in (\MM^i_{\reg})^{(\eps,A)},
\end{equation}
which means that neither property in Definition~\ref{Def_CNT_new} holds at $(\vec 0, 0)$.
Assumptions~\ref{Lem_close_to_min_i}, \ref{Lem_close_to_min_iv} combined with \eqref{eq_rloc_less_deltai} imply that $T_i, t_{0,i} \to \infty$.
Assumptions~\ref{Lem_close_to_min_i}, \ref{Lem_close_to_min_ii}, \ref{Lem_close_to_min_iii} imply uniform upper bounds on $\Theta_{(x',t')}(\tau')$ for any fixed $(x',t') \in \IR^3 \times\IR$ and $\tau' > 0$.
This allows us to pass to a subsequence such that the unit-regular Brakke flows associated with the almost regular mean curvature flows $\MM^i$ weakly converge to a unit-regular Brakke flow $( \mu^\infty_t )_{t \in \IR}$ on $\IR^3$ whose support we will denote by $\mathcal{M}^\infty \subset \IR^3 \times \IR$.
By Assumption~\ref{Lem_close_to_min_iv} we have
\begin{equation} \label{eq_Minfty_reg_origin}
 (\vec 0, 0) \in \MM^\infty_{\reg}, \qquad r_{\loc}^{\MM^\infty} (\vec 0, 0) \geq 1
\end{equation}
and the multiplicity of $( \mu^\infty_t )_{t \in \IR}$ restricted to the interior of the parabolic ball $P(\vec 0, 0,1)$ is one.

Let $\ZZ^\infty \subset \IR^3 \times \IR$ be the limit set of the sets $\ZZ^i$.
That is $(x^*,t^*) \in \ZZ^\infty$ if and only if, after passing to a subsequence, there is a sequence $(x^*_i,t^*_i) \in \ZZ^i$ with the property that $(x^*_i,t^*_i) \to (x^*,t^*)$.
Set
\[ I^* := \{ t^*  \in \IR \;\; : \;\; \ZZ^\infty_{t^*} \neq \emptyset \}. \]

\begin{Claim} \label{Cl_limit_min_shr}
Consider an interval $I' \subset \IR \setminus I^*$.
Then the support of the restricted flow $( \mu^\infty_t )_{t \in I'}$ is given by a stationary flow associated with a smooth, properly embedded minimal surface $\Sigma \subset \IR^3$ of possibly higher multiplicity satisfying
\begin{equation} \label{eq_Theta_genus_bounds}
\Theta(\Sigma) \leq A \textQQqq{and} \genus (\Sigma) \leq A.
\end{equation}
The flow $( \mu^\infty_t )_{t \in I'}$ itself may have higher, but locally constant multiplicity, so $d\mu_t^\infty = k d ( \HH^2 \lfloor \Sigma )$ for all $t \in I'$ and for some locally constant function $k : \Sigma \to \IN$.
\end{Claim}

\begin{proof}
Note that the restricted flow $\MM^\infty_{I'}$ is non-empty, because $(\vec 0, 0) \in \MM^\infty_{\reg}$.
By \eqref{eq_int_setminus_Z_delta} we have
\[ \int_{I'} \int_{\MM^i_{\reg,t} \cap B(\vec 0, \delta_i^{-1})}  \bigg| \frac{(\mathbf{x}-x_{0,i})^\perp}{2(t_{0,i} - \mathbf{t})} + \mathbf{H} \bigg|^2 \, t_{0,i} \, \rho_{(x_{0,i},t_{0,i})}  \, d\HH^2 \, dt  \longrightarrow 0. \]
By Assertion~\ref{Lem_close_to_min_ii} we have local uniform convergence
\[ \frac{(\mathbf{x}-x_{0,i})^\perp}{2(t_{0,i} - \mathbf{t})} \longrightarrow 0 \]
and local uniform lower bounds on $t_{0,i} \, \rho_{(x_{0,i},t_{0,i})}$.
So
\[ \int_{I'} \int_{\MM^i_{\reg,t} \cap B(\vec 0, \delta_i^{-1})}   |\mathbf{H}|^2 \, d\HH^2 \, dt  \longrightarrow 0. \]
So the claim follows from Lemma~\ref{Lem_Brakke_to_min_shrink}\ref{Lem_Brakke_to_min_shrink_a}. 
\end{proof}
\medskip

\begin{Claim} \label{Cl_Istar_shrinker}
If there is a time $t^* \in I^*$, then support of the restriction $( \mu^\infty_t)_{t < t^*}$ is given by the flow of a smooth shrinker $\Sigma \subset \IR^3$ based at some point $(x^*, t^*) \in \ZZ^\infty$, where the flow itself may have higher, but locally constant multiplicity.
That is there is a locally constant function $k : \Sigma \to \IN$ such that $$d\mu^\infty_{t} = k_t \, d \big(\HH^2 \lfloor \big(|t^*-t|^{1/2} (\Sigma - x^*) + x^* \big) \big)$$ for all $t < t^*$, where $k_t (x) := k(|t^*-t|^{-1/2} (x-x^*)+x^*)$.
Moreover, $\Sigma$ satisfies \eqref{eq_Theta_genus_bounds}.
\end{Claim}

\begin{proof}
This follows directly from Lemma~\ref{Lem_Brakke_to_min_shrink}\ref{Lem_Brakke_to_min_shrink_b} using Assumption~\ref{Lem_close_to_min_v}.
\end{proof}
\medskip

\begin{Claim} \label{Cl_T_infty}
The flow $( \mu^\infty_t )_{t \in \IR}$ is not the stationary flow associated with a properly embedded multiplicity one minimal surface $\Sigma \subset \IR^3$ that satisfies \eqref{eq_Theta_genus_bounds}.
\end{Claim}

\begin{proof}
Otherwise we have local, smooth convergence $\MM^i \to \MM^\infty$ on $\IR^3 \times \IR$.
This implies Property~\ref{Def_CNT_new_2} of Definition~\ref{Def_CNT_new} for large $i$.
\end{proof}
\medskip

\begin{Claim} \label{Cl_twots}
If $t^*_1, t^*_2 \in I^*$, where  $t^*_1 < t^*_2$, then the restriction $( \mu^\infty_t )_{t < t^*_2}$ is a  stationary affine plane of possibly higher, but constant multiplicity.
\end{Claim}

\begin{proof}
This is a direct consequence of Claim~\ref{Cl_Istar_shrinker} and the splitting principle for shrinkers.
(A similar principle was also used in \cite[Sec~3.2]{Cheeger_Haslhofer_Naber_13}.)
\end{proof}
\medskip

We can now finish the proof of the lemma by distinguishing the following cases:
\medskip

\textit{Case 1: $I^* = \emptyset$.} \quad
In this case we can apply Claim~\ref{Cl_limit_min_shr} to deduce that support of the flow $( \mu^\infty_t )_{t \in \IR}$ is given by a stationary minimal surface $\Sigma \subset \IR^3$ satisfying \eqref{eq_Theta_genus_bounds}, where the flow itself may have higher, locally constant multiplicity.
Note that this implies that $\MM^\infty_t = \Sigma$ for all $t \in \IR$.
By \eqref{eq_Minfty_reg_origin} we know that there is a component $\Sigma' \subset \Sigma$ such that $(\mu^\infty_t)_{t \in \IR}$ has multiplicity one on $\Sigma'$.
By Claim~\ref{Cl_T_infty} we find that $\Sigma' \neq \Sigma$.
Therefore, $\Sigma$ is not connected, so by Lemma~\ref{Lem_ends_min_surf}\ref{Lem_ends_min_surf_aa1} it is a union of finitely many parallel affine planes. The flow $(\mu^\infty_t)_{t \in \IR}$ may have higher multiplicity $\Sigma \setminus \Sigma'$, but must have multiplicity one on $\Sigma'$.

By \eqref{eq_origin_not_CNT} we know that the negation of Property~\ref{Def_CNT_new_3} in Definition~\ref{Def_CNT_new} holds for all $i$.
As the first two statements of this property hold for large $i$, the last one must fail.
So we have $\s(x'_i, t'_i) \geq 1+\eps > 1$ for large $i$, where $(x'_i,t'_i)$ is a bounded sequence of points and times.
After passing to a subsequence, we may assume that $(x'_i, t'_i) \to (x'_\infty,t'_\infty) \in \Sigma' \times \IR$.
By the avoidance principle this implies that for $t > t'_\infty$ sufficiently close to $t'_\infty$ the difference $\MM^\infty_t \setminus \Sigma' = \Sigma \setminus \Sigma'$ is disjoint from the ball $B(x'_\infty, 1 + \frac12\eps)$.
In other words, the planes $\Sigma \setminus \Sigma'$ must have distance $\geq 1 + \frac12 \eps$ from $\Sigma'$.
But this implies that $( \mu^\infty_t )_{t \in \IR}$ restricted to a slightly larger parabolic ball $P(\vec 0, 0,1 + \frac12 \eps)$ is the restriction of a stationary flow of an affine plane of multiplicity one.
By smooth convergence, we obtain a contradiction to Assumption~\ref{Lem_close_to_min_iv}.

\medskip

\textit{Case 2: $I^* = \{ t^* \}$ for some $t^* \in \IR$.} \quad 
By Claim~\ref{Cl_limit_min_shr} we know that the support of $( \mu^\infty_t )_{t < t^*}$ is the stationary flow associated with a smooth minimal surface $\Sigma \subset \IR^3$, where the flow itself may have higher, but locally constant multiplicity.
In addition, by Claim~\ref{Cl_Istar_shrinker} this flow must be a shrinker with respect to some point at time $t^*$.
Thus $\Sigma$ must be an affine plane and $\mu^\infty_t = k^- (\HH^2 \lfloor \Sigma)$ for $t < t^*$, where $k^- \in \IN_0$ is the multiplicity of the flow $( \mu^\infty_t )_{t < t^*}$.

Due to the avoidance principle, we must have $\MM^\infty_t \subset \Sigma$ for all $t \in \IR$.
By Claim~\ref{Cl_limit_min_shr} we obtain that  $\mu^\infty_t = k^+ (\HH^2 \lfloor \Sigma)$ for $t > t^*$, for some other multiplicity $k^+ \in \IN_0$.
As the multiplicity can only decrease and due to Claim~\ref{Cl_T_infty}  we have $k^+ < k^-$.
By \eqref{eq_Minfty_reg_origin} we know that $k^- = 1$ if $t^* > -1$.
Similarly, $k^+ = 1$ if $t^* < 1$.
Now suppose that $t^* > -1$.
Then $0 \leq k^+ < k^- = 1$, in contradiction to the unit-regularity of $(\mu^\infty_t )_{t \in \IR}$.
Thus we must have $t^* < 1$ and hence $k^- > k^+ = 1$.

Consider now the flows $\MM^i$ and $\MM^\infty$ restricted to the time-interval $(-1+ \frac12 \eps, 2\eps^{-1})$ and note that the convergence restricted to this time-interval is smooth.
We can now deduce Property~\ref{Def_CNT_new_3} of Definition~\ref{Def_CNT_new} for large $i$ using Lemma~\ref{Lem_zeta_choice}\ref{Lem_zeta_choice_d}.
However, this contradicts \eqref{eq_origin_not_CNT} for large $i$.

\medskip

\textit{Case 3: $I^* \neq \emptyset$ and $\# I^* > 1$.} \quad
Set $t^* := \sup I^* \leq \infty$.
By Claim~\ref{Cl_twots} we know that $( \mu^\infty_t )_{t < t^*}$ is a stationary affine plane $\Sigma$ with constant  multiplicity $k^- \in \IN_0$.
So $\MM^\infty_t \subset \Sigma$ for all $t \in \IR$.
By Claim~\ref{Cl_limit_min_shr}, the restriction $( \mu^\infty_t )_{t > t^*}$ must have constant multiplicity $k^+ \in \IN_0$.
We can now argue as in Case~2.
\end{proof}
\bigskip

\begin{proof}[Proof of Theorem~\ref{Thm_L2_bound}.]
Fix $A, \eps$, assume without loss of generality that $(x_0, t_0) = (\vec 0, 0)$ and $0 < \tau_1 < 1=\tau_2 = \frac12 \tau_0$.
By restricting the flow, we may also assume that $\MM$ is defined on the time-interval $[-2,0)$.
Let us first argue that it is enough to establish an integral bound for $\mathcal{U} = \MM_{\reg} \setminus \MM_{\reg}^{(\eps,A)}$.
In fact, assuming that the desired bound \eqref{eq_int_bounded_TT} is true for this choice of $\mathcal{U}$, we can use Definition~\ref{Def_almost_regular}\ref{Def_almost_regular_3} to slightly increase $\mathcal{U}$, creating a subset that is open within $\MM$ and contains $\MM_{\sing}$, while only marginally increasing the integral on the right-hand side \eqref{eq_int_bounded_TT}.
So if the right-hand side of \eqref{eq_int_bounded_TT} is positive, then we can guarantee a similar bound for this larger subset $\mathcal{U}$, possibly after adjusting the constant $C$.
On the other hand, if the right-hand side of \eqref{eq_int_bounded_TT} equals zero, then the flow $\MM_{[-\tau_2,-\tau_1]}$ must be a shrinker, so since it is smooth at almost all time-slices, it must be smooth everywhere.
Thus $\MM_{\sing} = \mathcal{U} = \emptyset$ and the claim is obvious.

Let $\delta(A,\eps)$ be the constant from Lemma~\ref{Lem_close_to_min}.
By Vitali covering, we can choose countably many points
\[ (x_j, t_j) \in \bigcup_{t \in [-1, -\tau_1]} \Big( \big( \MM_{\reg}  \setminus ( \MM_{\reg}^{(\eps,A)} \cup \XX^{0,\delta} )\big)_t \cap B(\vec 0, A \sqrt{-t}) \Big) \times \{ t \} =: \MM' \]
such that for $r_j :=  r_{\loc}(x_j,t_j) < \delta \sqrt{|t_j|}$ we have
\begin{equation} \label{eq_Ps_cover}
 \bigcup_j P\big(x_j, t_j, \tfrac1{10} r_j \big) \supset \MM'
\end{equation}
and such that the parabolic balls $P\big(x_j, t_j, \tfrac1{100} r_j \big)$ are pairwise disjoint.

\begin{Claim} \label{Cl_bounded_js}
There is a constant $N(A,\delta) < \infty$ such that for any $(x,t) \in \MM_{[-2,- \frac12 \tau_1]}$ the number of indices $j$ for which the following two assumptions hold is bounded by $N$:
\begin{equation} \label{eq_xtinP}
 (x,t) \in P(x_j, t_j, \delta^{-1} r_j), 
\end{equation}
\begin{equation} \label{eq_Theta_diff_assp}
 \big| \Theta_{(x,t)} \big(  \delta^{-2} r_j \big) -\Theta_{(x,t)} \big(  \delta^2 r_j \big) \big| > \delta. 
\end{equation}
\end{Claim}

\begin{proof}
If $(x,t) \in \MM_{[-2, -\frac12 \tau_1]}$ satisfies \eqref{eq_xtinP} for at least one $j$, then we have a bound of the form $|x| \leq C(A) \sqrt{-t}$.
Fix such a $(x,t)$.
The bound $\Theta_{(\vec 0, 0)}(2) \leq A$ implies a bound of the form $\Theta_{(x,t)} (\tau) \leq C(A)$ for every $0 < \tau < 2+t$.
So by monotonicity of the Gaussian area there is a subset $S_{(x,t)} \subset \IN$ with $\#S_{(x,t)} \leq C(A,\delta)$ such that whenever \eqref{eq_Theta_diff_assp} holds, we have  $r_j \in [2^{-k-1},2^{-k}]$ for some $k \in S_{(x,t)}$.
For each such $k$, we can use the disjointness assumption of the  $P(x_j,t_j, \tfrac1{100} r_j)$, which implies disjointness of the balls $P(x_j,t_j, \tfrac1{50} 2^{-k-1})$, to obtain a bound on the indices $j$ for which \eqref{eq_xtinP} holds.
\end{proof}

For each index $j$ define
\[ \ZZ^j := \big\{ (x,t) \in P (x_j, t_j,  \delta^{-1} r_j) \;\; : \;\;
\big| \Theta_{(x,t)} (\delta^{-2} r_j) - \Theta_{(x,t)}(\delta^2 r_j) \big| \leq \delta \big\}.  \]
Claim~\ref{Cl_bounded_js} implies that any $(x,t) \in \MM_{[-2,- \frac12 \tau_1]}$ is contained in at most $N$ many of the subsets $P(x_j, t_j,  \delta^{-1} r_j) \setminus \ZZ^j$.
Consequently, we obtain via Lemma~\ref{Lem_entropy_rho}\ref{Lem_entropy_rho_b} that
\begin{multline} \label{eq_sum_bounded_Thetas}
   \sum_{j}  \int_{\max \{ t_j - \delta^{-2} r_j^2 ,-2\}}^{\min \{ t_j + \delta^{-2} r_j^2,  - \frac12 \tau_1 \}}  \int_{(\MM_{\reg} \setminus \ZZ^j)_t \cap B(x_j,  \delta^{-1} r_j)}  \bigg| \frac{\mathbf{x}^\perp}{2|\mathbf{t}|} + \mathbf{H} \bigg|^2 \rho_{(\vec 0,0)} \, d\HH^2 \, dt \\
\leq N \int_{-2}^{- \frac12 \tau_1}  \int_{\MM_{\reg,t}}\bigg| \frac{\mathbf{x}^\perp}{2|t|} + \mathbf{H} \bigg|^2 \rho_{(\vec 0,0)} \, d\HH^2 \, dt 
\leq N\big( \Theta_{(\vec 0, 0)} (2 \tau_2) - \Theta_{(\vec 0, 0)} ( \tfrac12\tau_1) \big).
\end{multline}

For each $j$, we may now consider the parabolic rescaling $r_j^{-1} (\MM - (x_j, t_j))$.
By the contrapositive of Lemma~\ref{Lem_close_to_min} applied to any such rescaling, we obtain
\begin{equation} \label{eq_delta_leq_int}
  \delta \, \frac{r_j^2}{|t_j|} \leq \int_{\max \{ t_j - \delta^{-2} r_j^2 ,-2\}}^{\min \{ t_j + \delta^{-2} r_j^2,  - \frac12 \tau_1 \}}  \int_{(\MM_{\reg} \setminus \ZZ^j)_t \cap B(x_j,  \delta^{-1} r_j)}  \bigg| \frac{\mathbf{x}^\perp}{2|\mathbf{t}|} + \mathbf{H} \bigg|^2 \rho_{(\vec 0,0)} \, d\HH^2 \, dt.
\end{equation}
Combining \eqref{eq_sum_bounded_Thetas} and \eqref{eq_delta_leq_int} implies
\begin{equation} \label{eq_sum_rjtj}
 \sum_{j} \frac{r_j^2}{|t_j|}  \leq \frac{N}{\delta} \big( \Theta_{(\vec 0, 0)} (2 \tau_2) - \Theta_{(\vec 0, 0)} ( \tfrac12 \tau_1) \big). 
\end{equation}
On the other hand, we have by Lemma~\ref{Lem_rloc_Lipschitz}
\[ \int_{t_j- (\frac1{10} r_j)^2}^{t_j+ (\frac1{10} r_j)^2} \frac1{|t|} \int_{(\MM_{\reg} \cap P(x_j, t_j \frac1{10} r_j))_t} r_{\loc}^{-2}  \, d\HH^2 \, dt
\leq C \frac{r_j^2}{|t_j|}. \]
Summing this bound over $j$ and combining the result with \eqref{eq_Ps_cover} and \eqref{eq_sum_rjtj} implies the desired bound.
\end{proof}
\bigskip

\section{The main argument} \label{sec_main_argument}
In this section we prove the main theorems of this paper, as stated in Subsection~\ref{subsec_mainresults}.
We will proceed as follows.
In Subsection~\ref{subsec_cutoff}, we first construct a cutoff function $\eta$ on a given almost regular mean curvature flow.
Roughly speaking, this cutoff function decomposes the flow into regions where the flow can be understood in terms of the separation function $\s$ and in terms of the local scale function $r_{\loc}$, respectively.
In Subsection~\ref{subsec_mod_sep}, we introduce a modification $\s_{\loc}$ of the separation function, which has some improved properties.
In Subsection~\ref{subsec_integral_estimate}, we define the integral quantities $\mathfrak{S}_{(x_0,t_0)}(\tau)$ and $\mathfrak{S}_{(x_0,t_0),a}(\tau)$, which roughly measure the integral of $$(1-\eta) \log(\s_{\loc}) + \eta \log (\td r_{\loc}) \approx \log (r_{\loc})$$ against the ambient Gaussian weight.
We will also derive an integral estimate, which allows us to bound the decay of $\mathfrak{S}_{(x_0,t_0)}(\tau)$ as we decrease $\tau$.
This will then lead to a proof of the key estimate, Theorem~\ref{Thm_key_thm}, in Subsection~\ref{subsec_proof_key}.

We remark that most of our estimates are invariant under parabolic rescaling.
However, in order to reduce notational complexity, part of our discussion will focus on the quantity $\mathfrak{S}_{(x_0,t_0)}(\tau)$, which is not scaling invariant.
Once we have obtained the central estimate for this quantity, Proposition~\ref{Prop_step_bound}, we will deduce from this an estimate on the scaling invariant quantity $\mathfrak{S}_{(x_0,t_0), \sqrt{\tau}}(\tau)$ by parabolic rescaling in the proof of Proposition~\ref{Prop_step_bound_improved}.

\subsection{A global cutoff function} \label{subsec_cutoff}
Our initial objective is to construct an appropriate cutoff function $\eta$ on an almost regular mean curvature flow $\MM$.
On the set $\{ \eta < 1 \}$ the separation of the flow can be understood in terms of either the separation function $\s$ or the smoothed local scale function $\td r_{\loc}$.
Specifically, we will be able to ensure that either $\square \log \s \geq 0$ or $\square \log \td r_{\loc} \geq 0$ on this set.
The support $\supp \eta$, on the other hand, consists of two parts:
\begin{itemize}
\item The first part, denoted by $\YY$, corresponds to regions where the flow is locally modeled on stationary minimal surfaces.
Within these regions, we can achieve robust quantitative control over the functions $\s$ and $\td r_{\loc}$, leveraging the insights derived from these minimal surface models.
\item 
In the second part, the local geometry significantly deviates from that of a minimal surface, and our knowledge of the precise geometric details is limited. 
Despite this lack of refined geometric control, we can still derive an integral bound on $r_{\loc}^{-2}$, thanks to  Theorem~\ref{Thm_L2_bound}.
\end{itemize}
In both of these estimates, we need to exclude the set of points $\XX^{t_0, \delta}$ (compare with \eqref{eq_XXt0delta}), where the local scale function $r_{\loc}$ is already sufficiently well bounded from below.

\begin{Lemma} \label{Lem_asspt_eta_rho}
There is a universal constant $C_0 < \infty$ such that the following is true.
Consider constants $\alpha > 0$ and $A < \infty$ and a bounded almost regular mean curvature flow $\MM \subset \IR^3 \times I$ with 
\begin{equation} \label{eq_Thgeeta}
 \Theta(\MM) \leq A, \qquad \genus (\MM)  \leq A. 
\end{equation}
Then there is a continuous function $\eta : \MM \to [0,1]$ and a closed subset $\YY \subset \MM_{\reg}$ with the following properties: 
\begin{enumerate}[label=(\alph*)]
\item \label{Lem_asspt_eta_rho_a} We have $\supp (1-\eta) \subset \MM_{\reg}$ and $\eta |_{\MM_{\reg}}$ is locally Lipschitz and satisfies the following  derivative estimates almost everywhere on $\MM_{\reg}$
\[ |\nabla \eta| \leq C_0 r_{\loc}^{-1}, \qquad
 |\partial_t \eta| \leq C_0 r_{\loc}^{-2}. \]
 Here $\nabla$ denotes the spatial gradient and $\partial_t$ denotes the normal time-derivative.
\item \label{Lem_asspt_eta_rho_b} 
For any $(x,t) \in \supp (1-\eta)$ the following is true:
\begin{enumerate}[label=(b\arabic*)]
\item \label{Lem_asspt_eta_rho_b1} If $\tdr_{\loc}(x,t) \geq \frac1{10} \s(x,t)$, then $(x,t) \in \MM_{\good}$, so near $(x,t)$ on $\MM$ we have in the viscosity sense
\[ \square \big( \log  \s  \big) \geq 0. \]
\item \label{Lem_asspt_eta_rho_b2} If $\tdr_{\loc}(x,t) \leq \frac12 \s(x,t)$, then
\[ \square \big( \log  \tdr_{\loc}  \big)(x,t) > 0. \]
\end{enumerate}
\item \label{Lem_asspt_eta_rho_bb} On the closure of $\YY \cap \{ 0 < \eta < 1 \}$ we have
\[ \td r_{\loc} \geq 0.8 \s. \]
\end{enumerate}
Let us now consider a point $(x_0,t_0) \in \IR^3 \times \ov{I}$ and $0 < \tau_1 < \tau_2$ with $t_0 - 2\tau_2 \in  I$ and $\tau_2 \geq 2 \tau_1$.
Set, as in \eqref{eq_XXt0delta},
\[ \XX^{t_0,\delta} := \big\{ (x,t) \in \MM_{< t_0} \;\; : \;\;  r_{\loc}(x,t) \geq \delta \sqrt{t_0 -t} \big\}. \]
Then the following bounds hold:
\begin{enumerate}[label=(\alph*), start=4]
\item \label{Lem_asspt_eta_rho_c} There is a constant $\delta(\alpha,A) > 0$ such that
\begin{multline} \label{eq_lem_bound_over_YY}
\qquad \int_{t_0-\tau_2}^{t_0-\tau_1} \frac1{t_0 - t}  \int_{(\YY \setminus \XX^{t_0,\delta})_t \cap B(x_0, A \sqrt{t_0-t})} \bigg( |\partial_t \eta| 
+ \frac{|\partial_t \tdr_{\loc} |}{\tdr_{\loc}}
+ \frac{|\nabla \eta | \, |\nabla \s |}{\s} + \frac1{(t_0 - t)^{1/2} \tdr_{\loc}}
\bigg) \, d\HH^2 \, dt \\
 \leq   \alpha \log \bigg( \frac{\tau_2}{\tau_1} \bigg).
\end{multline}
\item \label{Lem_asspt_eta_rho_d} For any $D \in [A, \infty)$ there are constants $\delta'(\alpha, D) > 0$ and $C'(\alpha, D) < \infty$ such that
\begin{multline}  \label{eq_lem_bound_over_complement}
\qquad\qquad \int_{t_0-\tau_2}^{t_0-\tau_1} \frac1{t_0 - t} \int_{(\supp \eta \setminus (\XX^{t_0,\delta'} \cup\YY))_t \cap \MM_{\reg,t} \cap B(x_0, D \sqrt{t_0-t})} r_{\loc}^{-2}   \, d\HH^2 \, dt \\
 \leq C' \big( \Theta_{(x_0,t_0)}(2\tau_2) - \Theta_{(x_0,t_0)}( \tfrac12 \tau_1) \big) .
\end{multline}
\end{enumerate}
\end{Lemma}

\begin{proof}
Fix $A,\alpha$ and let $\eps,\delta > 0$ be constants whose values we will determine in the course of the proof.
Let $\mathcal{U} \subset \MM$ be the neighborhood of $\MM \setminus \MM^{(\eps, A)}_{\reg}$ from Theorem~\ref{Thm_L2_bound}.
We will define $\eta := \max \{ \eta', \eta'' \}$, where $\eta'$ will be a function that is supported on $\mathcal{U}$ and $\eta''$ a function that is supported on a subset $\YY$, which we will also construct.
We first construct $\eta'$ using the following claim.

\begin{Claim} \label{Cl_etap}
Given $\eps \leq \ov\eps$, there is a continuous function $\eta' : \MM \to [0,1]$ with the following properties:
\begin{enumerate}[label=(\alph*)]
\item  \label{Cl_etap_a} $\eta'|_{\MM_{\reg}}$ is locally Lipschitz.
\item \label{Cl_etap_b} $|\nabla \eta'| \leq 10^3 r_{\loc}^{-1}$ and $|\partial_t \eta'| \leq 10^3 r_{\loc}^{-2}$ almost everywhere on $\MM_{\reg}$.
\item \label{Cl_etap_c} $\supp \eta' \subset \mathcal{U}$.
\item \label{Cl_etap_d} $\supp (1-\eta') \subset \MM_{\reg}$ and any $(x,t) \in \supp (1-\eta')$ satisfies Assertion~\ref{Lem_asspt_eta_rho_b} of the lemma or Property~\ref{Def_CNT_new_2} of Definition~\ref{Def_CNT_new} with $\eps$ replaced by $10\eps$.
\end{enumerate}
\end{Claim}

\begin{proof}
Let $\varphi : \IR \to [0,1]$ be a smooth cutoff function with $\varphi \equiv 1$ on $[-10^{-2}, 10^{-2}]$ and $\supp \varphi \subset [-10^{-1}, 10^{-1}]$ and $|\varphi'| \leq 10^3$.
For any $(x,t) \in \IR^3 \times \IR$ define
\[ \eta'_{(x,t)} (x',t') := \varphi \bigg( \frac{|x-x'|}{r_{\loc}(x,t)} \bigg) \varphi \bigg( \frac{t'-t}{r^2_{\loc}(x,t)} \bigg). \]
Set $\eta'(x',t') := 1$ if $(x',t') \in \MM_{\sing}$.
If $(x',t') \in \MM_{\reg}$, set
\[ \eta'(x',t') := \sup \eta'_{(x,t)} (x',t'), \]
where the supremum is taken over all $(x,t) \in \MM_{\reg}$ with $ \MM \cap P(x,t,\frac12  r_{\loc}(x,t))  \subset \mathcal{U}$.
If there is no such $(x,t)$, then we set $\eta' :\equiv 0$.
Note that each function $\eta'_{(x,t)}$ has support on $P(x,t, \frac1{10} r_{\loc}(x,t)) \cap \MM$ and satisfies the derivative bounds from Assertions~\ref{Cl_etap_b} of this claim.
So since $\MM_{\sing} \subset \mathcal{U}$ and since $r_{\loc} \to 0$ near $\MM_{\sing}$, we find that $\eta'$ is continuous and satisfies Assertions~\ref{Cl_etap_a}--\ref{Cl_etap_c} of the claim.
For Assertion~\ref{Cl_etap_d} note that if $\eta' (x,t)<1$, then by the definition of $\eta'$, we must have $(x,t) \in \MM_{\reg}$ and $ \MM \cap P(x,t, \frac12 r_{\loc}(x,t)) \not\subset \mathcal{U}$.
So we can choose a point $$(x^*,t^*) \in  (\MM \setminus \mathcal{U}) \cap P\big(x,t, \tfrac12 r_{\loc}(x,t) \big)  \subset \MM^{(\eps, A)}_{\reg}.$$
By Lemma~\ref{Lem_rloc_Lipschitz} we have $\frac12 r_{\loc}(x,t) \leq r_{\loc}(x^*,t^*) \leq 2 r_{\loc}(x,t)$.

Suppose first that $(x^*,t^*)$ satisfies Property~\ref{Def_CNT_new_2} of Definition~\ref{Def_CNT_new}.
Then the same property is true for $(x,t)$ after replacing $\eps$ by $10\eps$.

Next, suppose that $(x^*,t^*)$ satisfies Property~\ref{Def_CNT_new_3} of Definition~\ref{Def_CNT_new}. 
Since $(x^*,t^*) \in \MM  \lb \cap  P(x,t, \lb \tfrac12 r_{\loc}(x,t) ) $, the points $(x,t)$ and $(x^*,t^*)$ can be connected by a path within $$\MM \cap P \big(x,t, \lb \tfrac12 r_{\loc}(x,t) \big)  \subset \MM \cap  P\big(x^*,t^*, 10 r_{\loc}(x^*,t^*) \big) .$$
So for sufficiently small $\eps$, Property~\ref{Def_CNT_new_3} of Definition~\ref{Def_CNT_new} implies that the component of $\MM \cap P(x,t, \lb 2 r_{\loc}(x,t)$ containing $(x,t)$ is sufficiently close to an open subset of an affine plane such that we have $(x,t) \in \MM_{\good}$.
This implies Assertion~\ref{Lem_asspt_eta_rho_b1} of the lemma.
Assertion~\ref{Lem_asspt_eta_rho_b2} follows from the last statement in Property~\ref{Def_CNT_new_3} of Definition~\ref{Def_CNT_new}.
\end{proof}
\medskip

Let $\YY' \subset \MM$ be the set of points $(x,t) \in \supp (1-\eta')$ that do not satisfy Assertion~\ref{Lem_asspt_eta_rho_b} of the lemma.
So any $(x,t) \in \YY'$ is regular and satisfies Property~\ref{Def_CNT_new_2} of Definition~\ref{Def_CNT_new} with $\eps$ replaced by $10\eps$.
We will now describe the flow in a neighborhood of $\YY'$.
We will show that there are annular regions around each component of $\YY'$, where the flow is close to the graph of a multivalued function.
For the following claim let $\beta > 0$ be a constant whose value we will determine later.
We will frequently consider parabolic balls and annuli of the form
\begin{align*}
P_{\beta} (x,t,r) &:= B(x,  r) \times [t - \beta^{-1} r^2, t+ \beta^{-1} r^2], \\
 A_{\beta} (x,t,r) &:= A(x,10^{-3}r,  r) \times [t - \beta^{-1} r^2, t+ \beta^{-1} r^2] . 
\end{align*}

\begin{Claim} \label{Cl_annulus}
There are constants $B (A, \beta) < \infty$ and $\eps  (A, \beta) > 0$ such that the following is true.
For any $(x,t) \in \YY'$ and $r := r_{\loc}(x,t)$ there is a number $b \in [10^5,B]$ with the following properties:
\begin{enumerate}[label=(\alph*)]
\item \label{Cl_annulus_a} We have 
\begin{align} \label{eq_Pb_reg}
  \MM \cap P_\beta (x, t, b  r)   &\subset \MM_{\reg},  \\
 \MM \cap A_\beta(x,t, br)   &\subset \MM_{\reg} \setminus \YY'. \label{eq_Ab_good}
\end{align}
\item \label{Cl_annulus_b} We have the following bounds:
\begin{alignat*}{2}
\frac{|\partial_t \tdr_{\loc} |}{\tdr_{\loc}} &\leq \frac{\beta}{ (br)^{2}}, \quad |\mathbf{H}| \leq \frac{\beta}{ br} \textQq{and} \tdr_{\loc} \geq \frac{r}{b} &&\qquad\quad \text{on} \quad \MM \cap P_\beta(x,t,br) , \\
\frac{ |\nabla \s |}{\s} &\leq \frac{\beta}{ br} \textQq{and}  \td r_{\loc} \geq 0.8 \s && \qquad\quad \text{on} \quad \MM \cap A_\beta(x,t,br) .
\end{alignat*}
\item \label{Cl_annulus_c} We have the bound
\[ \int_{t - 10^{-6} \beta^{-1} ( b r)^2}^{t + 10^{-6} \beta^{-1}  (b r)^2} \int_{\MM_t \cap B(x, 10^{-3} b r) } |A|^2 d\HH^2 \, dt \geq 10^{-6} \pi \beta^{-1} (br)^2. \]
\end{enumerate}
\end{Claim}

\begin{proof} 
Fix $A, \beta$ and take a sequence of counterexamples $\MM^i$ and $(x_i,t_i) \in \YY^{\prime, i}$ with $B_i \to \infty$ and $\eps_i \to 0$.
By definition, the flows $\MM^{\prime,i} := r_{\loc}^{-1}(x_i,t_i) (\MM^i - (x_i,t_i))$ are $10\eps_i$-close to the stationary flow of a minimal surface $\Sigma_i$, which passes through the origin and satisfies the bounds \eqref{eq_Theta_genus_bound}.

Consider the intrinsic exponential maps of $\Sigma_i$ based at the origin and apply Lemma~\ref{Lem_min_surf_conv_or_plane}.
Suppose first that, after passing to a subsequence, these exponential maps smoothly converge to the exponential map of a plane.
Since by construction we have $r_{\loc}(\vec 0, 0) = 1$ for the flows $\MM^{\prime,i}$, we can thus find a sequence of times $t'_i \in \IR$ with $\limsup_{i \to \infty} |t'_i| \leq 1$ such that for large $i$ the intersection $B(\vec 0, 1.1) \cap \MM^{\prime,i}_{t'_i}$ is disconnected.
Since $\MM^{\prime,i}$ is a small perturbation of a \emph{stationary} flow, this implies that for large $i$ the intersection $B(\vec 0, 1.2) \cap \MM^{\prime,i}_{0}$ is disconnected and hence we must have $\s(\vec 0, 0) \leq 1.2$ for $\MM^{\prime,i}$. 
So for large $i$ the bound 
\[ \tdr_{\loc}(x_i,t_i) \geq 0.9 r_{\loc}(x_i,t_i) \geq \frac{0.9}{1.2} \s(x_i,t_i) \geq 0.6 \s(x_i,t_i) \]
 holds for the original flow $\MM^i$.
Moreover, by the smooth convergence of the exponential map, we have $(x_i,t_i) \in \MM^i_{\good}$ for large $i$.
But this implies Assertion~\ref{Lem_asspt_eta_rho_b} for large $i$, in contradiction to the assumption that $(x_i,t_i) \in \YY^{\prime,i}$.

It follows that the intrinsic exponential maps of $\Sigma_i$, based at the origin, do not converge to the exponential map of a  plane.
Thus by Lemma~\ref{Lem_min_surf_conv_or_plane} we may pass to a subsequence and obtain smooth convergence $\Sigma_i \to \Sigma_\infty \subset \IR^3$, where $\Sigma_\infty$ is a non-trivial, connected and properly embedded minimal surface of finite total curvature, which passes through the origin and satisfies \eqref{eq_Theta_genus_bound}.

By Lemmas~\ref{Lem_ends_min_surf}, \ref{Lem_ends_min_surf_improved},  we can find a number $b \geq 10^5$ with the following properties: 
\begin{alignat*}{2}
    \td r_{\loc} &\geq \frac{10}{b} &&\textQQq{on}  \Sigma_\infty \cap B(\vec 0,  10 b) , \\
 \frac{|\nabla \s|}{\s} &\leq  \frac{ \beta}{10 b} &&\textQQq{on}  \Sigma_\infty \cap A(\vec 0, 10^{-4} b, 10 b)  \subset \Sigma_{\infty,\good},
\end{alignat*}
\[ \int_{\Sigma_\infty \cap B(\vec 0, 10^{-4} b)} |A|^2 d\HH^2 \geq \pi. \]
Note that on the associated stationary flow we have $r_{\loc} = \s$ on $A(\vec 0, 10^{-4} br, 10 br)$, so $\td r_{\loc} \geq 0.9 \s$ by Lemma~\ref{Lem_zeta_choice}\ref{Lem_zeta_choice_a}.
So Assertions~\ref{Cl_annulus_b}, \ref{Cl_annulus_c} hold for large $i$.
Here we have used Lemma~\ref{Lem_zeta_choice}\ref{Lem_zeta_choice_dd}; recall also that $b \leq B_i$ for large $i$.
To see that Assertion~\ref{Cl_annulus_a} holds for large $i$, note that \eqref{eq_Pb_reg} holds for large $i$ by smooth convergence.
Similarly, since $\Sigma_\infty \cap A(\vec 0, 10^{-4} br, 10 br)  \subset \Sigma_{\infty,\good}$, we obtain that for large $i$, Assertion~\ref{Lem_asspt_eta_rho_b1} of the lemma holds on $\MM \cap A_\beta(x,t, br)$ and Assertion~\ref{Lem_asspt_eta_rho_b2} is not applicable due to the bound $\td r_{\loc} \geq 0.8 \s$.
Thus \eqref{eq_Ab_good} holds for large $i$.
Therefore, we obtain a contradiction to the choice of the flows $\MM^i$ for large $i$, which finishes the proof of the claim.
\end{proof}
\medskip

We can now construct the function $\eta''$.
Let $\psi : \IR \to [0,1]$ be a cutoff function with $\psi \equiv 1$ on $[-10^{-1},10^{-1}]$ and $\supp \psi \subset (-1,1)$ and $|\psi'| \leq 10^2$.
For any $(x,t) \in \YY'$ and any $b \in [10^5,B]$ satisfying the assertions of Claim~\ref{Cl_annulus} we set
\[ \eta''_{(x,t),b} (x',t') := \psi \bigg( \frac{|x-x'|}{b \, r_{\loc}(x,t)} \bigg) \psi \bigg( \frac{t'-t}{\beta^{-1} (b r_{\loc}(x,t))^2} \bigg). \]
By Vitali Covering, we can find finitely many points $(x_j, t_j) \in \YY'$ and constants $b_j \in [10^5, B]$ satisfying Claim~\ref{Cl_annulus} such that for $r_j := r_{\loc}(x_j,t_j)$ the interiors of the parabolic balls $P_\beta (x_j, t_j, 10^{-1} b_j r_j)$ cover the closure of $\YY'$ and such that the parabolic balls $P_\beta (x_j, t_j, 10^{-3} b_j r_j)$ are pairwise disjoint. 
(Note that Vitali covering is not sensitive to the choice of $\beta$ as we can reparameterize the time-coordinate).
Observe that we have $\eta''_{(x_j,t_j),b_{j}} \equiv 1$ on $P_\beta (x_j, t_j, 10^{-1} b_j r_j)$ and $\supp \eta''_{(x_j,t_j),b_{j}} \subset P_\beta(x_j,t_j, b_jr_j)$.
So if we set
\[ \eta'' := \max_{j} \eta''_{(x_j,t_j),b_j}, \]
then $\eta'' \equiv 1$ on a neighborhood of $\YY'$ and
\[ \supp \eta'' \subset \bigcup_j P_\beta(x_j,t_j, b_jr_j)  . \]
Next, observe that by Lemma~\ref{Lem_rloc_Lipschitz}
\[ r_{\loc} \leq 10^3 \beta^{-1/2} b_jr_j \textQQqq{on} P_\beta (x_j, t_j, b_jr_j) , \]
which implies, using Claim~\ref{Cl_annulus}\ref{Cl_annulus_b}, that
\begin{align*}
|\nabla \eta''_{(x_j,t_j),b_j} | &\leq b^{-1} r_j^{-1} \leq 10^3 r^{-1}_{\loc},  \\
 |\partial_t \eta''_{(x_j,t_j),b_j} | &\leq \beta b^{-2} r_j^{-2} + |\mathbf{H}| b_j^{-1} r_j^{-1} \leq 2\beta b^{-2} r_j^{-2} \leq 10^7 r^{-2}_{\loc}.
\end{align*}
So
\begin{equation} \label{eq_etapp_loc_der}
|\nabla \eta''| \leq 10^3 r_{\loc}^{-1}, \qquad |\partial_t \eta''| \leq 10^7 r_{\loc}^{-2}.
\end{equation}
In particular, $\eta''$ is locally Lipschitz.
Let us now apply again Claim~\ref{Cl_annulus}\ref{Cl_annulus_b}.
From the corresponding bounds on $\eta''_{(x_j,t_j),b_j}$, we obtain that for almost all $(x',t') \in \supp \eta''$
\begin{multline} \label{eq_etar_bound_ae}
 \bigg( |\partial_t \eta''| 
+ \frac{|\partial_t \tdr_{\loc} |}{\tdr_{\loc}}
+ \frac{|\nabla \eta'' | \, |\nabla \s |}{\s} + \frac1{(t_0 - \mathbf{t})^{1/2} \tdr_{\loc}}
\bigg) (x',t') \\
 \leq  \min_{(x',t') \in P_\beta(x_j,t_j,b_jr_j)}  \bigg( 10^3 \beta(b_j r_j)^{-2} + \frac{b_j}{(t_0 - t')^{1/2}r_j} \bigg),
\end{multline}
where the minimum is taken over all $j$ such that $(x',t') \in P_\beta(x_j,t_j,b_jr_j)$.

We define
\[ \eta := \max \{ \eta', \eta'' \}, \qquad \YY := \{ \eta' \leq \eta'' \}. \]
Then Assertion~\ref{Lem_asspt_eta_rho_a} of the lemma follows from Claim~\ref{Cl_etap}\ref{Cl_etap_b} and \eqref{eq_etapp_loc_der}.
Next, note that $\eta = \eta'$ on $\MM \setminus \YY$.
Since $\eta'' \equiv 1$ in a neighborhood of the closure of $\YY'$, we obtain $\YY' \subset \YY$ and Assertion~\ref{Lem_asspt_eta_rho_b} via the definition of $\YY'$ and Claim~\ref{Cl_etap}\ref{Cl_etap_d}.
Assertion~\ref{Lem_asspt_eta_rho_bb} follows from the fact that $\eta = \eta''$ on $\YY$ and the last inequality of Assertion~\ref{Cl_annulus_b} of Claim~\ref{Cl_annulus}.
Consider now the situation from Assertions~\ref{Lem_asspt_eta_rho_c}, \ref{Lem_asspt_eta_rho_d} of the lemma.
We first observe that by Claim~\ref{Cl_etap}
\[ \supp \eta \setminus (\XX^{t_0,\delta} \cup \YY) \subset \supp \eta' \setminus \XX^{t_0,\delta}
\subset \mathcal{U} \setminus \XX^{t_0,\delta}. \]
Therefore, the bound in Assertion~\ref{Lem_asspt_eta_rho_d} follows from Theorem~\ref{Thm_L2_bound} and the constants $\delta', C'$ can be determined solely based on $D, A, \eps$.
At the end of the proof we will determine $\eps$ solely based on $A,\alpha$, which will establish Assertion~\ref{Lem_asspt_eta_rho_d}.

To show Assertion~\ref{Lem_asspt_eta_rho_c} we use the bound \eqref{eq_etar_bound_ae}.
Let $J$ be the set of indices $j$ for which there are $x' \in B(x_0, A \sqrt{t_0 - t'})$, $t' \in [t_0 - \tau_2, t_0 - \tau_1]$ such that $r_{\loc}(x',t') < \delta \sqrt{t_0 - t'}$ and $(x',t') \in P_\beta(x_j,t_j, b_jr_j)$.
By Claim~\ref{Cl_annulus} we have for all $j \in J$
\begin{equation} \label{eq_rjbjdelta}
 r_j \leq b_j \td r_{\loc}(x',t') \leq 1.1 b_j  r_{\loc}(x',t')
<1.1 b_j \delta \sqrt{t_0 - t'}. 
\end{equation}
So if $\delta \leq \ov\delta(\beta, B(A,\beta))$, then we may assume the following
\begin{itemize}
\item $P_\beta(x_j,t_j, b_jr_j)$ is contained in the time-slab $\IR^3 \times [t_0 - 2\tau_2, t_0 - \frac12 \tau_1]$.
\item We have
\[ \tfrac12(t_0 - t_j) \leq t_0 - \mathbf{t} \leq 2(t_0 - t_j) \textQQq{on} P_\beta(x_j,t_j, b_jr_j). \]
\item We have 
\[ (P_\beta(x_j, t_j, b_jr_j))_t \subset B\big(x_0, (A+1) \sqrt{t_0 - t} \big) \textQQq{for all}  t' \in [t_0 - 2\tau_2, t_0 - \tfrac12 \tau_1]. \]
\end{itemize}
Note also that due to \eqref{eq_Thgeeta} we have an upper area bound of the form 
\[ \HH^2 \big( ( \MM \cap P_\beta(x_j,t_j, b_jr_j) )_t \big) \leq C(A) (b_jr_j)^2. \]
So \eqref{eq_etar_bound_ae}, \eqref{eq_rjbjdelta} and Claim~\ref{Cl_annulus}\ref{Cl_annulus_c} imply that the left-hand side of \eqref{eq_lem_bound_over_YY} is bounded by
\begin{align*} \label{eq_etass_bound}
 C(A)  \sum_{j \in J} &\bigg( \frac{10^3  (b_j r_j)^2 }{t_0 - t_j} + \frac{\beta^{-1} b_j^5 r_j^3}{(t_0 - t_j)^{3/2}} \bigg)  \\
&\leq C(A)  \sum_{j \in J} \frac1{t_0 - t_j} \big(  (b_j r_j)^2  + \beta^{-1} b_j^6 \delta r_j^2 \big)  \\ \displaybreak[1]
&\leq C(A) (\beta + \delta B^4) \sum_{j \in J} \frac{\beta^{-1} (b_j r_j)^2}{t_0 - t_j}   \\ \displaybreak[1]
&\leq \frac{C(A)}{10^{-6}\pi} (\beta + \delta  B^4)  \sum_{j \in J}  \frac1{t_0-t_j} \int_{t_j - 10^{-6} \beta^{-1} ( b_j r_j)^2}^{t_j + 10^{-6} \beta^{-1}  (b_j r_j)^2} \int_{B(x_j, 10^{-3} b_j r_j) \cap \MM_{t}} |A|^2 d\HH^2 \, dt \\ \displaybreak[1]
&\leq C(A) (\beta + \delta  B^4)  \int_{t_0 - 2\tau_2}^{t_0 - \frac12 \tau_1} \frac1{t_0-t} \int_{M_{\reg,t} \cap B(x_0, (A+1) \sqrt{t_0 - t})} |A|^2 \,d\HH^2 \, dt \\ \displaybreak[1]
&\leq C(A) (\beta + \delta  B^4)  \log \bigg( \frac{4\tau_2}{\tau_1} \bigg) \\ \displaybreak[1]
&\leq 3 C(A) (\beta + \delta  B^4)  \log \bigg( \frac{\tau_2}{\tau_1} \bigg).
\end{align*}
In the last two steps we have used Lemma~\ref{Lem_intA2bound} and the fact that $\tau_2 \geq 2 \tau_1$.
We can now choose $\beta (A,\alpha) > 0$ such that $3 C(A) \beta \leq \frac12 \alpha$ and then $\delta (A,\alpha) > 0$ such that $C(A) \delta (B(A,\beta))^4 \leq \beta$.
Lastly, we choose $\eps(A,\beta) > 0$ according to Claim~\ref{Cl_annulus}.
This concludes the proof of Lemma~\ref{Lem_asspt_eta_rho}.
\end{proof}
\bigskip

\subsection{Modification of the separation function} \label{subsec_mod_sep}
In the following we will analyze the evolution of the separation function $\s$ on $\supp (1-\eta)$, where $\eta$ is the cutoff function constructed in Lemma~\ref{Lem_asspt_eta_rho}.
We will exploit the fact that $\square \log \s \geq 0$, as established in Proposition~\ref{Prop_log_u_super_sol}, which holds on $\MM_{\good} \cap \supp (1-\eta)$.
However, this bound may not hold at points that are not good, and it's possible that $\supp (1-\eta)$ contains such points. 
By Lemma~\ref{Lem_asspt_eta_rho}\ref{Lem_asspt_eta_rho_b}, we have $\td r_{\loc} < \frac1{10} \s$  at such points, so a similar bound, $\square \log \td r_{\loc} \geq 0$, holds for the smoothed local scale function.
To merge both bounds into a single one, we define the quantity
\[ \s_{\loc} := \min \big\{ \tfrac12 \s,   \tdr_{\loc} \big\} \]
and for any constant $a > 0$, thought of as having the dimension of length, we set
\[  r_{\loc,a} := \min \big\{  r_{\loc}, a \big\} , \qquad \tdr_{\loc,a} := \min \big\{  \tdr_{\loc}, a \big\} , \qquad
 \s_{\loc,a} := \min \big\{ \s_{\loc}, a \big\}. \]
Note that since $r_{\loc} \leq \s$, which follows directly from the definitions of $r_{\loc}, \s$, we have by Lemma~\ref{Lem_zeta_choice}
\begin{equation} \label{eq_r_s_comparable}
   \tfrac14  \tdr_{\loc,a}\leq \s_{\loc,a} \leq  \tdr_{\loc,a} \leq a. 
\end{equation}
We document the following crucial observation about Lemma~\ref{Lem_asspt_eta_rho}.

\begin{Lemma} \label{Lem_mod_sep_fct} 
In Assertion~\ref{Lem_asspt_eta_rho_b} of Lemma~\ref{Lem_asspt_eta_rho} the following bound holds in the viscosity sense on $\supp (1-\eta)$
\begin{equation} \label{eq_sq_sloc_a}
   \square \big( \log \s_{\loc,a}  \big) \geq 0. 
\end{equation}
Moreover, in Assertion~\ref{Lem_asspt_eta_rho_bb} we have $\s_{\loc,a} = \min \{ \frac12 \s, a \}$ on the closure of $\YY \cap \{ 0 < \eta < 1 \}$.
\end{Lemma}

\begin{proof}
Let $(x,t) \in \supp (1-\eta)$.
By Assertion~\ref{Lem_asspt_eta_rho_b} of Lemma~\ref{Lem_asspt_eta_rho} the following is true.
If $\td r_{\loc}(x,t) \geq \frac1{2} \s (x,t)$, then $(x,t) \in \MM_{\good}$, which implies $\square (\log \s) \geq 0$ at $(x,t)$ in the viscosity sense by Proposition~\ref{Prop_log_u_super_sol}.
On the other hand, if $\td r_{\loc}(x,t) \leq \frac1{2} \s (x,t)$, then $\square (\log \td r_{\loc})(x,t) \geq 0$.
The bound \eqref{eq_sq_sloc_a} now follows from the fact that the minimum of three viscosity supersolutions is a viscosity supersolution.
The last statement is a direct consequence of Proposition~\ref{Lem_asspt_eta_rho}\ref{Lem_asspt_eta_rho_bb} and the definition of $\s_{\loc}$.
\end{proof}

\bigskip

\subsection{An integral estimate} \label{subsec_integral_estimate}
We will now define an integral quantity involving $\s_{\loc,a}$, $\td r_{\loc,a}$ and $\eta$, whose dependence in time can be controlled.
In the following consider a bounded almost regular mean curvature flow $\MM \subset \IR^3 \times I$.
Fix some constants $A, \alpha$ such that \eqref{eq_Thgeeta} holds and apply Lemma~\ref{Lem_asspt_eta_rho} to obtain the continuous function $\eta : \MM \to [0,1]$, which is locally Lipschitz on $\MM_{\reg}$, and the closed subset $\YY \subset \MM_{\reg}$.

\begin{Definition} \label{Def_S}
For any $(x_0, t_0) \in \IR^3 \times \IR$ and $\tau, a > 0$ with $t_0 - \tau \in I$ we set
\begin{equation} \label{eq_SS_def}
 \mathfrak{S}_{(x_0,t_0), a}(\tau) := \int_{\MM_{\reg,t_0-\tau}}  \bigg( (1-\eta) \, \log \Big( \frac{\s_{\loc,a}}{a} \Big) + \eta \log \Big( \frac{ \tdr_{\loc,a}}{a} \Big) \bigg) \, \rho_{(x_0,t_0)}\, d\HH^2. 
\end{equation}
As we will often consider the case $a = 1$, we set
\[ \mathfrak{S}_{(x_0,t_0)}(\tau) := \mathfrak{S}_{(x_0,t_0), 1}(\tau). \]
\end{Definition}
\medskip

Note that the integrand in \eqref{eq_SS_def} is non-positive by definition.
Also $\mathfrak{S}_{(x_0,t_0), a}(\tau)$ is dimensionless and behaves as follows under parabolic rescaling:
\[ \mathfrak{S}^{\lambda \MM}_{(\lambda x_0, \lambda^2 t_0), \lambda a}( \lambda^2 \tau) = \mathfrak{S}_{(x_0,t_0), a}(\tau). \]
We will use the following estimate to bound the change of $\mathfrak{S}_{(x_0,t_0)}(\tau)$.
We remark that a more general estimate is also true for $\mathfrak{S}_{(x_0,t_0), a}(\tau)$.
This can be obtained via parabolic rescaling.
We will only need the case $a = 1$.

\begin{Lemma} \label{Lem_evol_formula}
Consider the situation described above.
Let $0 < \tau_1 < \tau_2 \leq 1$ with $[t_0 - 2\tau_2, t_0 - \tau_1] \subset I$ and $\tau_2 \geq 2 \tau_1$.
Suppose that $t_0 - \tau_2$ and $t_0 - \tau_1$ are regular times of $\MM$.
Then we have the following bound,  using the shorthand notation $\rho = \rho_{(x_0,t_0)}$,
\begin{multline*}
\mathfrak{S}_{(x_0,t_0)}(\tau_1) - \mathfrak{S}_{(x_0,t_0)}(\tau_2) 
\geq  \int_{t_0-\tau_2}^{t_0-\tau_1} \int_{ \MM_{\reg,t}} \bigg( -(\partial_t \eta) \, \log \Big( \frac{\s_{\loc,1}}{\tdr_{\loc,1}} \Big) \, \rho  + \eta \, \frac{\partial_t \tdr_{\loc,1}}{\tdr_{\loc,1}} \, \rho +  \frac{\nabla \eta \cdot \nabla \s_{\loc,1}}{\s_{\loc,1}} \, \rho \notag \\
 + 
 \frac{\nabla \tdr_{\loc,1} \cdot \nabla \rho}{\tdr_{\loc,1}} \eta  - \log \Big( \frac{\s_{\loc,1}}{\tdr_{\loc,1}} \Big) \, \nabla \eta \cdot \nabla \rho \bigg) d\HH^2 dt.
\end{multline*}
Note that the quantities $\partial_t \eta, \nabla \eta, \nabla \s_{\loc,1}$ are only defined almost everywhere, because the functions $\eta, \s_{\loc,1}$ may only be locally Lipschitz.
\end{Lemma}

\begin{proof}
To simplify our computation, we assume without loss of generality that $t_0 =0$ and we set 
\[ u_1 := \log (\s_{\loc,1}), \qquad u_2 := \log(\tdr_{\loc,1}). \]
With this notation, it suffices to show that
\begin{equation} \label{eq_int_v_bound}
\int_{\MM_{t}} (1-\eta) u_1  \rho \, d\HH^2 \bigg|_{t=-\tau_2}^{t=-\tau_1}
\geq 
 \int_{-\tau_2}^{-\tau_1} \int_{\MM_{\reg,t}} \Big( - (\partial_t \eta) u_1 \rho  + (\nabla \eta \cdot \nabla u_1) \rho   - u_1 \nabla \eta \cdot \nabla \rho \Big)  \, d\HH^2 \, dt
\end{equation}
and
\begin{equation} \label{eq_int_w_bound}
\int_{\MM_{t}}  \eta u_2 \rho \, d\HH^2 \bigg|_{t=-\tau_2}^{t=-\tau_1}
\geq \int_{-\tau_2}^{-\tau_1} \int_{\MM_{\reg,t}} \Big(  (\partial_t \eta) u_2 \rho + \eta (\partial_t u_2) \rho   + (\nabla u_2 \cdot \nabla \rho) \eta + u_2\nabla \eta \cdot \nabla \rho \Big)  \, d\HH^2 \, dt.
\end{equation}
Note that the right-hand sides of \eqref{eq_int_v_bound}, \eqref{eq_int_w_bound} are integrable, because $\supp(1-\eta) \cap \MM_{[-\tau_2, -\tau_1]} \subset \MM_{\reg}$ is compact and since the terms $\partial_t u_2$ and $|\nabla u_2|$ can be bounded by $C r_{\loc}^{-2}$ and $C r_{\loc}^{-1}$ using Lemma~\ref{Lem_zeta_choice}\ref{Lem_zeta_choice_b}.
Note also that we may approximate $1-\eta$ by a smooth function of slightly smaller support.
So we may assume without loss of generality that $\eta$ is smooth on $\MM_{\reg}$ and that $\square u_1 \geq 0$ in a \emph{neighborhood} of $\supp (1-\eta)$.

Consider now the conjugate heat operator $\square^* = -\partial_t - \triangle + |\mathbf{H}|^2$; see also the discussion in  Subsection~\ref{subsec_furth_props}.
To verify the bound \eqref{eq_int_v_bound}, we first claim that 
\begin{equation} \label{eq_intconj}
 \int_{\MM_{t}} (1-\eta) u_1  \rho \, d\HH^2 \bigg|_{t=-\tau_2}^{t=-\tau_1}\geq - \int_{-\tau_2}^{-\tau_1} \int_{\MM_{\reg,t}} u_1 \big( \square^* ((1-\eta) \rho) \big) \, d\HH^2 \, dt. 
\end{equation}
Indeed, if $u_1$ were smooth, then this follows from the fact that $\square u_1 \geq 0$ on $\supp (1-\eta)$ via integration by parts.
In the general case, in which we only have $\square u_1 \geq 0$ in the viscosity sense, we can use a partition of unity to express
$(1- \eta) \rho = \sum_{i} v_i$, 
as a sum of finitely many smooth functions $v_i \in C^\infty_0(\MM_{\reg,[-\tau_2,-\tau_1]})$, whose supports add up to $\supp (1-\eta) \subset \MM_{\reg}$, such that for each $i$ we can apply Lemma~\ref{Lem_conj_L}.
Next, we obtain using Lemma~\ref{Lem_entropy_rho}\ref{Lem_entropy_rho_c} and the fact that $u_1 \leq 0$
\begin{align}
- \int_{-\tau_2}^{-\tau_1} \int_{\MM_{\reg,t}} &u_1 \big( \square^* ((1-\eta) \rho) \big) \, d\HH^2 \, dt \notag
\\
&= - \int_{-\tau_2}^{-\tau_1} \int_{\MM_{\reg,t}} u_1 \big( (\partial_t \eta) \rho + (\triangle \eta)\rho + 2 \nabla \eta \cdot \nabla \rho + (1-\eta)\square^* \rho ) \big) \, d\HH^2 \, dt \notag \\
&\geq - \int_{-\tau_2}^{-\tau_1} \int_{\MM_{\reg,t}} u_1 \big( (\partial_t \eta) \rho + (\triangle \eta)\rho + 2 \nabla \eta \cdot \nabla \rho  ) \big) \, d\HH^2 \, dt. \label{eq_almostfirst}
\end{align}
The bound \eqref{eq_int_v_bound} follows from \eqref{eq_intconj}, \eqref{eq_almostfirst} via integration by parts.

To verify the bound \eqref{eq_int_w_bound}, consider a smooth cutoff function $\psi : [-\infty,\infty) \to [0,1]$ with $\psi \equiv 1$ on $[-\infty,1]$ and $\psi \equiv 0$ on $[2, \infty)$.
Fix a small $\eps > 0$ and set, recalling that $u_2 \leq 0$,
\[ \omega_\eps := \psi (\eps \log(-u_2)). \]
Then on $\supp \omega_\eps$ we have
\[ r_{\loc,1} \geq \tfrac12 \tdr_{\loc,1}
\geq \tfrac12 e^{-e^{2/\eps}} > 0. \]
So since $(\supp \omega_\eps)_{[-\tau_2,-\tau_1]} \subset \MM_{\reg,[-\tau_2,-\tau_1]}$ and since the time-slices of the latter set have bounded $\HH^{2}$-measure, we find that $(\supp \omega_\eps)_{[-\tau_2,-\tau_1]}$ must be compact.
This allows us to apply integration by parts in the following computation, again using Lemma~\ref{Lem_entropy_rho},
\begin{align*}
 \int_{\MM_{t}}  & \eta u_2 \rho \omega_\eps \, d\HH^2 \bigg|_{t=-\tau_2}^{t=-\tau_1}
= \int_{-\tau_2}^{-\tau_1} \int_{\MM_{\reg,t}} \big( \partial_t (\eta u_2 \rho \omega_\eps) - |\mathbf{H}|^2 \eta u_2 \rho \omega_\eps \big) d\HH^2 \, dt
 \\
 &= \int_{-\tau_2}^{-\tau_1} \int_{\MM_{\reg,t}} \big( \rho \omega_\eps \, \partial_t (\eta u_2)  + \eta u_2 \omega_\eps \, \partial_t \rho - |\mathbf{H}|^2 \eta u_2 \rho \omega_\eps + \eta u_2 \rho \, \partial_t \omega_\eps \big) d\HH^2 \, dt \\
 &= \int_{-\tau_2}^{-\tau_1} \int_{\MM_{\reg,t}} \big( \rho \omega_\eps \, \partial_t (\eta u_2)  - \eta u_2 \omega_\eps \triangle \rho - \eta u_2 \omega_\eps (\square^* \rho )  + \eta u_2 \rho \, \partial_t \omega_\eps \big) d\HH^2 \, dt \\
 &\geq \int_{-\tau_2}^{-\tau_1} \int_{\MM_{\reg,t}} \big( \rho \omega_\eps \, \partial_t (\eta u_2)  +  \omega_\eps \nabla \rho \cdot \nabla (\eta u_2)  + 
 \eta u_2 \rho \, \partial_t \omega_\eps + \eta u_2 \, \nabla \rho \cdot \nabla \omega_\eps  \big) d\HH^2 \, dt. 
\end{align*}
Note that as $\eps \to 0$, the left-hand side of this inequality converges to the left-hand side of \eqref{eq_int_w_bound} and the integral over the first two terms in the integrand on the right-hand side converge to the right-hand side of \eqref{eq_int_w_bound}.
So it remains to show that
\begin{equation} \label{eq_int_to_0}
 \int_{-\tau_2}^{-\tau_1} \int_{\MM_{\reg,t}} \big( \eta u_2 \rho \, \partial_t \omega_\eps + \eta u_2 \, \nabla \rho \cdot \nabla \omega_\eps  \big) d\HH^2 \, dt \xrightarrow[\eps \to 0]{} 0  
\end{equation}
To see this, we use a generic constant $C < \infty$ and bound the absolute value of this integral using Lemma~\ref{Lem_zeta_choice}\ref{Lem_zeta_choice_b} by
\begin{align*}
 C \int_{-\tau_2}^{-\tau_1} & \int_{ \MM_{\reg,t}}  |u_2| \big( |\partial_t \omega_\eps|  + |\nabla\omega_\eps| \big) d\HH^2 \, dt \\
&\leq  C\eps \int_{-\tau_2}^{-\tau_1} \int_{ \MM_{\reg,t}}\big( |\partial_t u_2|  + |\nabla u_2| \big) d\HH^2 \, dt \\
 &\leq  C\eps \int_{-\tau_2}^{-\tau_1} \int_{ \MM_{\reg,t}}  \bigg( \frac{|\partial_t \tdr_{\loc,1}|}{\tdr_{\loc,1}}  + \frac{|\nabla \tdr_{\loc,1}|}{\tdr_{\loc,1}} \bigg) d\HH^2 \, dt \\
 &\leq  C\eps \int_{-\tau_2}^{-\tau_1} \int_{\MM_{\reg,t}}  \Big(   r_{\loc,1}^{-2}  +   r_{\loc,1}^{-1}  \Big) d\HH^2 \, dt
\end{align*}
Since the last integral is finite by Definition~\ref{Def_almost_regular}, we obtain \eqref{eq_int_to_0}, which finishes the proof.
\end{proof}
\bigskip

\subsection{Proof of the key estimate} \label{subsec_proof_key}
In this subsection we prove the key estimate of this paper, Theorem~\ref{Thm_key_thm}.
To do this, we derive a bound on the decrease of $\mathfrak{S}_{(x_0,t_0), \sqrt{\tau}}(\tau)$ as we reduce the scale parameter $\tau$; see Proposition~\ref{Prop_step_bound_improved} for more details.
Theorem~\ref{Thm_key_thm} will then follow by an iterated application of this bound.
As a precursor to this proposition, we first establish a bound on the decrease of the quantity $\mathfrak{S}_{(x_0,t_0)}(\tau)$; see Proposition~\ref{Prop_step_bound}.
This bound will be deduced by combining Lemma~\ref{Lem_asspt_eta_rho} with the bound from Lemma~\ref{Lem_evol_formula}.

Note that Proposition~\ref{Prop_step_bound} is not invariant under parabolic rescaling, as it involves the quantity $\mathfrak{S}_{(x_0,t_0)}(\tau)$.
The proposition can be easily generalized to a scale invariant form by replacing this quantity with $\mathfrak{S}_{(x_0,t_0), a}(\tau)$, where $a$ is an additional scale parameter.
We will obtain Proposition~\ref{Prop_step_bound_improved}, which is scaling invariant, from Proposition~\ref{Prop_step_bound} via parabolic rescaling.

\begin{Proposition} \label{Prop_step_bound}
For any $\eps > 0$ and $A < \infty$ there are constants $\delta'(\eps,A), \alpha(\eps,A) > 0$ and $C_1(\eps, A), \lb C_2 (\eps,A), D(\eps, A) < \infty$ such that the following is true.
Consider a bounded almost regular mean curvature flow $\MM \subset \IR^3 \times I$ and suppose that
\begin{equation} \label{eq_geom_bounds}
 \Theta(\MM) \leq A, \qquad \genus (\MM) \leq A. 
\end{equation}
Let $\eta : \MM \to [0,1]$ be a function that satisfies the assertions from Lemma~\ref{Lem_asspt_eta_rho} for the parameters $\alpha$ and $A$.
Consider the quantity $\mathfrak{S}_{(x_0,t_0)}(\cdot)$ from Definition~\ref{Def_S} for some $(x_0,t_0) \in \IR^3 \times \ov{I}$.
Suppose that $0 < \tau_1 < \tau_2 \leq A$ with $2 \leq \frac{\tau_2}{\tau_1} \leq A$ and $t_0 - 2 \tau_2\in I$ and assume that $t_0 - \tau_2$ and $t_0 - \tau_1$ are regular times.
Then
\begin{equation} \label{eq_bad_step_bound}
 \mathfrak{S}_{(x_0,t_0)}(\tau_1) - \mathfrak{S}_{(x_0,t_0)}(\tau_2)  \geq - C_1.
\end{equation}
Moreover, one of the following is true:
\begin{enumerate}[label=(\alph*)]
\item \label{Prop_step_bound_b} We have
\begin{equation} \label{eq_good_step_bound}
\mathfrak{S}_{(x_0,t_0)}(\tau_1) - \mathfrak{S}_{(x_0,t_0)}(\tau_2) 
\geq - C_2 \big( \Theta_{(x_0,t_0)}(2\tau_2) - \Theta_{(x_0,t_0)}(\tfrac12 \tau_1) \big) - \eps.
\end{equation}
\item \label{Prop_step_bound_a} There is a point $(x,t) \in \MM$ with $x \in B(x_0, D \sqrt{\tau_1})$ and $t \in [t_0-\tau_2, t_0-\tau_1]$ such that $r_{\loc}(x,t) \geq \delta' \sqrt{\tau_1}$.
\end{enumerate}
\end{Proposition}

\begin{proof}
Fix $\eps, A$ and let $\alpha$ and $D \geq A$ be constants, which we will determine in the course of the proof.
We assume without loss of generality that $A \geq 10^4$.
Let $\delta(\alpha, A) > 0$ and $\XX^{t_0,\delta}, \YY \subset \MM$ be the constant and subsets from Lemma~\ref{Lem_asspt_eta_rho}.
Moreover, let $\delta'_A(\alpha, A) > 0$ and $\delta'_D(\alpha,D) > 0$ be the constants from Lemma~\ref{Lem_asspt_eta_rho}\ref{Lem_asspt_eta_rho_d}, where we have replaced $D$ with $A$ for the first constant.
Without loss of generality, we may assume that
\[ \delta'_D \leq \delta'_A \leq \delta. \]
After application of a translation, we may also assume without loss of generality that $(x_0,t_0) = (\vec 0,0)$.
By Lemma~\ref{Lem_evol_formula} we have
\begin{equation} \label{eq_SS_int_Q}
 \mathfrak{S}_{(\vec 0,0)}(\tau_1) - \mathfrak{S}_{(\vec 0,0)}(\tau_2) 
\geq \int_{-\tau_2}^{-\tau_1} \int_{ \MM_{\reg,t}}  Q \, d\HH^2 \,dt, 
\end{equation}
where 
\begin{equation*}
Q :=  -(\partial_t \eta) \, \log \Big( \frac{\s_{\loc,1}}{\tdr_{\loc,1}} \Big) \, \rho  + \eta \, \frac{\partial_t \tdr_{\loc,1}}{\tdr_{\loc,1}} \, \rho +  \frac{\nabla \eta \cdot \nabla \s_{\loc,1}}{\s_{\loc,1}} \, \rho 
 + 
 \frac{\nabla \tdr_{\loc,1} \cdot \nabla \rho}{\tdr_{\loc,1}} \eta  - \log \Big( \frac{\s_{\loc,1}}{\tdr_{\loc,1}} \Big) \, \nabla \eta \cdot \nabla \rho .
\end{equation*}

We first estimate the integral on the right-hand side of \eqref{eq_SS_int_Q} restricted to the domain $\YY \setminus \XX^{0,\delta}$.
To do this, we first estimate the integrand using \eqref{eq_r_s_comparable}, Lemma~\ref{Lem_zeta_choice}\ref{Lem_zeta_choice_a}, \ref{Lem_zeta_choice_b}, Lemma~\ref{Lem_asspt_eta_rho}\ref{Lem_asspt_eta_rho_a} and Lemma~\ref{Lem_mod_sep_fct}.
We obtain the following bound over $\YY_{[-\tau_2,-\tau_1]}$, where $C < \infty$ denotes a generic constant, which may depend on the given parameters in the indicated way, 
\begin{align*}
 |Q| &\leq \frac{C}{|\mathbf{t}|} \bigg( |\partial_t \eta| 
+ \frac{|\partial_t \tdr_{\loc} |}{\td r_{\loc}} 
+ \frac{|\nabla \eta | \, |\nabla \s_{\loc}|}{r_{\loc}} 
+ \frac{|\mathbf{x}|}{|\mathbf{t}| \td r_{\loc}}  \bigg) e^{- |\mathbf{x}|^2 / 4|\mathbf{t}|} \\
&\leq \frac{C(A)}{|\mathbf{t}|} \bigg( |\partial_t \eta| 
+ \frac{|\partial_t \tdr_{\loc} |}{\td r_{\loc}} 
+ \frac{|\nabla \eta | \, |\nabla \s|}{\td r_{\loc}} 
+ \frac{1}{|\mathbf{t}|^{1/2} \td r_{\loc}}  \bigg) e^{- |\mathbf{x}| / \sqrt{\tau_1}} \\
&=: \ov Q e^{- |\mathbf{x}| / \sqrt{\tau_1}}.
\end{align*}
Therefore, if we apply Lemma~\ref{Lem_asspt_eta_rho}\ref{Lem_asspt_eta_rho_c} for $x_0$ replaced by the points belonging to the grid $\sqrt{\tau_1} \IZ^3 \subset \IR^3$, then we obtain
\begin{align}
 \int_{-\tau_2}^{-\tau_1} \int_{ (\YY \setminus \XX^{0,\delta})_t}  Q \, d\HH^2 \,dt  
&\geq  - \sum_{z \in \sqrt{\tau_1} \IZ^3} \int_{-\tau_2}^{-\tau_1}  \int_{ (\YY \setminus \XX^{0,\delta})_t \cap B(z, 100 \sqrt{\tau_1})} \ov Q e^{-|\mathbf{x}|/\sqrt{\tau_1}} \, d\HH^2 \,dt \notag \\
&\geq  - e^{100} \sum_{z \in \sqrt{\tau_1} \IZ^3} e^{-|z|/\sqrt{\tau_1}} \int_{-\tau_2}^{-\tau_1}  \int_{ (\YY \setminus \XX^{0,\delta})_t \cap B(z, 100 \sqrt{\tau_1})}  \ov Q \, d\HH^2 \,dt \notag \\
&\geq - C(A) \alpha. \label{eq_Q_bound_YY}
\end{align}

Next, we estimate the integral on the right-hand side of \eqref{eq_SS_int_Q} restricted to the complement of $\YY \setminus \mathcal{X}^{0,\delta}$.
To do this, we use the following cruder bound over $\MM_{\reg,[-\tau_2,-\tau_1]}$
\[ |Q| \leq \frac{C(A)}{|\mathbf{t}|} \Big( r_{\loc}^{-2}+ |\mathbf{t}|^{-1/2} r_{\loc}^{-1} \Big)  e^{- |\mathbf{x}|/\sqrt{\tau_1}} \, \chi_{\supp \eta} =: \ov{Q}'e^{- |\mathbf{x}|/\sqrt{\tau_1}} \, \chi_{\supp \eta}  , \]
which follows from \eqref{eq_r_s_comparable}, Lemma~\ref{Lem_zeta_choice}\ref{Lem_zeta_choice_a}, \ref{Lem_zeta_choice_b}, Lemma~\ref{Lem_asspt_eta_rho}\ref{Lem_asspt_eta_rho_a} and  Lemma~\ref{Prop_sep}\ref{Prop_sep_c}.
Set
\[ \mathcal{S} := \big( (\supp \eta) \setminus (\YY \setminus \XX^{0,\delta}) \big) \cap \MM_{\reg} \]
and observe that
\[ \mathcal{S} \setminus \XX^{0,\delta'_D} \subset 
\mathcal{S} \setminus \XX^{0,\delta'_A} \subset 
\mathcal{S} \setminus \XX^{0,\delta} \subset \big( (\supp \eta) \setminus \YY \big) \cap \MM_{\reg}. \]
Let $B_0 := B(\vec 0, D \sqrt{\tau_1})$ and choose balls $B_i := B(z_i, 100 \sqrt{\tau_1})$ for a sequence of pairwise distinct $z_i \in \sqrt{\tau_1} \IZ^3$ such that $|z_i| \geq D -100 \sqrt{\tau_1}$ and such that $\bigcup_{i=0}^\infty B_i = \IR^3$.
We now apply Lemma~\ref{Lem_asspt_eta_rho}\ref{Lem_asspt_eta_rho_d} for $x_0$ replaced with each of the $z_i$, where we choose the parameter $D$ to be $A$ if $i \geq 1$ and to be the same constant as in this proof if $i = 0$.
We obtain that
\begin{align*}
 \int_{-\tau_2}^{-\tau_1} \int_{(\MM_{\reg} \setminus (\YY \setminus \XX^{0,\delta}))_t} &  Q \, d\HH^2 \,dt \\
\geq& 
-\int_{-\tau_2}^{-\tau_1}  \int_{(\mathcal{S} \setminus \XX^{0,\delta'_D} )_t \cap B(\vec 0,D \sqrt{\tau_1})} \ov{Q}'  e^{- |\mathbf{x}|/\sqrt{\tau_1}}  \, d\HH^2 \, dt \\
& 
-\int_{-\tau_2}^{-\tau_1}  \int_{(\mathcal{S} \cap \XX^{0,\delta'_D} )_t \cap B(\vec 0,D \sqrt{\tau_1})} \ov{Q}'  e^{- |\mathbf{x}|/\sqrt{\tau_1}}  \, d\HH^2 \, dt \\
&- \sum_{i = 1}^\infty \int_{-\tau_2}^{-\tau_1}  \int_{(\mathcal{S} \setminus \XX^{0,\delta'_A})_t \cap B(z_i,100 \sqrt{\tau_1})}  \ov{Q}'  e^{- |\mathbf{x}|/\sqrt{\tau_1}}  \, d\HH^2 \, dt  \\
&- \sum_{i = 1}^\infty \int_{-\tau_2}^{-\tau_1}  \int_{(\mathcal{S} \cap \XX^{0,\delta'_A})_t \cap B(z_i,100 \sqrt{\tau_1})}  \ov{Q}'  e^{- |\mathbf{x}|/\sqrt{\tau_1}}  \, d\HH^2 \, dt \displaybreak[1] \\
\geq& -C(\alpha, D) \big( \Theta_{(\vec 0, 0)} (2 \tau_2) - \Theta_{(\vec 0, 0)} (\tfrac12 \tau_1) \big) \\
&-\int_{-\tau_2}^{-\tau_1}  \int_{ \XX^{0,\delta'_D}_t \cap B(\vec 0,D \sqrt{\tau_1})} \frac{C(A) (r_{\loc}^{-2} + |t|^{-1/2} r_{\loc}^{-1})}{|t|}  e^{- |\mathbf{x}|/\sqrt{\tau_1}}  \, d\HH^2 \, dt  \\
&- \ \sum_{i=1}^\infty C(\alpha,A) \big( \Theta_{(z_i, 0)} (2 \tau_2) - \Theta_{(z_i, 0)} (\tfrac12 \tau_1) \big) e^{-|z_i|/\sqrt{\tau_1}} \\
&- \ \sum_{i=1}^\infty e^{- |z_i|/\sqrt{\tau_1}} \int_{-\tau_2}^{-\tau_1}  \int_{ \XX^{0,\delta'_A}_t \cap B(z_i,100 \sqrt{\tau_1})} \frac{C(A) (r_{\loc}^{-2} + |t|^{-1/2} r_{\loc}^{-1})}{|t|}   \, d\HH^2 \, dt  \displaybreak[1] \\
\geq& -C(\alpha,D) \big( \Theta_{(\vec 0, 0)} (2 \tau_2) - \Theta_{(\vec 0, 0)} (\tfrac12 \tau_1) \big) \\
&- \frac{C(A) (\delta^{\prime}_D)^{-2}}{\tau_1^2}  \int_{-\tau_2}^{-\tau_1}  \int_{\XX^{0,\delta'_D}_t \cap B(\vec 0,D \sqrt{\tau_1})}  e^{- |\mathbf{x}|/\sqrt{\tau_1}}  \, d\HH^2 \, dt  \\
&- \ \sum_{i=1}^\infty C(\alpha,A) A e^{-|z_i|/\sqrt{\tau_1}} \\
&- \ \sum_{i=1}^\infty C(A) (\delta^{\prime}_A)^{-2} e^{- |z_i|/\sqrt{\tau_1}} \displaybreak[1] \\
\geq& -C(\alpha, D) \big( \Theta_{(\vec 0, 0)} (2 \tau_2) - \Theta_{(\vec 0, 0)} (\tfrac12 \tau_1) \big) - C(\alpha, A,\delta'_A(\alpha, A)) e^{- D} \\
&-\frac{C(A) (\delta'_D(\alpha,D))^2}{\tau_1^2}  \int_{-\tau_2}^{-\tau_1}  \int_{ \XX^{0,\delta'_D}_t  \cap B(\vec 0,D \sqrt{\tau_1})}  e^{- |\mathbf{x}|/\sqrt{\tau_1}}  \, d\HH^2 \, dt  .
\end{align*}
Recall that $\delta = \delta (\alpha, A)$.
So combining this bound with \eqref{eq_SS_int_Q} and \eqref{eq_Q_bound_YY} implies that 
\begin{multline*}
 \mathfrak{S}_{(\vec 0,0)}(\tau_1) - \mathfrak{S}_{(\vec 0,0)}(\tau_2) 
\geq -C(D,\alpha) \big( \Theta_{(\vec 0, 0)} (2 \tau_2) - \Theta_{(\vec 0, 0)} (\tfrac12 \tau_1) \big) - C(\alpha,A) e^{-D} - C(A) \alpha \\-\frac{C(\alpha,A,D)}{\tau_1^2}  \int_{-\tau_2}^{-\tau_1}  \int_{ \XX^{0,\delta'_D}_t  \cap B(\vec 0,D \sqrt{\tau_1})}  e^{- |\mathbf{x}|/\sqrt{\tau_1}}  \, d\HH^2 \, dt.  
\end{multline*}
Now choose $\alpha = \alpha(\eps,A) > 0$ small enough such that $C(A) \alpha \leq \frac12 \eps$ and then choose $D = D(\eps, \alpha, A) <\infty$ large enough such that $C(\alpha, A) e^{- D}  \leq  \frac12\eps$. 
We therefore obtain
\begin{multline*}
 \mathfrak{S}_{(\vec 0,0)}(\tau_1) - \mathfrak{S}_{(\vec 0,0)}(\tau_2) 
\geq -C(\eps,A) \big( \Theta_{(\vec 0, 0)} (2 \tau_2) - \Theta_{(\vec 0, 0)} (\tfrac12 \tau_1) \big) - \eps \\
-\frac{C(\eps,A)}{\tau_1^2}  \int_{-\tau_2}^{-\tau_1}  \int_{\XX^{0,\delta'_D}_t  \cap B(\vec 0,D \sqrt{\tau_1})}  e^{- |\mathbf{x}|/\sqrt{\tau_1}}  \, d\HH^2 \, dt.
\end{multline*}
This implies Assertion~\ref{Prop_step_bound_b} if $\XX^{0,\delta'_D}_t \cap B(\vec 0,D \sqrt{\tau_1}) = \emptyset$ for all $t \in [-\tau_2,-\tau_1]$.
On the other hand, if $\XX^{0,\delta'_D}_t \cap B(\vec 0,D \sqrt{\tau_1}) \neq \emptyset$ for some $t \in [-\tau_2,-\tau_1]$, then Assertion~\ref{Prop_step_bound_a} holds for $\delta' = \delta'_D$.
Since we can bound the area of subsets $\XX_t^{0.\delta'_D} \cap B(\vec 0,r) \subset \MM_t \cap B(\vec 0,r)$ for $r \leq D \sqrt{\tau_1}$ via the Gaussian area, we also obtain the bound \eqref{eq_bad_step_bound}.
This finishes the proof.
\end{proof}
\bigskip

The following proposition is a consequence of Proposition~\ref{Prop_step_bound}.
Its main objective is the removal of the $\eps$-term in \eqref{eq_good_step_bound} by considering the quantities $\mathfrak{S}_{(x_0,t_0), \sqrt{\tau}} (\tau)$ in lieu of $\mathfrak{S}_{(x_0,t_0)} (\tau)$.
We also obtain a stronger bound in Assertion~\ref{Prop_step_bound_improved_a} via a limit argument.

\begin{Proposition} \label{Prop_step_bound_improved}
For any $A, R < \infty$ there are constants $\delta(A,R), \alpha(A) > 0$ and $C_3( A), \lb C_4 (A, \lb R) \lb < \infty$ such that the following is true.
Consider a bounded almost regular mean curvature flow $\MM \subset \IR^3 \times I$ and suppose that
\begin{equation} \label{eq_geom_bounds_improved}
 \Theta(\MM) \leq A, \qquad \genus (\MM) \leq A. 
\end{equation}
Let $\eta : \MM \to [0,1]$ be a function that satisfies the assertions from Lemma~\ref{Lem_asspt_eta_rho} for the parameters $A$ and $\alpha$.
Consider the quantity $\mathfrak{S}_{(x_0,t_0), \sqrt{\tau}}(\tau)$ from Definition~\ref{Def_S} for some $(x_0,t_0) \in \IR^3 \times \ov{I}$.
Suppose that $0 < \tau_1 < \tau_2$ with $2 \leq \frac{\tau_2}{\tau_1} \leq A$ and $t_0 - 2 \tau_2\in I$ and assume that $t_0 - \tau_2$ and $t_0 - \tau_1$ are regular times.
Then 
\begin{equation} \label{eq_bad_step_bound_improved}
 \mathfrak{S}_{(x_0,t_0), \sqrt{\tau_1}}(\tau_1) - \mathfrak{S}_{(x_0,t_0), \sqrt{\tau_2}}(\tau_2)  \geq - C_3
\end{equation}
Moreover, one of the following is true:
\begin{enumerate}[label=(\alph*)]
\item \label{Prop_step_bound_improved_b} We have
\begin{equation} \label{eq_good_step_bound_improved}
\mathfrak{S}_{(x_0,t_0), \sqrt{\tau_1}}(\tau_1) - \mathfrak{S}_{(x_0,t_0), \sqrt{\tau_2}}(\tau_2) 
\geq - C_4 \big( \Theta_{(x_0,t_0)}(2\tau_2) - \Theta_{(x_0,t_0)}(\tfrac12 \tau_1) \big) .
\end{equation}
\item \label{Prop_step_bound_improved_a} We have
\begin{equation} \label{eq_rloc_geq_delta_improved}
 r_{\loc} \geq \delta \sqrt{\tau_1} \textQQq{on} B(x_0, R \sqrt{\tau_1}) \times [t_0 - \tau_2,t_0-\tau_1].
\end{equation}
\end{enumerate}
\end{Proposition}

\begin{proof}
Let $\eps > 0$ be a constant whose value we will choose later, depending on $A$.
Invoking the constants from Proposition~\ref{Prop_step_bound}, set
\begin{equation} \label{eq_choice_alpha_C3}
\alpha (A) := \alpha(\eps, A), \qquad
C_3(A) := C_1(\eps, A). 
\end{equation}
The constants $\delta(A,R)$ and $C_4(A,R)$ will be chosen later such that
\begin{equation} \label{eq_C4_large}
 C_4(A,R) \geq C_2(\eps, A) + 1. 
\end{equation}

After parabolic rescaling and application of a translation in time and space, we may assume without loss of generality that $(x_0, t_0) = (\vec 0, 0)$ and $\tau_1 = 1$.
Note that this implies that $2 \leq \tau_2 \leq A$ and $\mathfrak{S}_{(x_0,t_0), \sqrt{\tau_1}}(\cdot) = \mathfrak{S}_{(\vec 0, 0)}(\cdot)$.
We also find that
\begin{multline*}
   \mathfrak{S}_{(\vec 0, 0)}(\tau_2) - \mathfrak{S}_{(\vec 0,0), \sqrt{\tau_2}}(\tau_2) \\
= \int_{\MM_{- \tau_2}} \bigg( (1-\eta) \bigg(  \log ( \s_{\loc,1}) - \log \Big( \frac{\s_{\loc, \sqrt{\tau_2}}}{\sqrt{\tau_2}} \Big) \bigg) 
+ \eta \bigg(  \log ( \tdr_{\loc,1}) - \log \Big( \frac{\td r_{\loc, \sqrt{\tau_2}}}{\sqrt{\tau_2}} \Big)  \bigg) \bigg) \, \rho_{(\vec 0,0)}\, d\HH^2 \\
= \int_{\MM_{ - \tau_2}} \bigg( (1-\eta)  \log \Big( \frac{\s_{\loc,1}\sqrt{\tau_2}}{\s_{\loc, \sqrt{\tau_2}}} \Big) 
+ \eta \log \Big( \frac{\td r_{\loc,1}\sqrt{\tau_2}}{\td r_{\loc, \sqrt{\tau_2}}} \Big)  \bigg) \, \rho_{(\vec 0,0)}\, d\HH^2 \geq 0.
\end{multline*}
This shows that \eqref{eq_bad_step_bound} of Proposition~\ref{Prop_step_bound} implies \eqref{eq_bad_step_bound_improved}.
In addition, if Assertion~\ref{Prop_step_bound_b} of Proposition~\ref{Prop_step_bound} holds, but Assertion~\ref{Prop_step_bound_improved_b} of this proposition is violated, then combining \eqref{eq_good_step_bound} with the reverse of \eqref{eq_good_step_bound_improved} and assuming \eqref{eq_C4_large}, implies that
\begin{multline} \label{eq_bound_on_int_Theta}
 \int_{\MM_{ - \tau_2}} \bigg( (1-\eta)  \log \Big( \frac{\s_{\loc,1}\sqrt{\tau_2}}{\s_{\loc, \sqrt{\tau_2}}} \Big) 
+ \eta \log \Big( \frac{\td r_{\loc,1}\sqrt{\tau_2}}{\td r_{\loc, \sqrt{\tau_2}}} \Big)  \bigg) \, \rho_{(\vec 0,0)}\, d\HH^2  + \big( \Theta_{(\vec 0,0)}(2\tau_2) - \Theta_{(\vec 0,0)}(\tfrac12) \big) \leq \eps. 
\end{multline}
Next, note that whenever $r_{\loc} \leq \tfrac12$, we have $\s_{\loc} \leq \td r_{\loc} \leq 1$, so
\[ \frac{\s_{\loc,1}\sqrt{\tau_2}}{\s_{\loc, \sqrt{\tau_2}}} = \sqrt{\tau_2} \geq \sqrt{2}, \qquad
\frac{\td r_{\loc,1}\sqrt{\tau_2}}{\td r_{\loc, \sqrt{\tau_2}}} \geq \sqrt{2}. \]
So \eqref{eq_bound_on_int_Theta} implies that
\begin{equation} \label{eq_boundontwo}
 \frac{\log 2}2 \int_{\MM_{ - \tau_2} \cap \{ r_{\loc} \leq \frac12 \}}   \, \rho_{(\vec 0,0)}\, d\HH^2  + \big( \Theta_{(\vec 0,0)}(2\tau_2) - \Theta_{(\vec 0, 0)}(\tfrac12) \big) \leq \eps.  
\end{equation}

\begin{Claim} \label{Cl_choice_eps}
If $\eps \leq \ov\eps (A)$, then \eqref{eq_boundontwo} implies that there is a point $x \in \MM_{-\tau_2} \cap B(\vec 0, 10 \sqrt{\tau_2})$ with $r_{\loc} (x,-\tau_2) \geq \frac12$.
\end{Claim}

\begin{proof}
Consider a sequence of counterexamples $\MM^i$ satisfying \eqref{eq_geom_bounds_improved} for a uniform $A$ and \eqref{eq_boundontwo} for a sequence $\eps_i \to 0$ and $\tau_{2,i} \in [2,A]$.
After passing to a subsequence, we may assume that $\tau_{2,i} \to \tau_{2,\infty}$ and that we have weak convergence of the associated Brakke flows restricted to $[-\tau_{2,i},-\frac12]$ to a Brakke flow $(\mu^\infty_t)_{t \in (-\tau_{2,\infty}, -\frac12)}$.
By Lemma~\ref{Lem_Brakke_to_min_shrink}\ref{Lem_Brakke_to_min_shrink_b} and \eqref{eq_boundontwo} there is a smooth shrinker $\Sigma \subset \IR^3$ such that $d\mu^\infty_t = k_t d (\HH^2 \lfloor (|t|^{1/2} \Sigma))$ for a locally constant function $k_t : |t|^{1/2} \Sigma \to \IN$.
This shrinker $\Sigma$ must intersect the ball $B(\vec 0, 10)$, because otherwise $2 S^2 \cup \Sigma$ would be a disconnected shrinker, in contradiction to Lemma~\ref{Lem_shrinker_connected}.
But, on the other hand, taking \eqref{eq_boundontwo} to the limit, combined with our contradiction assumption, we obtain that $\Sigma$ must be disjoint from $B(\vec 0, 10)$.
\end{proof}

Let us now fix $\eps$ for the remainder of the proof, which will determine the constants $\alpha$, $C_3$ via \eqref{eq_choice_alpha_C3}.
We will now prove by contradiction that the proposition is true if $\delta$ is chosen sufficiently small and $C_4$ is chosen sufficiently large depending on $A, R$.
So fix $A, R$ and consider a sequence of counterexamples $\MM^i$ with $\tau_{2,i} \in [2, A]$ that violate both Assertions~\ref{Prop_step_bound_b} and \ref{Prop_step_bound_a} for sequences $\delta_i \to 0$ and $C_{4,i} \to \infty$.
Since the left-hand side of \eqref{eq_good_step_bound_improved} is uniformly bounded from below due to \eqref{eq_bad_step_bound_improved}, this implies that
\[ \Theta_{(\vec 0, 0)}(2 \tau_{2,i}) - \Theta_{(\vec 0, 0)}(\tfrac12) \lto 0. \]
So by the same argument as in the proof of Claim~\ref{Cl_choice_eps}, may pass to a subsequence such that $\tau_{2,i} \to \tau_{2, \infty}$ and such that the associated Brakke flows weakly converge to a Brakke flow $(\mu^\infty_t)_{t \in (-\tau_{2,\infty}, -\frac12)}$ with $d\mu^\infty_t = k_t d (\HH^2 \lfloor (\sqrt{-t} \Sigma))$ for a smooth shrinker $\Sigma \subset \IR^3$.
By Lemma~\ref{Lem_shrinker_connected}, we even conclude that $k \in \IN$ is a constant.
If Assertion~\ref{Prop_step_bound_a} of Proposition~\ref{Prop_step_bound} holds for infinitely many $i$, then after passing to a subsequence, there is a sequence of points $(x_i,t_i) \in \MM^i$ converging to a point $(x_\infty, t_\infty)$ with $t_\infty \in [-\tau_{2,\infty}, -1]$ for which $r_{\loc}(x_i,t_i)$ is uniformly bounded from below.
On the other hand, if Assertion~\ref{Prop_step_bound_b} of Proposition~\ref{Prop_step_bound} holds for infinitely many $i$, then the same is true due to Claim~\ref{Cl_choice_eps}.
This implies that $k = 1$ and therefore we have \emph{smooth} convergence of the flows $\MM^i$ to the flow corresponding to $\Sigma$.
So Assertion~\ref{Prop_step_bound_improved_a} holds for large $i$, in contradiction to our assumption.
\end{proof}
\bigskip

We can now prove the key result of this paper.

\begin{proof}[Proof of Theorem~\ref{Thm_key_thm}.]
Choose $A \geq 3$ so that Proposition~\ref{Prop_step_bound_improved}  can be applied to $\MM$ (restricted to any subinterval).
Let $R$ be a constant whose value we will choose later, based on some universal constants, and let $\delta(A,R), \alpha(A) > 0$ and $C_3(A), C_4( A, R) < \infty$ be the constants from Proposition~\ref{Prop_step_bound}.
Since $r_{\loc} \geq \frac12 \tdr_{\loc} \geq \frac12 \s_{\loc}$, it suffices to prove the following bound for any $0 < \tau \leq t_0 \leq T$ with the property that $t_0 - \tau$ is a regular time
\begin{equation} \label{eq_S_lower_bound}
   \mathfrak{S}_{(x_0,t_0), \sqrt{\tau}} (\tau) \geq   -C (A, \MM_0) . 
\end{equation}
To achieve this, we consider the following quantity, which is defined whenever $t_0 - \tau$ is a regular time and $2 \tau \leq t_0$:
\begin{equation} \label{eq_S_is_inf}
 S_{t_0}(\tau) := \inf_{x_0 \in \IR^3} \big( \mathfrak{S}_{(x_0,t_0), \sqrt{\tau}}(\tau)  - C_4 \big( \Theta_{(x_0,t_0)}(\tfrac12 \tau) + \Theta_{(x_0,t_0)}( \tau)  + \Theta_{(x_0,t_0)}(2 \tau) \big) \big). 
\end{equation}
The following claim is a consequence of Proposition~\ref{Prop_step_bound_improved}.

\begin{Claim} \label{Cl_St1t2}
There is a choice for $R$, which we will fix henceforth, such that the following is true.
Suppose that $0 < t_0 \leq T$ and $0 < \tau_1 <\tau_2 \leq \frac12 t_0$ and $2 \leq \frac{\tau_2}{\tau_1} \leq 3$ such that $t_0 - \tau_2, t_0 - \tau_1$ are regular times.
Then one of the following is true:
\begin{enumerate}[label=(\alph*)]
\item \label{Cl_St1t2_a} $S_{t_0}(\tau_1) \geq S_{t_0} (\tau_2)$.
\item \label{Cl_St1t2_b} $S_{t_0}(\tau_1) \geq - C_5(A)$.
\end{enumerate}
\end{Claim}

\begin{proof}
Consider the two cases in Proposition~\ref{Prop_step_bound_improved}.
Suppose first that Case~\ref{Prop_step_bound_improved_b} holds for any $x_0 \in \IR^3$ that is close to realizing the infimum in \eqref{eq_S_is_inf}, in the sense that
\begin{equation} \label{eq_almost_inf}
   S_{t_0} (\tau_1) + 1 \geq \mathfrak{S}_{(x_0,t_0), \sqrt{\tau_1}}(\tau_1)  - C_4 \big( \Theta_{(x_0,t_0)}(\tfrac12 \tau_1) + \Theta_{(x_0,t_0)}( \tau_1)  + \Theta_{(x_0,t_0)}(2 \tau_1) \big).  
\end{equation}
For any such $x_0$ we have
\begin{align*}
  \mathfrak{S}_{(x_0,t_0), \sqrt{\tau_1}}&(\tau_1)   - C_4 \big( \Theta_{(x_0,t_0)}(\tfrac12 \tau_1) +  \Theta_{(x_0,t_0)}( \tau_1)  +  \Theta_{(x_0,t_0)}(2 \tau_1) \big)  \\
&\geq \mathfrak{S}_{(x_0,t_0), \sqrt{\tau_2}}(\tau_2)  \\&\qquad - C_4 \big( \Theta_{(x_0,t_0)}(\tfrac12 \tau_1) +  \Theta_{(x_0,t_0)}( \tau_1)  +  \Theta_{(x_0,t_0)}(2 \tau_1) +  \Theta_{(x_0,t_0)}(2 \tau_2) -  \Theta_{(x_0,t_0)}(\tfrac12 \tau_1) \big)  \\
&\geq \mathfrak{S}_{(x_0,t_0), \sqrt{\tau_2}}(\tau_2) - C_4 \big(    \Theta_{(x_0,t_0)}( \tau_1)  +  \Theta_{(x_0,t_0)}(2 \tau_1) +  \Theta_{(x_0,t_0)}(2 \tau_2) \big)  \\
&\geq \mathfrak{S}_{(x_0,t_0), \sqrt{\tau_2}}(\tau_2)    - C_4 \big(    \Theta_{(x_0,t_0)}( \tfrac12 \tau_2)  +  \Theta_{(x_0,t_0)}(\tau_2) +  \Theta_{(x_0,t_0)}(2 \tau_2) \big)  \\
 &\geq S_{t_0}(\tau_2)   .
\end{align*}
Taking the infimum over all $x_0$ implies Assertion~\ref{Cl_St1t2_a}.

Now suppose that there is an $x_0 \in \IR^3$ satisfying \eqref{eq_almost_inf} and Proposition~\ref{Prop_step_bound_improved}\ref{Prop_step_bound_improved_a}.
Cover $\IR^3$ by $B(x_0, R \sqrt{\tau_1})$ and balls of the form $B_z := B(x_0 + \sqrt{\tau_1} z, 100 \sqrt{\tau_1})$ for all $z \in \IZ^3$ with $|z| \geq \tfrac12 R$.
Then we obtain that for some generic constants $C'(A,R), C'' < \infty$
\begin{align*} 
   \mathfrak{S}_{(x_0,t_0), \sqrt{\tau_1}}(\tau_1) 
&\geq - C'(A,R) + C'(A,R) \log (\delta(A,R)) + C'' \sum_{\substack{z \in \IZ^3 \\ |z| \geq \frac12 R}} e^{-|z|} \mathfrak{S}_{(x_0 + \sqrt{\tau_1} z,t_0), \sqrt{\tau_1}} (\tau_1) \\
&\geq - C'(A,R) +  \bigg( C''  \sum_{\substack{z \in \IZ^3 \\ |z| \geq \frac12 R}} e^{-|z|} \bigg) S_{t_0} (\tau_1).
\end{align*}
Let us now choose and fix $R$ large enough so that the term in the parentheses is $< \frac12$.
So combined with \eqref{eq_almost_inf}, we obtain
\[ S_{t_0} (\tau_1) + 1 + C_4(A,R)\cdot 3A \geq - C'(A,R) + \tfrac12 S_{t_0} (\tau_1). \]
This implies Assertion~\ref{Cl_St1t2_b}.
\end{proof}

Iterating Claim~\ref{Cl_St1t2} yields that for any $0 < t_0 \leq T$ and $0 < \tau \leq \frac1{10} t_0$ such that $t_0 - \tau$ is a regular time we have
\begin{equation} \label{eq_S_geq_min}
   S_{t_0}(\tau) \geq \min \bigg\{ - C_5(A), \inf_{\substack{\tau' \in [\frac1{10} t_0, \frac12 t_0] \\ \text{$t_0- \tau'$ regular}}} S_{t_0}(\tau') \bigg\}. 
\end{equation}
To bound the last term, we use the following claim.

\begin{Claim} 
Suppose that $0 < t^* \leq t_0 \leq t_1 \leq T$, where $t^*$ is a regular time.
Then
\begin{equation} \label{eq_St1_geq_tstar}
   S_{t_1}(t_1 - t^*) \geq - C_6(A, t_0 - t^*, t_1 - t^*, \diam \MM_0) \big( S_{t_0}(t_0 - t^*) - 1 \big) 
\end{equation}
where the dependence of the constant $C_6$ on the indicated quantities is continuous.
\end{Claim}

\begin{proof}
Since the Gaussian area terms in \eqref{eq_S_is_inf} are bounded by $A$, it suffices to show that for any $x_1 \in \IR^3$
\[ \mathfrak{S}_{(x_1, t_1), \sqrt{t_1 - t^*}}(t_1 - t^*) \geq  C_6(A, t_0 - t^*, t_1 - t^*, \diam \MM_0)  \Big(  \inf_{x_0 \in \IR^3} \mathfrak{S}_{(x_0,t_0), \sqrt{t_0 - t^*}}(t_0 - t^*) - 1 \Big). \] 
This bound follows from a basic covering argument.
\end{proof}

Using short-time control, we can find a constant $\theta (\MM_0) > 0$, which depends continuously on $\MM_0$ such that if $0 < \tau < t_0 < \theta$, then $S_{t_0}(\tau)$ can directly be bounded from below by a constant of the form $-C_7(\MM_0)$, which depends continuously on $\MM_0$.
So the theorem follows by repeated application of \eqref{eq_S_geq_min} and \eqref{eq_St1_geq_tstar}.
\end{proof}
\bigskip

\section{Proofs of the main theorems} \label{sec_remaining_thms}
Here we carry out the proofs of all theorems and corollaries from Subsection~\ref{subsec_mainresults}, except for Theorem~\ref{Thm_key_thm}, which was established in Subsection~\ref{subsec_proof_key}.

\begin{proof}[Proof of Theorem~\ref{Thm_main_compactness}.]
Let us first prove
Assertion~\ref{Thm_main_compactness_a}.
By short-time estimates, it is enough to derive a uniform \emph{local} $L^2$-bound on $r_{\loc}^{-1}$ on $\IR^3 \times \Int I$.
We will obtain this bound from Theorem~\ref{Thm_L2_bound}.
Choose $A < \infty$ such that $\Theta(\MM^i) \leq A$ and $\genus (\MM^i) \leq A$ for all $i$.
Recall that Theorem~\ref{Thm_L2_bound} provides an integral bound on $r_{\loc}^{-2}$ over $\MM^i$ restricted to parabolic neighborhoods of the form $\MM^i \cap B(x_0, A \sqrt{\tau_1}) \times [t_0 - \tau_1, t_0 -\tau_2]$, excluding the subsets $\MM^{i, (\eps, A)}_{\reg}$ and $\XX^{i, t_0, \delta}$.
Here $\eps > 0$ can be chosen arbitrarily and $\delta = \delta (A,\eps) > 0$.
Since the subset $\XX^{i,t_0, \delta}$ consists of points $(x,t)$ where $r_{\loc}(x,t) \geq \delta \sqrt{t_0 - t}$, we have a uniform upper bound on $r_{\loc}^{-2}$ there.

It thus remains to argue that there is a choice of $\eps > 0$, which is independent of $i$, such that $\MM^{i, (\eps, A)}_{\reg} = \emptyset$ for all $i$.
Suppose this was not the case and fix a sequence $\eps_i \to 0$ and a sequence of points $(x_i, t_i) \in \MM^{i, (\eps_i, A)}_{\reg}$ and set $r_i := r_{\loc}(x_i,t_i)$.
Then Definition~\ref{Def_CNT_new} implies that
\[ \frac{\sup I_i - t_i}{r_i^2} \xrightarrow[i\to\infty]{} \infty \]
and the intrinsic exponential maps of the rescalings $r_i^{-1} (\MM_{t_i} - x_i)$ based at the origin smoothly converge to a minimal surface $\Sigma \subset \IR^3$.
If $\Sigma$ is an affine plane, then there is a sequence $a_i \to \infty$ such that
\[ \sup_{B(x_i, a_i r_i)} \frac{r_{\loc}(\cdot, t_i)}{r_i} \xrightarrow[i\to\infty]{} 1. \]
Using this fact and Lemma~\ref{Lem_ends_min_surf} if $\Sigma$ is not an affine plane, allows us to conclude that given any constant $H < \infty$, we can find a constant $Q(H) < \infty$ such that for large $i$
\[ \int_{\MM_{\reg,t_i}} \bigg( \log \Big( \frac{  r_{\loc}(\cdot, t_i) }{ Q r_i }  \Big) \bigg)_-\rho_{(x_i, t_i + Q^2 r_i^2)}(\cdot, t_i) \, d\HH^2 \leq - H. \]
However, this contradicts Theorem~\ref{Thm_key_thm} for large $H$ and $i$.
This concludes the proof of Assertion~\ref{Thm_main_compactness_a}.

Let us now prove Assertion~\ref{Thm_main_compactness_b}.
We first pass to a subsequence such that we have weak convergence of the Brakke flows $(\mu^i_t)_{t \in [0,T_i)}$ associated with each $\MM^i$ to a unit-regular Brakke flow $(\mu^\infty_t)_{t \in [0,T_\infty)}$.
Call its support $\MM^\infty$ and note that this is a weak set flow.
Recall that the definition of the local scale function $r_{\loc}^{\MM^i}(x,t)$ of $\MM^i$ makes sense on all $(x,t) \in \IR^3 \times I_i$.
We obtain that:

\begin{Claim} \label{Cl_int_limit}
For any bounded open subset $U \subset \IR^3$ the following is true:
\begin{enumerate}[label=(\alph*)]
\item  \label{Cl_int_limit_a} For any $(x,t) \in \IR^3 \times I_\infty$ we have $r_{\loc}^{\MM^\infty}(x,t) \geq \limsup_{i\to \infty} r_{\loc}^{\MM^i}(x,t)$.
\item \label{Cl_int_limit_b} For any $t \in I_\infty$ we have
\[ \int_{U} \big(r_{\loc}^{\MM^\infty} \big)^{-2}(\cdot,t)  \, d\mu^\infty_t \leq \liminf_{i \to \infty} \int_{U} \big(r_{\loc}^{\MM^i}\big)^{-2}(\cdot,t) \, d\mu^i_t. \]
\item \label{Cl_int_limit_cc} For almost all $t \in I_\infty$ the following is true: $t$ is a regular time for all $\MM^i$ and we have
\[ \liminf_{i \to \infty}  \int_{U} \big(r_{\loc}^{\MM^i}\big)^{-2}(\cdot,t) \, d\mu^i_t < \infty. \]
\item \label{Cl_int_limit_c} For any $[T_1, T_2] \subset I_\infty$ we have
\[ \int_{T_1}^{T_2} \int_{U} \big(r_{\loc}^{\MM^\infty} \big)^{-2}   d\mu^\infty_t dt \leq C(T_1,T_2,U) < \infty. \]
\end{enumerate}
\end{Claim}

\begin{proof}
Assertion~\ref{Cl_int_limit_a} follows via local derivative estimates.
For Assertion~\ref{Cl_int_limit_b}, recall that $r_{\loc}(\cdot, t)$ is $1$-Lipschitz.
So we may pass to a subsequence and assume that we have uniform convergence $r^{\MM^i}_{\loc}(\cdot, t) \to r'$ for some $1$-Lipschitz function $r' :U \to [0,\infty]$.
By Assertion~\ref{Cl_int_limit_a} we have $r_{\loc}^{\MM^\infty}(\cdot,t) \geq r'$ on $U$.
Thus, for any $\delta > 0$ we have
\begin{multline*}
 \int_{U} \big(r_{\loc}^{\MM^\infty} (\cdot, t) + \delta  \big)^{-2} d\mu^\infty_t \leq \int_{U} (r' + \delta)^{-2} d\mu^\infty_t = \lim_{i \to \infty} \int_{U} (r' + \delta)^{-2} d\mu^i_t \\
 = \lim_{i \to \infty} \int_{U} \big(r_{\loc}^{\MM^i} (\cdot, t) + \delta \big)^{-2} d\mu^i_t
 \leq \liminf_{i \to \infty} \int_{U} \big(r_{\loc}^{\MM^i}   \big)^{-2} (\cdot, t) d\mu^i_t.
\end{multline*}
Assertion~\ref{Cl_int_limit_b} follows by letting $\delta \to 0$.
Assertions~\ref{Cl_int_limit_cc}, \ref{Cl_int_limit_c} follow from Assertion~\ref{Cl_int_limit_b} of the claim and Assertion~\ref{Thm_main_compactness_a} of this theorem via Fatou's Lemma.
\end{proof}

\begin{Claim} \label{Cl_components}
For almost all $t \in I_\infty$ and any bounded open subsets $U \Subset U' \subset \IR^3$ there is a constant $C'(t,U, U') > 0$ such that, after passing to a subsequence, the following is true:
\begin{enumerate}[label=(\alph*)]
\item \label{Cl_components_a} The number of components of $\MM^i_t \cap U'$ intersecting $U$ is bounded by $C'(t,U, U')$
\item \label{Cl_components_b} For any component $N \subset \MM^i_t \cap U'$ that intersects $U$ we have $C'(t,U, U') r_{\loc}^{\MM^i}(\cdot,t) \geq  \diam (N \cap U')$ on $N \cap U$.
\end{enumerate}
\end{Claim}

\begin{proof}
Fix an index $i$, a regular time $t$ and a component $N \subset \MM^i_t \cap U'$ that intersects $U$.
Let $x_0 \in N \cap U$ be the point where $r := r_{\loc}^{\MM^i}(\cdot, t)$ attains its minimum and choose points $x_j \in N$ with $|x_0 - x_j| = 2^j r$, whenever $2^j r < \frac12 \diam N$.
Since $r_{\loc}^{\MM^i}(\cdot, t)$ is $1$-Lipschitz, we obtain that $r_{\loc}^{\MM^i}(x_j, t) \leq (2^j+1) r \leq 2^{j+1}r$, so there is a universal constant $c_0 > 0$ such that
\[ \int_{B(x_j, 2^{j-2}r) \cap N} \big( r_{\loc}^{\MM^i} \big)^{-2}(\cdot, t) d\HH^2 \geq c_0 > 0. \]
Choose a bounded open subset $U'' \subset \IR^3$ with the property that for any $y \in U$ and $s > 0$ with $B(y,s) \subset U'$ we even have $B(y, 10s) \subset U''$.
Then whenever $x_j$ exists, we must have $B(x_j, 2^{j-2}r) \subset U''$.
Since the balls $B(x_j, 2^{j-2}r)$ are pairwise disjoint, we can use Claim~\ref{Cl_int_limit}\ref{Cl_int_limit_cc} applied to $U''$ to derive an upper bound on the number of $j$ for which both $x_j$ exists, so for which $2^j r < \frac12 \diam (N \cap U')$.
This proves Assertion~\ref{Cl_components_b}.
For Assertion~\ref{Cl_components_a} we argue similarly that $\int_N (r_{\loc}^{\MM^i})^{-2}(\cdot, t) d\HH^2 \geq c_0 > 0$ for each component $N \subset \MM^i_t \cap U'$ that intersects $U$ and has the property that $\inf_{N \cap U} r_{\loc}^{\MM^i}(\cdot, t) < 1$.
So the number of these components is bounded by a uniform constant.
The number of components $N \subset \MM^i_t \cap U'$ that intersect $U$ and that have the property that $\inf_{N \cap U} r_{\loc}^{\MM^i}(\cdot, t) \geq 1$ is bounded by a simple ball packing argument.
\end{proof}

Applying Claim~\ref{Cl_components} to each component of $\MM^i_t$ implies that for almost all $t \in I_\infty$ the following is true for a subsequence.
There is a decomposition $\MM^i_t = M'_i \,\dotcup\, M''_i$, where each part is a union of connected components of $\MM^i_t$, such that $M'_i$ smoothly converges to a smooth submanifold $M'_\infty \subset \IR^3$ and such that $M''_i$ converges in the Hausdorff sense to a discrete set of points $M''_\infty \subset \IR^3$, which is disjoint from $M'_\infty$.
So for any bounded open subset $U \subset \IR^3$ and any $t' > t$ close enough to $t$, we can use the weak set flow property of $\MM^i$ to show that the components of $M''_i$ that intersect $U$ become extinct for large enough $i$, while the evolution of $M'_i \cap U$ remains controlled.
Thus for any such $t'$ we have local \emph{smooth} convergence of $\MM^i_{t'} \cap U$ to some smooth submanifold $M'_{\infty,t', U} \subset \IR^3$.
Since we have weak convergence $\mu^i_{t'} = \HH^2 \lfloor \MM^i_{t'} \to \mu^\infty_{t'}$, we obtain that $M'_{\infty,t', U} = \MM^\infty_{t'} \cap U$.
To summarize, we have shown:

\begin{Claim}
For any bounded open subset $U \subset \IR^3$ and almost all $t \in I_\infty$ there is a $\delta(t, U) > 0$ such that, after passing to a subsequence, we have smooth convergence $\MM^i_{t'} \cap U \to \MM^\infty_{t'} \cap U$ for all $t' \in (t,t+\delta(t))$.
\end{Claim}

Fix $U \subset \IR^3$ for a moment.
Let $S_j \subset I_\infty$ be the set of times $t$ for which the previous discussion applies and for which $\delta(t) > \frac1j$.
Then $S_1 \subset S_2 \subset \ldots$ and $I^* \cap \bigcup_j S_j$ has full measure for any bounded subinterval $I^* \subset I_\infty$.
It follows that the measure of $I^* \cap (S_j + \frac1j) $ converges to the measure of $I^*$ and for all $t' \in S_j + \frac1j$ we have smooth convergence $\MM^i_{t'} \cap U \to \MM^\infty_{t'} \cap U$ for a subsequence.

So for any bounded open $U \subset \IR^3$ almost every time $t'$ we have smooth convergence $\MM^i_{t'} \cap U \to \MM^\infty_{t'} \cap U$ for a subsequence.
So almost every time is a regular time for $\MM^\infty$.
Combined with Assertion~\ref{Thm_main_compactness_a}, this implies that $\MM^\infty$ is almost regular.
Lastly note that if $t$ is a regular time of $\MM^\infty$, then due to the weak convergence of the Brakke flows and their unit-regularity, we even have local smooth convergence $\MM^i_{t} \to \MM^\infty_{t}$ without having to pass to a subsequence.
\end{proof}
\bigskip

\begin{proof}[Proof of Theorem~\ref{Thm_main}.]
This is a direct consequence of Theorem~\ref{Thm_main_compactness}.
\end{proof}
\bigskip

\begin{proof}[Proof of Corollary~\ref{Cor_generic}.]
This follows from \cite[Proposition~9.2]{chodosh2023mean}.
Using Theorem~\ref{Thm_main} we can exclude the second case in this proposition.
\end{proof}
\bigskip

\begin{proof}[Proof of Corollary~\ref{Cor_no_min_surface_blowup}.]
Any such minimal surface $\Sigma \subset \IR^3$ would have to have $\Theta(\Sigma) < \infty$ and finite genus and would therefore have finite total curvature.
By Lemma~\ref{Lem_ends_min_surf} its blow-down would be a higher multiplicity plane.
Such a higher multiplicity plane would also be a blow-up (in the Brakke sense) of the original flow, which is impossible by Theorem~\ref{Thm_main}.
\end{proof}
\bigskip

\begin{proof}[Proof of Theorem~\ref{Thm_Msing_dim}.]
Assume the almost regular flow $\MM$ is defined on the interval $I$.  

\medskip
\ref{Thm_Msing_dim_a}  \quad
This is a consequence of \cite[Theorem~1.14]{Cheeger_Haslhofer_Naber_13}. 
By unit regularity, the singular set $\SS := \MM_{\sing} \subset \IR^3\times (\bar I\setminus\inf I)$ is the set of points $(x,t)\in \IR^3\times \IR$ with Gaussian density $>1$.   
Using the terminology of \cite{Cheeger_Haslhofer_Naber_13}, it therefore  suffices to show that for some $\eta>0$ we have $\SS\subset \SS^1_{\eta,r}$ for all $r>0$ sufficiently small, where $\SS^j_{\eta,r}$ denotes the quantitative $j$-stratum. 
Suppose this were false.  
Then there would exist  sequences $(x_j,t_j)\in\SS$, $\eta_j\rightarrow 0$, $r_j\rightarrow 0$ such that $(x_j,t_j)\not\in\SS^1_{\eta_j,r_j}$.
Unwinding definitions, this implies that after passing to a subsequence,  for some sequence $s_j\in [r_j,1]$, $s_j \to 0$, the sequence $s_j^{-1}(\MM-(x_j,t_j))$ converges to a backward selfsimilar Brakke flow $(\mu^\infty_t)_{t < 0}$, which is either a static cone or invariant under a $2$-dimensional group of spatial translations.
By  assumption $\Theta (\MM ,(x_j,t_j))>1$, so $(\mu^\infty_t)_{t < 0}$ is nontrivial.   However,   by Corollary~\ref{Cor_no_min_surface_blowup} and Theorem~\ref{Thm_main} it follows that $(\mu^\infty_t)_{t < 0}$ is a multiplicity $1$ static plane, so by Brakke regularity $(x_j, t_j)$ is backward regular for large $j$, contradicting the fact that its Gaussian density is larger than $1$.

\medskip
\ref{Thm_Msing_dim_b}  \quad  By \cite[Corollary 1.2]{bernstein_wang_topological_property} there is $\delta_0>0$ 
such that  if $(x,t)$ is a nongeneric singular point, then $\Theta (\MM, (x,t))\geq \la_1+\delta_0$, where $\la_1$ is the entropy of the cylinder. 
Hence the set of non-generic singularities is closed by upper semicontinuity of Gaussian density.

To see that the nongeneric singularities have Minkowski dimension $0$, we again use \cite[Theorem~1.14]{Cheeger_Haslhofer_Naber_13} as in Assertion~\ref{Thm_Msing_dim_a}.  
It suffices to show that for some $\eta>0$ the nongeneric singular set is contained in $\SS^0_{\eta,r}$ for all $r>0$ sufficiently small.  
Arguing by contradiction as in Assertion~\ref{Thm_Msing_dim_a}, we find a sequence $(x_j,t_j)\in \IR^3\times(\bar I\setminus\inf I)$ of nongeneric singular points and sequences  $\eta_j\rightarrow 0$, $s_j\rightarrow 0$ such that the sequence $s_j^{-1}(\MM -  (x_j,t_j)) \to (\mu^\infty_t)_{t < 0}$, which  is backward selfsimilar and  either a static cone, or the backward part is a shrinker which is invariant under a $1$-dimensional group of spatial translations.    
By Corollary~\ref{Cor_no_min_surface_blowup}, Theorem~\ref{Thm_main} and the classification of shrinkers in $\IR^2$,  it follows that $(\mu^\infty_t)_{t < 0}$ is either a multiplicity $1$ static plane or shrinking cylinder, contradicting the fact that $\Theta((\mu^\infty_t)_{t < 0})\geq \la_1+\delta_0$.  

Since the nongeneric singular set is contained in the $0^{th}$ singular stratum, it is countable \cite{white_stratification,Cheeger_Haslhofer_Naber_13}; this also follows from backward isolation established in Assertion~\ref{Thm_Msing_dim_b}.

\medskip
\ref{Thm_Msing_dim_c}\quad  Suppose $(x_j,t_j)\rightarrow (x_\infty,t_\infty)$ where $(x_j,t_j)\neq(x_\infty,t_\infty)$ is nongeneric for $j\in \IN\cup\{\infty\}$, and $t_j\leq t_\infty$.   

{\em Case 1:  
 $\liminf_{j \to \infty} \frac{|x_j-x_\infty|}{ \sqrt{t_\infty-t_j}}<\infty$.}   \quad
Let $r_j:=\sqrt{t_\infty-t_j}$.  After passing to a subsequence,  by Theorem~\ref{Thm_main} the blow up sequence $r_j^{-1}(\MM-(x_\infty,t_\infty))$ converges smoothly on compact subsets of  $\IR^3\times (-\infty,0)$ to a smooth (multiplicity one) shrinker, and $r_j^{-1}((x_j,t_j)-(x_\infty,t_\infty))\rightarrow (\bar x_\infty,-1)$ to a point on this shrinker.  This contradicts the fact that $(x_j,t_j)$ is a singular point.

{\em Case 2:  $\frac{\sqrt{t_\infty-t_j}}{|x_j-x_\infty|}\rightarrow 0$.}  \quad 
Letting $r_j:=|x_j-x_\infty|$, after passing to a subsequence, by Theorem~\ref{Thm_main}
the blow up sequence $r_j^{-1}(\MM-(x_\infty,t_\infty))$ converges smoothly on compact subsets of  $\IR^3\times (-\infty,0)$ to a smooth (multiplicity one) shrinker $\NN$, and $r_j^{-1}((x_j,t_j)-(x_\infty,t_\infty))\rightarrow (\bar x_\infty,0)\in S^2\times\{0\}$.   
We have 
\begin{equation}
\label{eqn_theta_n_1}
\Theta^{\NN}_{ (\bar x_\infty,0)}(1) \geq \la_1+\delta_0
\end{equation} 
where $\delta_0$ is as in Assertion~\ref{Thm_Msing_dim_b}.
By the same reasoning, any tangent flow $\NN'$ of $\NN$ at $(\bar x_\infty,0)$ is a smooth (multiplicity one) shrinker, which in addition must split off a line.
So $\NN'$ is a plane, sphere or cylinder contradicting \eqref{eqn_theta_n_1}.   Thus non-generic singularities are backward isolated.



Now suppose that $\MM$ is an outermost flow; the case of innermost flows is similar.    Using \cite{chodosh2023mean}  we show that non-generic singularities imply a loss of genus locally;  the argument here is a minor variation on \cite[Section 9]{chodosh2023mean} except that  our task is  substantially simpler because we are able to use $1$-sided approximation by generic flows and the convergence result Theorem~\ref{Thm_main}.

Let $\MM^j$ be a sequence of generic flows such that $\MM^j\rightarrow \MM$ as $j\rightarrow \infty$, and $\MM^j_0\rightarrow \MM_0$ smoothly from the outside as in Corollary~\ref{Cor_generic}.  
Let $(x_0,t_0)\in\MM$ be a non-generic singularity, and $r>0$.  

\begin{Claim}
For $j$ sufficiently large, $\MM^j$ loses genus locally in $P(x_0,t_0,r)$.  
\end{Claim}

Here we say that a generic flow $\hat \MM$ {\bf loses genus locally in an open subset $U\subset \IR^3\times \IR$} if there is a smooth isotopy of compact domains with smooth boundary $\{X_t\}_{t\in [t_1,t_2]}$ where for all $t\in [t_1,t_2]$ we have  $X_t\subset U_t$, 
 $\partial X_t$ intersects $\hat \MM_t$ transversely at regular points, and $\genus(X_{t_2}\cap\hat\MM_{t_2})<\genus(X_{t_1}\cap\hat\MM_{t_1})$.  

\begin{proof}


Suppose the claim were false.  After passing to a subsequence we may assume $\MM^j$ does not lose genus locally in $P(x_0,t_0,r)$ for all $j$.  By $1$-sided convergence, after passing to a subsequence there is a sequence $r_j\rightarrow 0$ such that the blow-up sequences $r_j^{-1}(\MM^j-(x_0,t_0))$, $r_j^{-1}(\MM-(x_0,t_0))$ converge (as in Theorem~\ref{Thm_main}) to $\hat \MM$ and  $\NN$ respectively, where $\hat \MM$ is an ancient almost regular flow,  $\NN$ is a tangent flow for $\MM$ at $(x_0,t_0)$, and $\hat \MM\neq \NN$ lies on one side of $\NN$.   We now argue as in \cite[Proof of Proposition~9.7]{chodosh2023mean}:  $ \hat \MM$ is disjoint from $\NN$  (\cite[Proof of Lemma~9.8 or Proposition~9.7]{chodosh2023mean}),
and hence it is the unique flow (up to parabolic rescaling) asserted in \cite[Proposition~5.6]{chodosh2023mean}.  By \cite[Proposition~8.5]{chodosh2023mean}  and the fact that the convergence   $r_j^{-1}(\MM^j-(x_0,t_0))\rightarrow \hat \MM$ is smooth at almost every time, we have a contradiction.  
\end{proof}

Now suppose the outermost flow $\MM$ has nongeneric singularities $(x_1,t_1),\ldots,(x_k,t_k)$ for $k>\genus(\MM_0)$.    Taking $r>0$ small enough that the parabolic balls $\{P(x_i,t_i,r)\}_{1\leq i\leq k}$ are disjoint and applying the claim, for large $j$ and $t>\max_i t_i$ we find that $\genus(\MM^j_t)\leq \genus(\MM_0)-k<0$, which is a contradiction.

\medskip
\ref{Thm_Msing_dim_d}\quad  This follows  from \cite{colding_minicozzi_singular_set_generic}, since the methods used there and in the previous paper \cite{colding_minicozzi_uniqueness_blowups} are local in spacetime.
\end{proof}

\bigskip

\begin{proof}[Proof of Theorem~\ref{Thm_asymptotic_structure_tangent_flow}]
The argument here follows along the lines indicated by Ilmanen  \cite[p.37]{Ilmanen_lectures_mcf_related_equations} as well as Wang's proof  \cite{Wang_Lu_asymptotic_2016}; see also \cite{song_maximum_principle_self_shrinkers}.

Let $\CC$ be the set of points $x\in   \IR^3\setminus \{0\}$ such that $\MM$ has positive Gaussian density at $(x,0)\in\IR^3\times\{0\}$; note that
we even have  $\Theta(\MM, (x,0)) \geq 1$ at these points.   The set $\CC$ is a cone, i.e., it is invariant under scaling.   If $x\in \CC$, since $x \neq \mathbf{0}$, then any tangent flow $\hat \MM$ of $\MM$ at $(x,0)$ is a smooth multiplicity $1$ shrinker which is  invariant under spatial translation in the direction $x$, and is therefore either a plane or cylinder with axis $\IR x$.  

\begin{Claim} \label{Cl_xj_cyl_plane}
Suppose that $x_j \in \IR^3$, $x_j \to \infty$ and that $\frac{x_j}{|x_j|} \to v$.
Choose points $x'_j \in \CC$ such that $d_j := |x'_j - x_j|$ is minimal.
Then one of the following is true:
\begin{enumerate}[label=(\alph*)]
\item $d_j \to \infty$ and $\MM - (x_j, 0)$ smoothly converges to an empty mean curvature flow.
\item $d_j$ remains bounded and $\MM - (x'_j,0)$ smoothly  converges the tangent flow of $\MM$ at $(v,0)$.
\end{enumerate}
\end{Claim}


\begin{proof}
For each $j$ pick $a_j \to \infty$ such that $|x''_j - x_j|$ is minimal for $x''_j := a_j v$; then $x''_j  - x_j$ is perpendicular to $v$.
Set $r_j := \max \{ 1, |x''_j - x_j| \}$ and note that $\frac{r_j}{a_j} \to 0$.

Using the standard fact that any blow-up sequence of a tangent flow is a blow-up sequence of the original flow, we may apply Theorem~\ref{Thm_main}, and pass to a subsequence such that $\MM - (x_j, 0)$ and $r_j^{-1}(\MM - (x''_j,0))$ converge to almost regular, ancient flows $\NN, \NN''$.
Since the points $x''_j$ are multiples of $v$, we obtain by scaling invariance of $\MM$ that $\NN''$ is a smooth shrinker, which is the tangent flow of $\MM$ at $(v,0)$.
Moreover, $\NN''$ is invariant under spatial translations in the direction $v$, so it is either empty, a cylinder or a plane.
By monotonicity of the Gaussian area we have
\begin{equation} \label{eq_ThetaNNp}
\Theta(\NN) \leq \Theta(\NN'') \qquad \text{if \qquad $r_j \to \infty$}.
\end{equation}
After passing to a subsequence, we may also assume that $r_j^{-1}(x_j-x''_j) \to x_\infty$, where $|x_\infty| \leq 1$ and $x_\infty$ is perpendicular to $v$.

If $\NN = \emptyset$, then we must have $d_j \to \infty$, because $\Theta^\MM_{(x_j, 0)}(1) \to 0$, but $\Theta^\MM_{(x''_j, 0)}(1) \geq 1$. 

Suppose now that $\NN \neq \emptyset$.
So we also have $\NN'' \neq \emptyset$ by \eqref{eq_ThetaNNp}, which implies that $\NN''$ is either a cylinder or an affine plane, parallel to $\IR v$.

Consider first the case in which $\NN''$ is a cylinder.
Since $\NN \neq \emptyset$, the point $(x_\infty, 0)$ must have positive density, so it must be a multiple of $v$.
However, since $x_\infty$ is also perpendicular to $v$, this implies that $x_\infty = \mathbf{0}$.
This can only happen if $r_j$ and thus also $d_j$ are uniformly bounded.

Next consider the case in which $\NN''$ is an affine plane containing the line $\IR v$.
In this case, the smooth convergence of $r_j^{-1}(\MM - (x''_j,0))$ to $\NN''$ can be extended up until time $0$. 
So $\NN$ must be either empty or an affine plane depending on whether $d_j \to \infty$ or not.
\end{proof}

Let $\Sigma' \subset \CC$ be the set of points where $\CC$ is locally a smooth surface and define $\Sigma''$ to be the union of one-ended cylinders of radius $2$ around the rays $\IR_+ v$ for $v \in \CC \setminus \Sigma'$.
Set $\Sigma := \Sigma' \cup \Sigma''$.
The Claim implies that there are compact subsets $K_1, K_2 \subset \IR^3$ such that $\Sigma \setminus K_1$ is a smooth surface and such that $\MM_{-1} \setminus K_2$ is a graph over $\Sigma \setminus K_2$ of a function, which decays at infinity, along with all its derivatives.
This implies that $\MM_{-1}$ has finite topology.
The characterization of the Theorem now follows from  \cite[Theorem 1.1]{Wang_Lu_asymptotic_2016} or directly from this description, using the invariance of $\MM$ under parabolic rescaling.
\end{proof}

\bigskip\bigskip
For the next proofs we need the following lemma, which characterizes mean curvature flows up to the first time at which a tangent flow is not of generic type, i.e., not a sphere or a cylinder.

\begin{Lemma} \label{Lem_nonfat_prop}
Consider a compact, smoothly embedded surface $\Sigma \subset \IR^3$ and let $\MM, \MM^-, \MM^+$ be the level set flow, innermost flow and outermost flow, respectively, starting from $\Sigma$ (see Subsection~\ref{subsec_terminology} for more details).
Then there is a time $T_{\gen} \leq \infty$ such that:
\begin{enumerate}[label=(\alph*)]
\item \label{Lem_nonfat_prop_a} $\MM_{[0,T_{\gen}]}=\MM^-_{[0,T_{\gen}]}=\MM^+_{[0,T_{\gen}]}$.
\item \label{Lem_nonfat_prop_b} $\MM_{[0,T_{\gen})}$ is the support of a unit-regular, cyclic Brakke flow $(\mu_t)_{t \in [0,T_{\gen})}$.
\item \label{Lem_nonfat_prop_c} \label{property_MM_3}
Any blow-up limit at any point in $\MM_{(0,T_{\gen})}$ (except possibly at the final time time $T_{\gen}$) is a ``generic'' ancient flow, i.e., an affine plane, a round shrinking sphere, a round shrinking cylinder or a translating bowl soliton, each time of multiplicity one.
\item \label{Lem_nonfat_prop_d} There is a point $(x,t) \in \MM_{T_{\gen}}$ with the property that any tangent flow at $(x,t)$ is a shrinker that is not one of the following shrinkers: an affine plane, a sphere or a cylinder, each time with multiplicity one.
\item \label{Lem_nonfat_prop_e} $\MM_{[0, T_{\gen})}$ is almost regular and $\genus (\MM_{[0, T_{\gen})}) \leq \genus (\Sigma)$.  
\end{enumerate}
\end{Lemma}

\begin{proof}
Assertions~\ref{Lem_nonfat_prop_a}--\ref{Lem_nonfat_prop_d} are paraphrasings of the results of \cite{Hershkovits_White_nonfattening, Choi_Haslhofer_Hershkovits_White_22,Choi_Haslhofer_Hershkovits_2022}.

To see Assertions~\ref{Lem_nonfat_prop_e}, note that almost every time is regular by \cite[Proposition~8.2]{Choi_Haslhofer_Hershkovits_2022}.
Let now $[T_1, T_2] \subset [0, T_{\gen})$.
We claim that there is a constant $C_{[T_1,T_2]} < \infty$ such that for any $(x,t) \in \MM_{\reg,[T_1,T_2]}$ we have
\begin{equation} \label{eq_rloc_H}
 r_{\loc}^{-2}(x,t) \leq C_{[T_1,T_2]} \big( |\mathbf{H}|^2 (x,t) + 1 \big). 
\end{equation}
In fact, if this was false, then we could find a sequence $(x_i, t_i) \in \MM_{\reg,[T_1,T_2]}$ with $r_{\loc}(x_i,t_i) \to 0$ and $r_{\loc}(x_i,t_i) |\mathbf{H}|(x_i,t_i) \to 0$.
However, this would imply that any subsequential limit of the parabolic rescalings $r_{\loc}^{-1}(x_i, t_i) (\MM - (x_i,t_i))$ satisfies $\mathbf{H} (\vec 0, 0) \equiv 0$.
So by Assertion~\ref{Lem_nonfat_prop_c} above it would have to be an affine plane, which is impossible due to the fact that the rescaling factor was $r_{\loc}^{-1}(x_i,t_i)$. 

Next, observe that due to \eqref{eq_TT_minus} we have for an arbitrary point $(x_0,t_0) \in \IR^3 \times (T_2, \infty)$
\[ \int_{T_1}^{T_2} \int_{\MM_{\reg,t}} \bigg| \frac{\mathbf x^\perp}{t_0 - \mathbf{t}} + \mathbf{H} \bigg|^2 \rho_{(x_0, t_0)}  \, d\HH^2 dt < \infty, \]
which implies that
\[ \int_{T_1}^{T_2} \int_{\MM_{\reg,t}} \big| \mathbf{H} \big|^2 \rho_{(x_0, t_0)} \, d\HH^2 dt < \infty. \]
Combined with \eqref{eq_rloc_H}, this implies the bound in Definition~\ref{Def_almost_regular}\ref{Def_almost_regular_3}. 

It remains to bound the genus of every regular time-slice $\MM_t$.
To see this, note that Assertion~\ref{Lem_nonfat_prop_c}  implies a canonical neighborhood property in a neighborhood of every singular point; see also \cite[Corollary~1.18]{Choi_Haslhofer_Hershkovits_White_22}.
This property implies that the topology of time-slices can only change by a reverse (self) connected sum or the removal of components.
Both operations do not increase the genus.
\end{proof}
\bigskip

\begin{proof}[Proof of Corollary~\ref{Cor_S2}.]
Let $\MM$ be the level set flow starting from a smoothly embedded sphere $\Sigma \subset \IR^3$ and choose $T_{\gen}$ according to Lemma~\ref{Lem_nonfat_prop}.
Then $\MM_{[0,T_{\gen})}$ is almost regular and almost every time-slice must have genus zero.
By Theorem~\ref{Thm_main} all tangent flows at time $T_{\gen}$ must have multiplicity one.
So if $T_{\gen}$ was finite, then Assertion~\ref{Lem_nonfat_prop_d} would show the existence of a tangent flow which is not a plane, sphere or cylinder, but which must have genus zero.
This is impossible by \cite{Brendle_genus0_16}. 
So $T_{\gen} = \infty$, which proves the corollary.
\end{proof}
\bigskip

\begin{proof}[Proof of Theorems~\ref{Thm_nonfattening_almost_regular}, \ref{Thm_Tfat_Tdisc}.]
Let us first consider the outermost flow $\MM^+$ starting from some smooth compact, smooth embedded surface $\Sigma \subset \IR^3$ bounding a compact domain $K^+ \subset \IR^3$.
Let $\mathcal{K}^+$ be the weak set flow starting from $K^+$; so $\MM^+ = \partial \mathcal{K}^+$ within $\IR^3 \times [0,\infty)$.

Apply Corollary~\ref{Cor_generic} to find a sequence of surfaces $\Sigma_i \to \Sigma$, bounding domains $K^+ \subset K^+_i \subset \IR^3$, whose level set flows $\MM^i \subset \IR^3 \times [0, \infty)$ only incur generic singularities; so these satisfy the assertions of Lemma~\ref{Lem_nonfat_prop} with $T_{\gen} = \infty$.
After passing to a subsequence, we may assume that $K^+_1 \supset K^+_2 \supset \ldots$ and $\bigcap_i K^+_i = K$.
It is elementary to see that the level set flows $\mathcal{K}^{i,+}$ starting from $K^+_i$ also satisfy
\begin{equation} \label{eq_KK_intersection}
 \mathcal{K}^{1,+} \supset  \mathcal{K}^{2,+} \supset \ldots, \qquad \mathcal{K}^+ = \bigcap_i  \mathcal{K}^{i,+}. 
\end{equation}
Recall that each flow $\MM^i$ is the boundary of $\mathcal{K}^{i,+}$ within $\IR^3 \times [0,\infty)$.
It follows using \eqref{eq_KK_intersection} that for every point $(x,t) \in \mathcal{M}^+$ there is a sequence of points $(x_i,t_i) \in \MM^i$ with $(x_i,t_i) \to (x,t)$.

Since the flows $\MM^i$ are almost regular, we can apply Theorem~\ref{Thm_main_compactness} to show that, after possibly passing to a subsequence, we have weak convergence of the Brakke flows associated with $\MM^i$ to a Brakke flow, which is associated with an almost regular mean curvature flow $\MM^\infty$ starting from $\Sigma$.
It follows that the $\MM^i$ converge to $\MM^\infty$ in the Hausdorff sense as subsets of $\IR^3 \times [0,\infty)$.
Thus $\MM^+ \subset \MM^\infty$.
We claim that the reverse inclusion holds, too.
To see this, observe first that $\MM^\infty \subset \mathcal{K}^+$ since $\Sigma \subset K^+$.
Consider now a point $(x,t) \in \MM^\infty \subset \mathcal{K}^+$.
We need to show that $(x,t)$ is not contained in the interior of $\mathcal{K}^+$. 
Let $\eps > 0$.
By Theorem~\ref{Thm_main_compactness} we can find a $t' < t$, arbitrarily close to $t$, such that $\MM^i_{t'} \to \MM^\infty_{t'}$ smoothly.
If $t'$ is chosen sufficiently close to $t$, then by the avoidance principle, the ball $B(x,\eps)$ must intersect $\MM^i_{t'}$ for large $i$.
So $B(x,\eps) \not\subset \mathcal{K}^{i,+}_{t'}$ for large $i$, which implies that $B(x,\eps) \not\subset \mathcal{K}^+_{t'}$.
Since $\eps$ can be chosen arbitrarily small, we conclude that $(x,t)$ is not contained in the interior of $\mathcal{K}^+$, which proves the claim. 

It follows that $\MM^+ = \MM^\infty$ is almost regular.
The analogous argument, using $K^- := \ov{\IR^3 \setminus K^+}$ instead of $K^+$ and corresponding level set flow $\mathcal{K}^-$ starting from $K^-$, shows that the innermost flow is also almost regular.

To prove that a non-fattening flow is almost regular, we consider the more general setting from Theorem~\ref{Thm_Tfat_Tdisc}.
Let us now consider a level set flow $\MM \subset \IR^3 \times [0,\infty)$ starting from a smooth embedded surface $\Sigma \subset \IR^3$.
Consider the outermost and innermost flows $\MM^+ = \partial \mathcal{K}^+, \MM^- = \partial \mathcal{K}^-$ and the approximating flows $\mathcal{K}^{i,+}$ from before.
Repeating the same argument for the flow $\mathcal{K}^{-}$ yields flows $\mathcal{K}^{i,-}$ satisfying the same property as in \eqref{eq_KK_intersection}.
It follows that
\[ \mathcal{K}^+ \cap \mathcal{K}^- \subset \mathcal{K}^{i,+} \cap \mathcal{K}^{i,-} \]
By \cite[4E]{Ilmanen_1993} (see also \cite{Hershkovits_White_avoidance}), we have $\partial \mathcal{K}^{i,\pm} \cap \mathcal{K}^\pm = \emptyset$.
Combined with the inclusion property, this implies $\partial \mathcal{K}^{i,\pm} \subset \mathcal{K}^\mp \subset \Int\mathcal{K}^{i,\mp}$.
It follows that $\mathcal{K}^{i,+} \cup \mathcal{K}^{i,-} = \IR^3 \times [0,\infty)$ and 
\[ \partial \big( \mathcal{K}^{i,+} \cap \mathcal{K}^{i,-} \big)  = \partial\mathcal{K}^{i,+} \,\dotcup \, \partial\mathcal{K}^{i,-} . \]
So every connected submanifold that is disjoint from $\mathcal{K}^{i,+} \cap \mathcal{K}^{i,-}$ must be disjoint from either $\mathcal{K}^{i,+}$ or $\mathcal{K}^{i,-}$.
It follows that $\mathcal{K}^{i,+} \cap \mathcal{K}^{i,-}$ is a weak set flow.
Since these flows form a decreasing sequence of subsets and since their intersection is $\mathcal{K}^+ \cap \mathcal{K}^-$, we obtain that $\mathcal{K}^+ \cap \mathcal{K}^-$ is a weak set flow, so it must be contained in $\MM$.
On the other hand, $\MM$ is contained in both $\mathcal{K}^+$ and  $\mathcal{K}^-$ by the inclusion property.
Hence
\[ \MM = \mathcal{K}^+ \cap \mathcal{K}^-. \]
By our previous discussion, for all times $t$ that are regular for both $\MM^+$ and $\MM^-$, the boundary components of $(\mathcal{K}^{i,+} \cap \mathcal{K}^{i,-})_t$ smoothly converge to boundary components of $\mathcal{K}^+_t$ or $\mathcal{K}^-_t$.
It follows that for all such $t$ the set $\MM_t = (\mathcal{K}^+ \cap \mathcal{K}^-)_t$ has non-empty interior if and only if $\MM_t \neq \MM_t^+, \MM_t^-$.

We can now prove the last statement of Theorem~\ref{Thm_Tfat_Tdisc}.
Suppose that $I \subset [0,\infty)$ is an open interval over which the flow is non-fattening.
Then $\MM_t = \MM_t^- = \MM_t^+$ for all $t \in I$ that are regular for both $\MM^+$ and $\MM^-$.
Now consider some $(x,t) \in \MM$.
By the avoidance principle, there must be a sequence $(x_i,t_i) \to (x,t)$ such that $(x_i, t_i) \in \MM$ and such that $t_i$ is regular for both $\MM^+$ and $\MM^-$.
So $(x_i,t_i) \in \MM^\pm$, which implies $(x,t) \in \MM^\pm$.
This implies that $\MM_I \subset \MM^\pm_I$.
The reverse inclusion is clear by definition.
To see the first statement of Theorem~\ref{Thm_Tfat_Tdisc}, choose $T \leq \infty$ maximal such that $\MM$ does not fatten over $[0,T)$.
\end{proof}

\bibliography{references}	
\bibliographystyle{amsalpha}

\end{document}